\documentclass[a4paper, 11pt, reqno]{article}

\usepackage[utf8]{inputenc}
\usepackage[left=2.65cm,right=2.65cm,top=2.7cm,bottom=2.7cm]{geometry}

\usepackage{graphicx}
\usepackage[hyphens,spaces,obeyspaces]{url}

\usepackage{amsmath,amssymb,amscd,mathtools}
\usepackage{amsthm}
\usepackage{tikz}
\usepackage{tikz-3dplot}
\usepackage{tikz-cd}
\usepackage{faktor}
\usepackage{enumerate}
\usepackage{authblk}

\usepackage[colorlinks,allcolors=blue]{hyperref}   % load once, near the end
\usepackage{cleveref}                              % right after hyperref

\newtheorem{theorem}{Theorem}
\numberwithin{theorem}{section}
\newtheorem*{theorem*}{Theorem}

\newtheorem {lemma}[theorem] {Lemma}
\newtheorem {prop}[theorem]      {Proposition}
\newtheorem* {prop*}     {Proposition}

\newtheorem {corollary}[theorem]      {Corollary}

\theoremstyle{definition}
\newtheorem {definition}[theorem] {Definition}
\newtheorem {remark} [theorem]         {Remark}

\newtheorem{example}[theorem]{Example}
\newtheorem*{example*}{Example}
\newtheorem{notation}[theorem]{Notation}

\crefname{prop}{proposition}{propositions}
\crefname{Corollary}{corollary}{corollaries}
\crefname{Claim}{claim}{claims}
\crefname{property}{property}{properties}
\crefname{notation}{notation}{notations}
\crefname{Assumption}{assumption}{assumptions}
\crefname{Question}{question}{questions}
\crefname{Conjecture}{conjecture}{conjectures}
\crefname{varTheorem}{theorem}{theorems}
\crefname{RefTheorem}{theorem}{theorems}
\crefname{refproposition}{proposition}{propositions}

\AddToHook{env/lemma/begin}{\crefalias{theorem}{lemma}}
\AddToHook{env/prop/begin}{\crefalias{theorem}{prop}}
\AddToHook{env/Corollary/begin}{\crefalias{theorem}{Corollary}}
\AddToHook{env/corollary/begin}{\crefalias{theorem}{corollary}}
\AddToHook{env/definition/begin}{\crefalias{theorem}{definition}}
\AddToHook{env/remark/begin}{\crefalias{theorem}{remark}}
\AddToHook{env/property/begin}{\crefalias{theorem}{property}}
\AddToHook{env/example/begin}{\crefalias{theorem}{example}}
\AddToHook{env/notation/begin}{\crefalias{theorem}{notation}}
\AddToHook{env/Claim/begin}{\crefalias{equation}{Claim}}

\def\R{\mathbb{R}}
\def\Z{\mathbb{Z}}
\def\N{\mathbb{N}}

\def\C{\mathbb{C}}

\def\S{\mathbb{S}}

\newcommand{\cF}{\mathcal{F}}
\newcommand{\cK}{\mathcal{K}}

\newcommand{\cL}{\mathcal{L}}

\newcommand{\cB}{\mathcal{B}}

\newcommand{\cM}{\mathcal{M}}

\newcommand{\cJ}{\mathcal{J}}
\newcommand{\cVF}{\mathcal{KVF}}
\newcommand{\cDK}{\mathcal{DK}}

\newcommand{\cP}{\mathcal{P}}

\newcommand{\cY}{\mathcal{Y}}

\newcommand{\omegacK}{\omega_{\cK}}

\newcommand{\cMline}{\overline{\cM}}

\newcommand{\T}{\mathbb{T}}

\DeclareMathOperator{\CS}{CS}

\DeclareMathOperator{\CpS}{CpS}
\DeclareMathOperator{\Id}{Id}

\DeclareMathOperator{\ind}{ind}
\DeclareMathOperator{\ev}{ev}
\DeclareMathOperator{\rk}{rk}
\DeclareMathOperator{\image}{Im}

\DeclareMathOperator{\vol}{vol}
\DeclareMathOperator{\supp}{supp}

\renewcommand{\d}{\mathrm{d}}

\newcommand{\wdt}[1]{\widetilde{#1}}

\newcommand{\barJ}{\overline{J}}
\newcommand{\baromega}{\overline{\omega}}

\title{Fillability obstructions for high-dimensional confoliations}

\author{Robert Cardona and Fabio Gironella}

\newcommand{\Addresses}{{
  \bigskip
  \footnotesize

  R.~Cardona, \textsc{Departament de Matem\`atiques i Inform\`atica, Universitat de Barcelona, Gran Via de Les Corts Catalanes 585, 08007 Barcelona, Spain; Centre de Recerca Matemàtica, Campus de Bellaterra, Edifici C, 08193, Barcelona, Spain.}\par\nopagebreak
  \textit{E-mail address}: \texttt{robert.cardona@ub.edu}

  \medskip

  F.~Gironella, \textsc{CNRS - Laboratoire de Mathématiques Jean Leray, Nantes Université, 2 Chem. de la Houssinière, 44322 Nantes, France.}\par\nopagebreak
  \textit{E-mail address}: \texttt{fabio.gironella@cnrs.fr}

}}

\date{}

\begin{document}

\maketitle

\begin{abstract}
    In this paper, we study confoliations in dimensions higher than three, mostly from the perspective of symplectic fillability.
    Our main result is that Massot--Niederkrüger--Wendl's bordered Legendrian open book, an object that obstructs the weak symplectic fillability of contact manifolds, admits a generalization for confoliations equipped with symplectic data. Applications include
    the non-fillability of the product of an overtwisted contact manifold and a class of symplectic manifolds, and the fact that Bourgeois contact structures associated with overtwisted contact manifolds admit no weak symplectic fillings for which the symplectic structure 
    restricts at the boundary to a positive generator of the second cohomology of the torus factor. In addition, along the lines of the original $3$-dimensional work of Eliashberg and Thurston, we give a new definition of approximation and deformation of confoliations by contact structures and describe some natural examples.
\end{abstract}

{
  \hypersetup{linkcolor=black}
  \tableofcontents
}

%%%%%%%%%%%%%%%%%%%%%%%%%%%%%%%%%%%%%%%%%%%%%%%%%%%%%%%%%%%%%%%%%%%%%%
%%%%%%%%%%%%%%%%%%%%%%%%%%%%%%%%%%%%%%%%%%%%%%%%%%%%%%%%%%%%%%%%%%%%%%

%%%%%%%%%%%%%%%%%%%%%%%%%%%%%%%%%%%%%%%%%%%%%%%%%%%%%%%%%%%%%%%%%%%%%%
%%%%%%%%%%%%%%%%%%%%%%%%%%%%%%%%%%%%%%%%%%%%%%%%%%%%%%%%%%%%%%%%%%%%%%

\section{Introduction}
\label{sec:intro}

\subsection{Context}
In their pioneering work \cite{ET_confoliations}, Eliashberg and Thurston proved that, besides the foliation by $2$-spheres on $S^1\times S^2$, any (codimension $1$) $C^2$-foliation on a $3$-dimensional manifold can be $C^0$-approximated, in the space of plane fields, by contact structures.
The striking nature of this result comes from the fact that foliations and contact structures have very different local behaviors. The first ones are \emph{integrable} plane fields, i.e.\ admit through every point a local surface which is everywhere tangent to the plane field. In contrast, contact structures are \emph{maximally non-integrable}, i.e.\ the maximal dimension of a local submanifold everywhere tangent to the plane field is one. 
Equivalently, (coorientable) foliations and (positive) contact structures on an ambient (oriented) $3$-manifold are plane fields defined by a one-form $\alpha$ satisfying, respectively, $\alpha\wedge \d\alpha = 0$ and $\alpha\wedge \d\alpha>0$.

In \cite{ET_confoliations}, the authors also studied the relationships between the properties of a foliation $\cF$ and those of the approximating contact structures $\xi_n$.
For instance, it is proven that if $\cF$ is \emph{taut}, i.e.\ admits a closed transverse loop through each point of the manifold, then $\xi_n$ must be, for $n$ big enough, \emph{tight}, and in fact even \emph{weakly symplectically semi-fillable}.
Here, ``tight'' means that it admits no embedded \emph{overtwisted disk}, as defined in \cite{Eli89}; ``weakly symplectically semi-fillable'' means that it lies at the boundary of a symplectic manifold $(W,\omega)$ with possibly disconnected boundary, and that $\omega\vert_{\xi_n}>0$. 
This result then says that the rigid objects in foliation theory, namely taut foliations, can only be approximated by the rigid objects in contact topology, namely tight contact structures.

In order to make the bridge between foliations and contact structures, Eliashberg--Thurston introduce in their work the notion of \emph{confoliation}: a plane field which is defined by a one-form $\alpha$ satisfying $\alpha\wedge\d\alpha\geq 0$.
What the authors prove in \cite{ET_confoliations} is, in fact, that, besides the aforementioned exceptional example of the foliation by spheres, $C^2$-confoliations can be $C^0$-approximated by contact structures.

After their introduction, there has been work by many authors aimed at a better understanding both of this approximation phenomenon \cite{Col02,Vog16,Bow16a,Bow16b,Bow16c,CKR19,ColHon22} and of the topological properties of confoliations themselves \cite{Hin97,Vog11}.
Especially related to the content of this paper is the work \cite{Hin97}, where the author uses the strategy of ``filling by holomorphic disks'', which goes back to \cite{Gro85}, to study fillings of embedded spheres in the confoliated boundary of a symplectic $4$-manifold.
The strategy relies on the approximation result \cite{ET_confoliations}, which allows leveraging the well-understood case of contact boundary and ``passing to the limit''.

\subsection{High-dimensional confoliations}
The present work is concerned with the high-dimensional picture, namely, on ambient manifolds of arbitrary dimension $2n+1$.
Foliations and contact structures in high dimensions are analogously defined as hyperplane fields defined by a one-form $\alpha$ satisfying, respectively, $\alpha\wedge\d\alpha =0 $ and $\alpha\wedge \d\alpha^n\neq 0$.
Generalizing the notion of confoliation is, however, less straightforward, as the naive condition $\alpha\wedge\d\alpha^n\geq0$ does not seem to be very useful.
Even though the rest of their work focuses on the $3$-dimensional setup, Eliashberg--Thurston propose the following generalization in \cite[Section 1.1.6]{ET_confoliations}: they call \emph{confoliation} in high odd dimensions a hyperplane field $\xi = \ker\alpha$ such that there is a complex structure $J_\xi$ on $\xi$ (i.e.\ a vector bundle endomorphism $J_\xi\colon \xi \to \xi$ such that $J_\xi^2=-\Id$) satisfying that $\d\alpha(X,J_\xi X)\geq 0$ for all $X\in \xi$. We will refer to this as $J_{\xi}$ being \emph{weakly tamed} by $\d\alpha$.
This definition is naturally motivated by a discussion about (weak) pseudo-convexity in the CR setting \cite[Section 1.1.5]{ET_confoliations}, which is a natural language to reformulate the confoliation condition in dimension $3$.
There is no further study of this high-dimensional notion in \cite{ET_confoliations} and, to the knowledge of the authors, in any other work in the literature, besides the works \cite{AltWu00,DatRuk04,DatKho12}, which unfortunately all have a rather limited applicability. 

The first work \cite{AltWu00} deals with deformations of confoliations to contact structures in any ambient odd dimensions, constructing explicit PDEs that are used to flow the confoliated forms to contact ones.
These methods, which are of course very analytical in nature, apply only under the assumption called ``conductiveness'', namely when the corank of $\d\alpha\vert_{\xi}$ is \emph{exactly} $2$, for a chosen (and hence any) defining one-form $\alpha$ for the confoliation $\xi$. One could hope for a possible analogue of this strategy to more general setups, but a constant-rank assumption seems necessary in order to find the desired PDEs.

The work \cite{DatRuk04} studies explicit necessary and sufficient conditions for the existence of \emph{linear deformations}\footnote{In \cite{ET_confoliations} a \emph{contact deformation} is a smooth family $(\alpha_t)_{t\in[0,\varepsilon)}$ with $\alpha_0$ defining the given (con)foliation and $\alpha_t$ contact for all $t>0$, and a contact deformation is \emph{linear} if it satisfies $\frac{\d}{\d t}\vert_{t=0}\alpha_t\wedge \d \alpha_t^n>0$. 
However, in \cite{DatRuk04}, a linear deformation means instead a contact deformation $(\alpha_t)_{t\in[0,\varepsilon)}$ of the form $\alpha_0 + t\beta$.} of foliations to contact structures in high dimensions, but in the case of a foliation defined by a closed one-form.
Lastly, \cite{DatKho12} considers \emph{affine deformations}, namely contact deformations $(\alpha_t)_{t\in[0,\varepsilon)}$ with $\alpha_t = c_t \alpha_0 + d_t\beta$ for some $c,d\colon [0,\varepsilon)\to \R$ with $c(0)=1$ and $d(0)=0$, restricting to the case where $\beta$ is a contact form.

\medskip

In this paper, we take the approach of equipping smooth confoliations (by which we simply mean hyperplane fields $\ker\alpha$ with $\alpha\wedge \d \alpha^n\geq 0$) with additional geometric structures, namely symplectic and complex; in the second case, our notion is a more restrictive version of the proposed definition in \cite{ET_confoliations}.  
Moreover, in ambient dimension three, our definitions are completely equivalent to the one in \cite{ET_confoliations}; i.e.\ in that ambient dimension, geometric structures always exist.

The first reason for considering such additional geometric structure is that doing so will allow us to develop the technical tools to obstruct symplectic fillability via pseudo-holomorphic curves theory (for curves that are closed or have a totally-real boundary condition). 
The second reason, which is in fact related to the first one, is that the original definition in high dimensions given in \cite{ET_confoliations} (namely, that there is an almost complex structure $J_\xi$ on $\xi$ such that $\d\alpha(v,J_\xi v)\geq 0$ for all $v\in \xi$) leads to counter-intuitive situations: even if we add the condition $\alpha\wedge \d\alpha^n\geq 0$ to the original definition, a positive confoliation can be locally the product of a negative contact structure with some other even-dimensional factor.
The motivation in this paper to add the aforementioned additional complex and symplectic structures (compatible in a way which we will detail below) comes exactly from these two reasons. We are now going to describe first the additional complex data.

\medskip

The first one is that of a \textbf{tame confoliation}. This is a hyperplane field $\xi=\ker\alpha$ on an odd dimensional manifold $M$, satisfying $\alpha\wedge (\d\alpha)^n\geq 0$, and for which there is a complex structure $J_\xi$ such that for any one-form $\alpha$ defining $\xi$ we have
\begin{equation}
\label{eqn:strictly_tamed_intro}
\d\alpha\vert_{\xi}(v,J_\xi v) \geq 0 \, ,
\text{ and equality holds if and only if }
v\in \cK_\xi, \,
\end{equation}
where $\cK_\xi$ is the \textbf{characteristic distribution} $\cK_\xi=\ker(\d\alpha\vert_{\xi})\subset \xi$ of $\xi$. 
We say in this case that $J_\xi$ is tamed by the conformal partially-symplectic class induced by $\xi$ and denoted by $\CpS_{\xi}$. 
The additional information that tameness guarantees is that vectors in $\cK_\xi$ are the \emph{only ones} for which $\d\alpha(v,J_\xi v)=0$, and in particular, this implies that $\cK_\xi$ is $J_\xi$-invariant. 
Such a complex structure captures much better the geometry of the confoliation than one for which one simply requires $\d\alpha(v,J_\xi v)\geq 0$ as proposed in \cite{ET_confoliations}, since it distinguishes more effectively between contact and integrable directions; as we will describe below, in the context of symplectic fillings this will translate into a topological control of pseudo-holomorphic curves touching the confoliated boundary tangentially.

We point out that $\cK_\xi$ is a distribution in a generalized sense, namely, it is a closed subset of $\xi$ that is the union of vector subspaces $(\cK_\xi)_p$ of $\xi_p$ for each $p\in M$.
The rank of $(\cK_\xi)_p$ is not necessarily constant, but it is at least upper semicontinuous, and always even. 
Moreover, $\cK_\xi$ is involutive, i.e.\ closed under Lie brackets, and hence will be referred to as the \textbf{characteristic foliation} of $\xi$.
Regular points of $\cK_\xi$, namely points of the underlying manifold $M$ around which the rank of $\cK_\xi$ is constant, form an open and dense subset of $M$, and it is not hard to see that $\cK_\xi$ is in fact a regular foliation over this set of regular points.

The introduction of this stronger notion in \eqref{eqn:strictly_tamed_intro} is motivated as well by the need for additional control on pseudo-holomorphic curves in symplectic fillings.
Due to the integrability of $\cK_\xi$, there might be $J_\xi$-holomorphic curves $u\colon \Sigma\to M$ whose differential has values in $\xi$. The issue with the definition given in \cite{ET_confoliations} is that the behavior of these curves is overly dependent on the choice of almost complex structure.
For instance, even if $\alpha$ is a contact structure, one could have a $J_\xi$ that maps a vector $v$ to another vector for which $\langle v, J_{\xi}(v)\rangle$ is an isotropic subspace with respect to $\d\alpha|_{\xi}$\footnote{
An explicit local example is, for instance, as follows. Consider the standard contact structure $\xi = \ker(\alpha)$ with $\alpha=\d z + x_1\d y_1 + x_2 \d y_2$ on $\R^5$, and $J_\xi$ so that at $0\in\R^5$ satisfies $(J_\xi)_0(\partial_{x_1}) = \partial_{y_2}$ and $(J_\xi)_0(\partial_{y_1}) = \partial_{x_2}$.
}. This is, of course, not a desirable property, and one indeed usually chooses in the contact case a complex structure $J_\xi$ which is tamed by $\d\alpha$.
The advantage of the notion of tameness in \eqref{eqn:strictly_tamed_intro} over that of weak-tameness in \cite{ET_confoliations} is that one gains control of the behavior of pseudo-holomorphic curves tangent to $\xi$, which now only depends on the geometric properties of the (singular) distribution $\cK_\xi$, possibly endowed with additional symplectic data. 

The way we implement the notion of almost complex structure tamed by $\CpS_{\xi}$ in relation to fillings is as follows. First, we say that a two-form $\mu$ on $M$ (possibly not of maximal rank) is \textbf{symplectically compatible}\footnote{Note that this symplectic compatibility is a well-defined property for $\xi$, i.e. does not depend on the choice of defining one-form.} with the conformal partially-symplectic class $\CpS_\xi$ of $\xi$ if $f\mu + g\d\alpha$ is positively non-degenerate on $\xi$ for all $f$ and $g$ positive functions in $M$.

Extending the notion of weak fillings of contact structures from \cite[Definition 4]{MNW}, we call a symplectic manifold $(W,\Omega)$ a \emph{symplectic filling of a confoliation}
$\xi$ on its smooth connected boundary $M=\partial W$ if $\Omega\vert_\xi$ is non-degenerate and symplectically compatible with $\CpS_\xi$, and $\Omega|_{\cK_{\xi}}$ is non-degenerate (see Definition \ref{def:symplectic_filling_confoliation}).
We also talk about \emph{symplectic filling of a tame confoliation} if moreover
there is a complex structure $J_\xi$ on $\xi$ which is at the same time tamed by $\CpS_\xi$ and tamed by $\Omega\vert_\xi$.
(Note that, technically speaking, such a $J_\xi$ implies in particular that $\Omega\vert_\xi + t\d\alpha\vert_\xi > 0$ for all $t\geq 0$.) We mention as well that the existence of a $J_{\xi}$ tamed by $\CpS_{\xi}$ implies the existence of a symplectically compatible form $\mu$.
The converse, i.e.\ obtaining a taming $J_{\xi}$ from a symplectically compatible $\mu$, is true in the contact case \cite{MNW} but is in general false in the confoliated case, because $\cK_\xi$ does not have constant rank and can a priori change from point to point in a way that prohibits the existence of a $J_\xi$ that leaves it invariant (see \Cref{sec:confoliations} for a more detailed discussion, and more specifically \Cref{exa:non-tame_confoliation}, due to an anonymous referee).

\medskip

The obstruction to fillability that we will introduce relies on pseudo-holomorphic methods. Since the behavior of pseudo-holomorphic curves depends on the symplectic data in $\cK_{\xi}$, this data has to be fixed a priori, and the symplectic form $\Omega$ of a filling is required to satisfy some compatibility with it.
In particular, our obstruction is naturally formulated in terms of a confoliation $\xi$ directly equipped with a symplectically compatible $\mu$, whose existence is guaranteed by the very definition of tame confoliation. 
The only relevant data that we need from the two-form $\mu$ is the open cone $\CS_{\xi,\mu}\coloneqq \{f \mu + g\d\alpha \; \vert f,g\colon M \to \R_{>0}\}$, which is made entirely of non-degenerate two-forms on $\xi$. 
If in addition $\mu\vert_{\cK_\xi}$ is non-degenerate, we will call the pair $(\xi, \CS_{\xi,\mu})$ a \textbf{cone almost-symplectic confoliation}, or \textbf{c-symplectic confoliation} for short.
(See \Cref{def:c-sympl_confoliation} for more details.) 
It suffices to say here that this symplectic cone approach is motivated by the fact that the fundamentally important additional data is the conformal class of the symplectic structure on the characteristic distribution $\cK_{\xi}$.
Moreover, we will call a c-symplectic confoliation $(\xi,\CS_{\xi,\mu})$ \textbf{tame} if for any element $\omega\in \CS_{\xi,\mu}$ there is a complex structure $J_\xi$ on $\xi$ tamed by $\CpS_\xi$ and by $\omega$. 
This simultaneous tameness is the one required for controlling pseudo-holomorphic curves using both the geometry of the confoliation and its symplectic properties when endowed with the additional symplectic data.  
Another property that we will need to use pseudo-holomorphic curves is that the confoliation is \textbf{geodesible}, i.e.\ it admits a defining one-form $\alpha$ and a vector field $R$ with $\alpha(R)=1$ and $\iota_R\d\alpha=0$ (cf.\ \Cref{rmk:geodesibility}).

If by ``positive confoliation'' one now means ``tame c-symplectic confoliation'', we also solved the aforementioned ``issue'' that a split confoliation given by the product of a negative contact manifold with a symplectic manifold should not be a ``positive confoliation'' (which could happen with the notion in \cite{ET_confoliations}). A note of caution regarding this notion: in general, \emph{a c-symplectic confoliation is not necessarily tame}. 
(An explicit example, due to the anonymous referee, of a c-symplectic confoliation that does not even admit an almost complex structure that simply preserves $\cK_\xi$ will be described in \Cref{exa:non-tame_confoliation}. Note that if the underlying hyperplane field of a c-symplectic confoliation is itself a contact structure, then tameness is automatic; see \cite[Theorem 2.4]{MNW}.)

As in the contact setting, different notions of fillability can then be defined for these objects,
thus speaking of a \textbf{weak or quasi-strong (semi-)filling of a (tame) c-symplectic confoliation}, or strong (semi-)filling of a strong symplectic confoliation. Here, each of these types of filling is stronger than the previous ones, and the distinction between semi-filling and filling is that the former is allowed to have a disconnected boundary; see Definition \ref{def:symplectic_filling_sympl_confoliation_different_versions}.
To give an example of one of these notions, a quasi-strong symplectic (semi-)filling is one for which $\Omega\vert_{\xi}\in \CS_{\xi,\mu}$.
We will also talk of \emph{semi-filling} of $(M,\xi,\CS_{\xi,\mu})$ when the boundary of the filling is possibly disconnected.

\subsection{Main results}

We are ready to state the main result of the paper, namely a symplectic fillability obstruction for c-symplectic confoliations.
The obstruction takes the form of a submanifold which generalizes the \emph{bordered Legendrian open book}, or \emph{bLob} in short, introduced in the contact setting in \cite{MNW} (that in turn was a generalization of the contact Plastikstufe introduced in \cite{Nie06});
we hence call it \textbf{confoliated bLob}.
The generalization from \cite{MNW} to our confoliated setup is fairly natural and takes the form of a submanifold near which the confoliation splits, in a product-like neighborhood, as the sum of a contact structure near a contact bLob and a trivial foliation by closed symplectic leaves containing a weakly exact Lagrangian,
we refer to \Cref{def:confoliated_bLob} for a precise definition.
We then have the following result:
\begin{theorem}
    \label{thm:obstruction_fillability}
    If a tame, geodesible and ``aspherical'' c-symplectic confoliation $(\xi,\CS_{\xi,\mu})$ on a closed connected oriented $(2n+1)$-dimensional manifold $M$ admits a confoliated bLob $\tilde N$, then it admits no semi-positive weak symplectic semi-filling\footnote{If the filling has more than the boundary component corresponding to $(M,\xi)$, we refer to \Cref{rmk:case_semifilling} for what is assumed in the other boundary components.} $(W,\Omega)$ such that $\tilde N$ is transversely-exact with respect to $\Omega$.
\end{theorem}

    \noindent
    In the case where $\xi$ is a contact structure, the transversely-exact assumption above just means that $\Omega$ is exact over $\tilde N$;
    in particular, \Cref{thm:obstruction_fillability} is a generalization of \cite[Theorem 4.4]{MNW}, see \Cref{def:transversely-exact_bLob} for a precise definition.  
    The condition that $(\xi,\CS_{\xi,\mu})$ is \textbf{aspherical} (see Definition \ref{def:confo_aspherical}) is of a technical nature and required to exclude certain sphere bubbling phenomena. 
    In particular, it can be dropped (together with semi-positivity of the filling) if the filling itself is assumed to be aspherical.

\begin{remark}
    \label{rmk:semipositivity}
    \emph{Semi-positivity} means that the first Chern class $c_1$ of a given symplectic manifold $(X^{2n},\Omega)$ satisfies the following condition: $c_1(A)\geq 0$ for every spherical $A\in H_2(X;\Z)$ such that $c_1(A)\geq 3-n$ and $\Omega(A)>0$.
    (Note that semi-positivity is then automatic in dimension $6$.)
    This assumption allows us to deduce that bubbling of spheres is a ``codimension $2$ phenomenon'', where the latter can be formalized using the theory of \emph{pseudo-cycles} \cite[Section 6.5]{McDSalBook}.
    
    The semi-positivity assumption can likely be weakened in \Cref{thm:obstruction_fillability} (and the following corollaries) using polyfold theory, similarly to what is done in \cite{SSZ} for the case of weak fillings of contact structures, where the authors substituted semi-positivity with much weaker assumptions on the second Stiefel-Whitney class of $\tilde N$ (or its core, to be defined later) minus its binding and boundary.
    As this goes beyond the scope of this work, we stick to the semi-positive case throughout the paper.
\end{remark}

The proof of \Cref{thm:obstruction_fillability} follows the strategy in \cite{MNW} (and in \cite{Nie06} before that for the Plastikstufe case), which can be traced back to \cite{Gro85,Eli90}.
More precisely, the geometry of the confoliated bLob allows one to find (for a well-chosen $J$ near the binding) a Bishop family of pseudo-holomorphic disks near the binding, having boundary conditions in the ``core" of the confoliated bLob, which is a totally real submanifold in the filling.
One can then propagate such a Bishop family to a moduli space of pseudo-holomorphic disks with boundary on the core of the confoliated bLob, and again the geometry of the latter allows us to reach a contradiction by studying the possible degenerations, as described by the Gromov compactness theorem, of families of disks in this moduli space.

\begin{remark}
    \label{rmk:maximum_principle_and_hypothesis_on_Jxi}
    For this last step, it is of fundamental importance to have a \textbf{(weak) maximum principle} in the confoliated setup, which we explain in Section \ref{sec:weak_plurisubharm_max_princ}. 
    This essentially follows from two ingredients.
    First, the ``weak maximum principle'' and the ``boundary point lemma'' described in \cite[Section 3.1]{NieNotes} in the setting of weakly pluri-subharmonic functions, which say, for example, that a closed pseudo-holomorphic curve touching the confoliated boundary in an interior point must be tangent to the boundary itself.
    Then, the second assumption in \eqref{eqn:strictly_tamed_intro} guarantees more control over the behavior of a curve; in this closed curve example, we would know that the curve is not only tangent to the boundary, but also to $\cK_\xi$ at each point. 
    A similar control can be obtained for curves with boundaries, such as disks.

    While tangency to the boundary can be deduced via the arguments in \cite[Section 3.1]{NieNotes} also for a complex structure $J_\xi$ on $\xi$ which is just weakly tamed, i.e.\ as in the definition given by \cite{ET_confoliations}, controlling the curve inside the boundary just by using the information on $\xi$ is impossible unless some additional compatibility of $J_\xi$ with $\xi$ is asked, such as in the tameness condition in \eqref{eqn:strictly_tamed_intro}. 
\end{remark}

\begin{remark}
    \label{rmk:geodesibility}
    It should be stressed that, independently of the choice of using a tamed or just a weakly tamed $J_\xi$, in order to have a weak maximum principle near the boundary of a filling, one needs the confoliation to be geodesible. Indeed, together with (weak) tameness, this is what allows the collar coordinate to be a tame weakly pluri-subharmonic (or just weakly pluri-subharmonic) function (see \Cref{lem:t_is_weak_plurisubharm}).
    As we will detail in \Cref{sec:confoliations} (concretely, Lemma \ref{lem:geodesible_confol_holonomy_strict_contacto} and Remark \ref{rem:geodesible2}), this poses certain constraints on the global behavior of the leaves of the characteristic foliation, namely on their holonomy.
\end{remark}

\bigskip

We now describe some applications of \Cref{thm:obstruction_fillability} above. 
The first immediate consequence concerns \textbf{fillability of split products of a contact manifold and a symplectic one}, which, even though not surprising, to the knowledge of the authors has never been explicitly stated in the literature:
\begin{corollary}
    \label{cor:split_strong_sympl_conf}
    Let $(M,\eta)$ be a $(2n+1)$-dimensional \emph{overtwisted} contact manifold, and $(X,\Omega_X)$ a symplectic $2k$-dimensional one with $k>0$; both are assumed closed and connected.
    Assume moreover that $(X,\Omega_X)$ is symplectically aspherical and admits some weakly exact Lagrangian.
    Consider on $M\times X$ the hyperplane distribution $\xi = \pi_M^*\eta$ and the two-form $\omega =\pi_X^*\Omega_X$, where $\pi_M\colon M\times X \to M$ and $\pi_X\colon M\times X \to X$ are the projections into each factor.
    Then the following hold:
    \begin{itemize}
        \item[-] the strong symplectic confoliation $(\xi,\omega)$ admits no semi-positive strong symplectic semi-filling,
        \item[-] the c-symplectic confoliation $(\xi,\CS_{\xi,\omega})$ admits no semi-positive quasi-strong symplectic semi-filling if $\dim M\geq 5$,
        \item[-] the c-symplectic confoliation $(\xi,\CS_{\xi,\omega})$ admits no semi-positive symplectic weak semi-filling $(W,\Omega)$ 
        for which a confoliated bLob of the form $\tilde N=N_c\times X$, where $N_c \subset M$ is a contact bLob, is transversely-exact with respect to $\Omega$.
    \end{itemize}
    
\end{corollary}

\Cref{cor:split_strong_sympl_conf} implies in turn the following \textbf{weak fillability obstruction for Bourgeois contact structures} \cite{Bou02} in high dimensions:
\begin{corollary}
    \label{cor:obstruction_weak_filling_bourgeois}
    Let $(M,\eta)$ be a $(2n-1)$-dimensional \emph{overtwisted} contact manifold, and $\xi$ the Bourgeois contact structure on $M\times \T^2$ associated to any choice of open book decomposition on $M$ supporting $\eta$.
    Then, $(M\times \T^2,\xi)$ constructed as in \cite{Bou02} admits no semi-positive weak filling or semi-filling $(W,\Omega)$ (in the contact sense) such that the $2$-cohomology class $[\Omega\vert_{M\times \T^2}]$ is a \emph{positive} (w.r.t.\ the orientation on $\T^2$) generator of the factor $H^2(\T^2;\R)\simeq H^0(M;\R)\otimes  H^2(\T^2;\R)\subset H^2(M\times \T^2;\R)$.
\end{corollary}
This should be compared with the known fact that if $(M,\eta)$ is strongly fillable then $(M\times \T^2,\xi)$ admits a weak symplectic filling $(W,\Omega)$ for which $[\Omega|_{M\times \T^2}]$ is a positive generator of the factor $H^2(\T^2;\R)$; see (the proofs of) \cite[Theorem A.(a)]{LMN} or \cite[Proposition 5.1]{Gir20_constructions}.
\Cref{cor:obstruction_weak_filling_bourgeois} gives a ``partial converse'', i.e.\ if $(M\times \T^2,\xi)$ has a (semi-positive) weak symplectic filling $(W,\Omega)$ for which $[\Omega|_{M\times \T^2}]$ is a positive generator of the factor $H^2(\T^2;\R)$ then $(M,\eta)$ must be tight.
Note that the requirement on the cohomology class is important: \cite[Theorems A.(a) and B]{LMN} imply that the Bourgeois contact manifold associated with an overtwisted contact structure on the sphere supported by the open book with page $T^*\S^n$ and monodromy the inverse of the Dehn--Seidel twist is weakly fillable, but following the proofs one can see that the cohomology class of the symplectic form at the boundary is a \emph{negative} generator of $H^2(\T^2;\R)$.

\medskip

A further application concerns the \textbf{Weinstein conjecture}. 
Namely, we obtain the following existence result for (in fact, contractible) closed Reeb orbits of stable Hamiltonian structures, which extends a well-known result of \cite{Hof93,AlbHof09} in the overtwisted contact case and its generalization in \cite{MNW} to the case of contact bLobs:
\begin{theorem}
    \label{thm:contractible_reeb_orbit}
    Let $(\xi,\CS_{\xi,\mu})$ be a tame c-symplectic confoliation on a closed connected oriented $(2n+1)$-dimensional manifold $M$ which is geodesible and aspherical. 
    Assume that $(\xi,\CS_{\xi,\mu})$ admits a confoliated bLob.
    Then any stable Hamiltonian structure $(\lambda,\omega)$, admitting a semi-positive weak symplectic cobordism to $(\xi,\CS_{\xi,\mu})$ for which the confoliated bLob is transversely-exact, must have a contractible closed Reeb orbit.
\end{theorem}

Here, \emph{symplectic cobordism from a stable Hamiltonian structure to a confoliation} means a symplectic manifold $(W,\Omega)$ with boundary $\partial W = M \sqcup \overline{M'}$, with $\Omega\vert_M$ satisfying the compatibility condition with $(\xi,\CS_{\xi,\mu})$ as in the definition of weak symplectic filling of c-symplectic confoliations, and $\Omega\vert_{M'}$ satisfying the compatibility condition at the negative end of a strong symplectic cobordism between stable Hamiltonian structures as defined in \cite[Definition 6.2]{CV}.
Moreover, with the stable Hamiltonian structure $(\lambda,\omega)$ being \emph{semi-positive} we simply mean that we require its symplectization $((-\varepsilon,\varepsilon)\times M', \d(t\lambda)+\omega)$ to be semi-positive (in the symplectic sense).

We point out that the reason why a stable Hamiltonian structure is needed here is because SFT-type breakings at the lower end of the symplectic cobordism need to be considered to deduce the existence of contractible closed Reeb orbits.

\medskip

To further motivate the general study of confoliations in high dimensions, we also discuss some explicit examples and give a notion of contact approximations, which is natural from the perspective of c-structures adopted in this paper. 

First, we should point out that there is already a very natural candidate notion of convergence, that is, with respect to the $C^k$-topology in the space of almost contact structures, i.e.\ in the space of hyperplane fields equipped with conformal classes of non-degenerate two-forms.
This approach is formalized, for instance, in \cite[Chapter 2]{TouThesis}, where several explicit examples and counterexamples are also presented in detail (e.g. \ sequences of contact structures converging in the space of hyperplane fields but not in that of almost contact structures).
The drawback of this definition is that, at all points where the conformal partially-symplectic class $\CpS_\xi$ of a confoliation $\xi$ is well-defined (i.e.\ where $\d\beta\vert_\xi$ does not vanish, with $\beta$ being any defining one-form for $\xi$), $C^k$-convergence with $k\geq 1$ of $\xi_n$ towards $\xi$ in the space of hyperplane fields already implies that $\CS_{\xi_n}\to \CpS_\xi$.
In particular, all information is lost about ``what conformal symplectic structure $\xi_n$ gives at the limit along $\cK_\xi$''.
We hence look at a new notion of \textbf{contact $C^k$-approximations} and \textbf{contact $C^k$-deformations}, which is exactly meant to require control on this ``limit behavior of $\xi_n$ along $\cK_\xi$''. Roughly speaking, we require that the leading term of $\d\alpha_n\vert_{\cK_\xi}$ as $n\to \infty$ defines a conformal class that is compatible with the c-symplectic structure in $\xi$.
The precise definition is slightly technical, hence we omit it in this introduction (see \Cref{def:contact_approximation_deformation}). 

We further give, in the following two propositions, several classes of natural examples in support of this new notion (i.e.\ of \Cref{def:contact_approximation_deformation}).

\begin{prop}
    \label{prop:examples_contact_deformation}
    \begin{enumerate}

    \item\label{item:examples_contact_deformation__fiber_sum_branched_cover} Contact branched covers and contact fibered sums defined in \cite{Gei97} are naturally time-$1$ structures of contact $C^\infty$-deformations of a c-symplectic confoliation on the underlying smooth branched covers and fibered sums.
    
    \item\label{item:examples_contact_deformation__bourgeois} Let $\xi$ be a contact structure on $M^{2n-1}$.
    Then, Bourgeois contact structures \cite{Bou02} on $M\times \T^2$, associated to any choice of supporting open book for $\xi$, naturally give a contact $C^\infty$-deformation of the c-symplectic confoliation $(\xi \oplus T\T^2,\CS_{\xi\oplus T\T^2,\omega})$ on $M\times \T^2$, where $\omega$ is any positive area form on $\T^2$.

    \item\label{item:examples_contact_deformation__mori}  Mori's $(0,1]$-family of contact structures \cite{Mor19} on $\S^5$ is a contact $C^\infty$-deformation of Mitsumatsu's symplectic foliation \cite{Mit18}.

    \item\label{item:examples_contact_deformation__bertelson_meigniez} The linear deformations to contact structures of locally conformal symplectic foliations considered in \cite{BerMei23}, as well as their natural generalization to the case of hyperplane fields with constant order, are contact $C^k$-deformations of the c-symplectic structure given by the leafwise locally conformal symplectic structure on the underlying smooth foliation.

    \item\label{item:examples_contact_deformation__giroux_torsion} There is a c-symplectic confoliation of the form $(\xi,\CS_{\xi,\omega})$, where $\xi=\ker(\d s)$ with $(s,t)$ coordinates on $\T^2$, such that each of the contact structures $\xi_k$ on the manifold of the form $M\times\T^2$ constructed as in \cite[Theorem A]{MNW} is the time-$1$ structure of a contact $C^\infty$-deformation $\xi_k^t$ for all $k$.
    \end{enumerate}
\end{prop}

{It has already been proved in \cite{Gir20_constructions,GirThesis} that contact branched covers and contact fibered sums can be realized respectively as time-$1$ structures of families $(\xi_t)_{t\in[0,1]}$ and $(\xi_t')_{t\in[0,1]}$ with $\xi_0$ a natural confoliation (in the sense of \cite{ET_confoliations}) on the underlying smooth branched cover and $\xi_0'$ a natural confoliation (again in the sense of \cite{ET_confoliations}) on the underlying smooth fibered sum.
\Cref{item:examples_contact_deformation__fiber_sum_branched_cover} above then simply follows from the observation that $\xi_0,\xi_0'$ are c-symplectic confoliations, and that the two families $(\xi_t)_t$ and $(\xi_t')_t$ from \cite{Gir20_constructions,GirThesis} are in fact contact $C^\infty$-deformations (in the sense introduced above). 
}

\begin{remark}
    Recall that \cite[Proposition 2.3.2]{ET_confoliations} states that there is a foliation on the $3$-torus that is deformable to infinitely many pairwise-distinct contact structures, distinguished by their Giroux torsion.
    In fact, the contact structures $\xi_k$ on $M\times \T^2$ from \cite[Theorem A]{MNW} are a natural generalization to higher dimensions of these universally tight contact structures on $\T^3$.
    Indeed, as explained in \cite{MNW}, the contact structures $\xi_k$'s are also pairwise distinct, as they have different cylindrical contact homologies.
    In particular, \Cref{item:examples_contact_deformation__giroux_torsion} above gives the first examples of a c-symplectic confoliation which is deformable to infinitely many pairwise distinct contact structures in high dimensions, thus generalizing \cite[Proposition 2.3.2]{ET_confoliations}.
\end{remark}

The second proposition in this direction also gives a new explicit construction of high dimensional tamed c-symplectic confoliations which are not everywhere contact.
\begin{prop}
    \label{prop:open_book_construction}
    Let $(B,\pi\colon M\setminus B \to S^1)$ be a smooth open book decomposition of a $(2n+1)$-dimensional oriented closed manifold $M$, and denote by $W$ the closure of a fiber $\pi^{-1}(\mathrm{pt})$, so that $\partial W=B$.
    Assume also that there is a tame c-symplectic confoliation $(\xi_B=\ker\alpha_B,\CS_{\xi_B,\omega_B})$ on $B$, as well as a symplectic structure $\Omega_W$ on $W$ such that the monodromy of $(B,\pi)$ is given by a symplectomorphism of $(W,\Omega_W)$ which is the identity near $B$, and such that $(W,\Omega_W)$ is a weak symplectic filling of $(\xi_B,\CS_{\xi_B,\omega_B})$.
    Then, there are a hyperplane field $\xi$ and a closed two-form $\Omega_M$ on $M$ such that $(\xi,\CS_{\xi,\Omega_M})$ is a tame c-symplectic confoliation, with the following properties. 
    Along $B$, we have $\xi\cap TB=\xi_B$ and $\Omega_M\vert_{B}=\Omega_W\vert_B+\delta\,\d\alpha_B$ for some constant $\delta>0$; in particular, $\Omega_M\vert_{\xi_B}\in\CS_{\xi_B,\omega_B}$.
    Away from a neighborhood of $B$, $\xi$ coincides with $\ker \d\pi$ and the restriction of $\Omega_M$ to the fibers of $\pi$ is $\Omega_W$.

    Moreover, the following can be arranged.
    \begin{enumerate}
\item If $(\alpha_B,\Omega_W\vert_B)$ is a stable Hamiltonian structure on $B$, then $\xi$ admits a defining one-form $\alpha$ such that $(\alpha,\Omega_M)$ is a stable Hamiltonian structure on $M$; in particular, $\xi$ is geodesible.
        \item\label{item:open_book_construction__deformation} In the case where the pages of the given open book are Liouville manifolds and the monodromy is an exact symplectomorphism, the c-symplectic confoliation $(\xi,\CS_{\xi,\Omega_M})$ can be chosen to be $C^\infty$-deformable to contact structures, all of which are supported by the open book in the contact sense \cite{Gir02}.
    \end{enumerate}
\end{prop}

\noindent
In order to construct finite energy pseudo-holomorphic rank $2$ foliations on the symplectization of a given contact manifold, the works \cite{ACH05,Wen10} (using contact open books in dimension $3$) and \cite{Mor19} (using special contact spinal open books in high dimensions) already considered ``deformations to contact structures'' of certain stable Hamiltonian structures. 
Namely, they constructed families $\xi_t$ that are contact for $t>0$ and $\xi_0=\ker\lambda_0$ with $(\lambda_0,\omega_0)$ the given stable Hamiltonian structure; moreover, in their case there is $J_t$ tamed by $\xi_t$ converging to a $J_0$ tamed by $\omega_0$. 
The above \Cref{item:open_book_construction__deformation} offers a similar construction for open books in any dimension, but from the perspective of contact $C^k$-deformations as defined in this paper.

Recall that, in ambient dimension $3$, Etnyre \cite{Et} proved that any contact structure is a $C^\infty$-deformation of a foliation. 
Using the existence of contact supporting open books for any contact manifold \cite{Gir02}, \Cref{item:open_book_construction__deformation} above implies in particular that any contact structure is a (small) $C^\infty$-deformation of a confoliation ``supported'' by the same open book (where the latter is to be interpreted in the sense of \Cref{prop:open_book_construction}).
However, improving this to a contact $C^\infty$-deformation starting from a symplectic foliation is, in general, a very hard task.
Indeed, already the example treated in \cite{Mit18,Mor19} (and reinterpreted according to our notion of contact deformation in \Cref{item:examples_contact_deformation__mori} in \Cref{prop:examples_contact_deformation}) shows that to construct symplectic foliations from contact open books one needs very special geometry on the binding (cf.\ \cite{GirTou24}).

\bigskip

\subsection*{Organization of the paper}
In \Cref{sec:confoliations_high_dim}, we start by introducing some notions from the theory of distributions, the definition of confoliation, and the different symplectic data that we can equip them with. 
We proceed by defining what symplectic fillability means in our context and give a construction of c-symplectic confoliations using open books. 
In \Cref{sec:restrictions_J-hol_from_strict_cotameness} we show how the maximum principles adapt from pluri-subharmonic functions to tame weakly pluri-subharmonic functions, a notion directly related to tame c-symplectic confoliations. In
\Cref{sec:confoliated_bLob} we define the main geometric object of this work, the confoliated bLob. 
We give normal forms near its singularities and study pseudo-holomorphic curves in those neighborhoods. 
\Cref{sec:proof_obstruction_fillability} contains the proof of \Cref{thm:obstruction_fillability}, \Cref{sec:applications} the proof of \Cref{cor:split_strong_sympl_conf,cor:obstruction_weak_filling_bourgeois}, and \Cref{thm:contractible_reeb_orbit}. 
The last \Cref{sec:contact_approx} is devoted to defining the new notion of contact approximation of c-symplectic confoliations and to analyzing several examples of this phenomenon.

\subsection*{Acknowledgments}
The authors are grateful to Klaus Niederkr\"uger for useful discussions concerning the strategy of filling by holomorphic disks, as well as to the anonymous referees for their various and very useful comments, for pointing out several inaccuracies in the original manuscript, and for the explicit \Cref{exa:non-tame_confoliation} of a c-symplectic confoliation that is not tame.

The first author acknowledges partial support from the AEI grant PID2023-147585NA-I00, the Departament de Recerca i Universitats de la Generalitat de Catalunya (2021 SGR 00697), and the Spanish State Research Agency, through the Severo Ochoa and María de Maeztu Program for Centers and Units of Excellence in R\&D (CEX2020-001084-M).
The second author benefits, or has benefited, from the support of: the region Pays de la Loire, via the project Étoile Montante 2023 SymFol; the ANR grant SiSyFo ``Singular Symplectic Foliations in Contact, Symplectic and Poisson topology'', bearing the reference ANR-25-ERCS-0010; of the ANR Grant COSY ``New challenges in symplectic and contact topology'', bearing the reference ANR-21-CE40-0002; and of the French government's ``Investissements d’Avenir'' program integrated to France 2030, bearing the following reference ANR-11-LABX-0020-01.

%%%%%%%%%%%%%%%%%%%%%%%%%%%%%%%%%%%%%%%%%%%%%%%%%%%%%%%%%%%%%%%%%%%%%%
%%%%%%%%%%%%%%%%%%%%%%%%%%%%%%%%%%%%%%%%%%%%%%%%%%%%%%%%%%%%%%%%%%%%%%

\section{Confoliations in high dimensions}
\label{sec:confoliations_high_dim}

In this section, we review some background notions for hyperplane distributions and introduce our definition of confoliation in high dimensions and different symplectic data that one can equip it with.

%%%%%%%%%%%%%%%%%%%%%%%%%%%%%%%%%%%%%%%%%%%%%%%%%%%%%%%%%%%%%%%%%%%%%%

\subsection{Characteristic and kernel distributions of a hyperplane field}
\label{sec:char_distr_normal_form}

We first describe the well-known normal form for a hyperplane distribution near its points of ``constant order'', using
the notion of characteristic distribution.
This part follows closely the presentation in \cite[Chapter II, Section 3]{ExtDiffSyst_book}, rephrasing things in a way that is more suited to our purposes and including the necessary proofs for completeness.
Afterward, we introduce a second distribution, which we call \emph{kernel distribution}, associated with a given hyperplane field, which is ``differentiable'' in a suitable sense and coincides with the characteristic distribution at points of constant order.
This second part follows the short presentation in \cite[Appendix 4]{LibMarBook}.

\medskip

Let $\xi$ be a cooriented\footnote{From now on, all manifolds will be implicitly assumed to be oriented, and all hyperplane fields to be cooriented. This induces, in particular, an orientation on the latter, and all additional geometric structure we will consider on them later on (symplectic forms, complex structures) will be implicitly assumed to be compatible with this induced orientation.} hyperplane distribution on $M$ and $\alpha$ be a defining one-form for $\xi$.
\begin{definition}
    \label{def:char_distr}
    The \emph{characteristic distribution} of $\xi$ is 
    \[
    \cK_\xi \coloneqq \{\,v\in \xi \; \vert \;
    (\iota_v \d\alpha)\wedge\alpha = 0 \,\} \, .
    \]
\end{definition}

\noindent
Note that the condition $(\iota_v \d\alpha)\wedge\alpha=0$ only depends on $\xi$ and not on the auxiliary choice of $\alpha$, so that the definition is well-posed.
Moreover, such a condition is just equivalent to the fact that $\iota_v\d\alpha$ is proportional to $\alpha$ at the point $p=\pi(v)$ with $\pi\colon TM\to M$ the natural projection.
We also point out that, in the case that $\xi$ is defined by a closed one-form $\alpha$, the characteristic distribution of $\xi$ is nothing else than the distribution itself.

\begin{remark}
\label{rmk:char_distr_xi_exterior_ideals}
This notion finds its motivation in the theory of exterior differential systems.
A subset $I\subset \Omega^*(M)$ is called an \emph{exterior differential ideal} if it is an ideal with respect to the wedge product and closed under the exterior derivative. 
For such an $I$, one can similarly define its \emph{characteristic distribution} 
\[
\cK_I:=\{v\in T_pM \mid \iota_v I_p \subset I_p  \} \, .
\]
A question of interest in this general setup, and that has to do with the nature of the singularities of $\cK_I$, is whether vectors in the latter extend locally to vector fields that are in $\cK_I$ at every point. 
We will not be concerned with this question here, but still introduce the distribution that is made of pointwise evaluation of such vectors below in \Cref{def:kernel_distribution} in order to point out the differences with $\cK_\xi$.
\end{remark}

The general identities $\cL_X = \d\iota_X + \iota_X \d$ and $\iota_{[X,Y]} = [\cL_X,\iota_Y]$, for any choice of vector fields $X$ and $Y$ on $M$ (and hence in particular for those in $\cK_\xi$),
immediately imply the following key properties of characteristic distributions:
\begin{lemma}
    \label{lem:char_distr_involutive_and_invariant}
    The characteristic distribution $\cK_\xi$ is involutive: namely, if $X,Y$ are vector fields tangent to $\cK_\xi$ then their Lie bracket $[X,Y]$ is also tangent to it.
    Moreover, $\xi$ is $\cK_\xi$-invariant, i.e.\ for any defining one-form $\alpha$ and any vector field $X$ tangent to $\cK_\xi$, $\cL_X\alpha = f \alpha$ for some $f\colon M \to \R$.
\end{lemma}
\begin{remark}
\label{rmk:char_distr_involutive_invariant}
    Involutivity and invariance under the flow of tangent vector fields are general properties of characteristic distributions of differential ideals of the graded algebra $\Omega^*(M)$.
\end{remark}
\begin{remark}
    \label{rmk:transverse_contact}
    Near a point $p\in M$ where $\rk\cK_\xi$ is constant, $\xi$ induces a well-defined \emph{contact structure} $\xi/\cK_\xi$ on $TM/\cK_\xi$.
\end{remark}

There is also a useful numeric invariant associated with a hyperplane field:
\begin{definition}
    \label{def:order}
    Let $\xi$ be a hyperplane distribution on $M$, and $\alpha$ a defining one-form.
    The \emph{order of $\xi$ at a point $p\in M$} is the integer $k=k(p)$ such that $\alpha\wedge (\d\alpha)^k\neq 0$ and $\alpha\wedge(\d\alpha)^{k+1}=0$ at $p$.
\end{definition}
Note that the above conditions involving $\alpha$ actually only depend on $\xi$, so the definition is indeed well-posed.

\begin{remark}
    \label{rmk:corank_char_fol}
    Basic linear algebra tells us that, if $k$ is the order of $\xi$ at $p\in M$, then the corank of $(\cK_\xi)_p$ in $T_pM$ is $2k+1$.
    In particular, if $k$ is constant near a point $p$, $\cK_\xi$ has constant rank near $p$, equal to $\dim(M)-2k-1$, and is hence a foliation by \Cref{lem:char_distr_involutive_and_invariant} and Frobenius' theorem.
\end{remark}

We are now ready to give the desired normal form near points of constant order:
\begin{prop}
    \label{prop:normal_form_xi}
    Let $\xi$ be a hyperplane field on an $m$-dimensional smooth manifold $M$ and $\cK_\xi$ its characteristic distribution. 
    Assume that the order $k$ of $\xi$ is constant near $p\in M$.
    Then, there are local coordinates $(z,x_1,y_1,\ldots, x_k,y_k, w_1 , \ldots, w_{m-2k-1})$ centered at $p$ where 
    \[
    \xi = \ker(\d z + \sum_{i=1}^k x_i \d y_i) 
    \quad \text{and} \quad
    \cK_\xi = \langle \partial_{w_{1}},\ldots,\partial_{w_{m-2k-1}} \rangle_\R \, .
    \]
\end{prop}

\noindent
We provide here just a sketch of proof, assuming the reader is familiar with Darboux's theorem for contact structures.

\begin{proof}[Sketch of proof]
    According to \Cref{rmk:corank_char_fol}, $\cK_\xi$ is a foliation of rank $m-2k-1$ near $p$.
    Near $p$ one can then find a chart $U$ diffeomorphic to $\R^{2k+1}\times \R^{m-2k-1}$, with $p$ identified with $(0,0)$, in such a way that the leaves of $\cK_\xi\vert_U$ are just $\{pt\}\times \R^{m-2k-1}$.
    
    Note now that, by \Cref{lem:char_distr_involutive_and_invariant}, $\xi$ is invariant under the flow of vector fields tangent to $\cK_\xi$. 
    In this coordinate chart, this amounts to saying that $\xi$ is defined by a one-form $\alpha$ which is a pullback $\pi^*\alpha_0$ to $\R^{2k+1}\times \R^{m-2k-1}$ of a (non-vanishing) one-form $\alpha_0$ on the first factor $\R^{2k+1}$, via the natural projection $\pi\colon \R^{2k+1}\times \R^{m-2k-1} \to \R^{2k+1}$.
    
    At this point, note that $\alpha_0$ is in fact a contact form on $\R^{2k+1}$ because $\alpha_0$ has order $k$.
    Then, a direct application of Darboux's theorem for contact forms to $\alpha_0$ gives the desired conclusion.
\end{proof}

\medskip

We point out that there is another natural distribution, in general strictly contained in $\cK_\xi$, which is associated with a hyperplane field.
As this will not be used in what follows, we just briefly recall its definition and relationship with the characteristic distribution, and this will make it a good occasion to point out some properties of the latter, which will be useful in the rest of the paper.

Given an open subset $U\subset M$, we denote by $\cVF(U)$ the vector space made of those vector fields tangent to $\cK_\xi$ which are defined over all $U$.
The association $U\mapsto \cVF(U)$ then gives a sheaf $\cVF$ on $M$, which one can see as valued in the natural category whose objects are the vector spaces of sections of $TM\vert_U$, for every $U\subset M$ open.

Moreover, by \Cref{lem:char_distr_involutive_and_invariant}, if $X,Y\in \cVF(U)$ then $[X,Y]\in \cVF(U)$ as well.
In other words, the functor $\cVF$ takes values in the category of Lie algebras of sections of $TM\vert_U$ for all open $U\subset M$.

From this functor $\cVF$ one can easily get a (singular, i.e.\ not constant-rank) distribution $\cDK_\xi\subset TM$ on $M$:
\begin{definition}
    \label{def:kernel_distribution}
    The \textbf{differentiable characteristic distribution}, or \textbf{kernel distribution},
    $\cDK_\xi$ on $M$ is the (singular) distribution made of all those $v\in TM$ such that $v=X(p)$ for some $U\subset M$ open with $p\in U$ and $X\in \cVF(U)$.
\end{definition}

The relationship between $\cK_\xi$ and $\cDK_\xi$ is simply as follows:
\begin{lemma}
    \label{lem:charact_distr_vs_kernel_distr}
    The following properties hold.
    \begin{enumerate}
    \item $\rk(\cK_\xi)\colon M \to \N$ is upper semicontinuous, while $\rk(\cDK_\xi)\colon M \to \N$ is lower semicontinuous.
    In particular, both of them are locally constant on an open and dense subset of $M$.
    
    \item If the rank $\rk(\cK_\xi)$ is locally constant near $p\in M$, then the same holds for $\rk(\cDK_\xi)$. In this case, $\cK_\xi=\cDK_\xi$ near $p$. In particular, $\cK_\xi$ and $\cDK_\xi$ coincide over a dense and open subset of $M$.
    \end{enumerate}
\end{lemma}

\begin{proof}[Sketch of proof]
    In the first point of the lemma, the statement about the continuity properties of the ranks can be explicitly checked from the definitions of the two distributions, while the second sentence is a general property of upper- and lower-semicontinuous functions valued in a discrete set.
    The second point is a direct consequence of the first point and \Cref{prop:normal_form_xi}.
\end{proof}

\begin{remark}
    \label{rmk:charact_distr_vs_kernel_distr_qualitatively}
    Singular distributions defined by sheaves of Lie algebras of local vector fields as in the case of $\cDK_\xi$ are oftentimes called \textbf{differentiable}.
    This is the reason for the notation $\cDK_\xi$, as it is the largest differentiable distribution contained in $\cK_\xi$.
    The ``differentiability'' of $\cDK_\xi$ means, more precisely, that for any $v\in (\cDK_\xi)_p$ and any sequence $p_n$ in $M$ converging to $p$ there is a sequence $v_n\in (\cDK_\xi)_{p_n}$ converging to $v$ in the topology induced by $TM$.
    The characteristic distribution $\cK_\xi$ is, on the contrary, \emph{not} differentiable, exactly because its rank is upper-semicontinuous and not lower-semicontinuous.
    However, this is the distribution that we will need for our purposes, as it detects exactly the directions in which $\xi$ is not contact.
\end{remark}

\medskip

As we will need it later on, we introduce also the following notion:
\begin{definition}
    \label{def:conformal_partially-symplectic_class}
    Let $\xi$ be a hyperplane field on $M$.
    We call \textbf{conformal partially-symplectic class} of $\xi$, and denote it $\CpS_\xi$, the equivalence class of $\d\alpha\vert_\xi$ under the conformal equivalence relation $\sim$ identifying any alternating two-form $\mu$ on $\xi$ with $f \mu$, for any $f\colon M\to \R_{>0}$.
\end{definition}
Again, note that this class only depends on $\xi$. We also point out that the name ``partially-symplectic'' is here justified by the fact that $\d\alpha\vert_\xi$ induces a non-degenerate two-form (i.e.\ a symplectic one in the vector-bundle sense) only on $\xi/\cK_\xi$, or equivalently on any (auxiliary) complementary subspace to the characteristic distribution $\cK_\xi$ in $\xi$.
Note that $\d\alpha$ is \emph{not} symplectic on $\xi/\cDK_\xi$ at those points where $\cDK_\xi$ is strictly contained in $\cK_\xi$; according to \Cref{lem:charact_distr_vs_kernel_distr}, this happens only on a subset with empty interior.

\begin{notation}
    Throughout the work, to lighten the notation a bit and be consistent with the one used for the characteristic distribution, we will use the following convention: 
    $\mathcal{F}$ will denote at the same time a regular foliation, intended as a partition in leaves, and the distribution given by the tangent spaces to said leaves.
\end{notation}

%%%%%%%%%%%%%%%%%%%%%%%%%%%%%%%%%%%%%%%%%%%%%%%%%%%%%%%%%%%%%%%%%%%%%%

\subsection{Confoliations in high dimensions}
\label{sec:confoliations}

As explained in the introduction, we propose a new take on confoliations in high dimensions based on auxiliary geometric structures on them. 
First, let us point out that there is no consensus on what the definition of confoliation in high dimensions should be. 
We hence stick to the following candidate notion, which is seemingly the simplest one:
\begin{definition}
\label{def:smooth_confoliation}
    Let $\xi$ be a hyperplane distribution on a manifold $M$ of dimension $2n+1$. We say that $\xi$ is a (positive) \textbf{confoliation} if $\xi=\ker \alpha$ with $\alpha\wedge (\d\alpha)^n\geq 0$.
\end{definition}
While being quite natural, this notion has the clear drawback of allowing for some undesired situations (e.g.\ when studying symplectic fillability questions) such as $\xi \oplus TN$ being a positive confoliation in $M\times N$, but $\xi$ being a negative contact structure on $M$. 
For this reason, these objects are not likely to lead to any theory that is as interesting as the $3$-dimensional one.

As discussed in the introduction, Eliashberg and Thurston \cite{ET_confoliations} already take the approach of considering an auxiliary geometric structure on such a $\xi$ in the hope of having more meaningful properties: namely, they require the existence of an almost complex structure $J$ on $\xi$ such that $\d\alpha(v, Jv)\geq 0$ for all $v\in \xi$. 
Even with this requirement, there is the drawback of allowing the previous product example to be a positive confoliation. (As we will see in \Cref{rmk:tameimpliesconfo}, it is the symplectic data, rather than the complex one, that rules out this example.)
Solving this ``issue'' is one of the motivations to add additional geometric structures on $\xi$, namely a complex structure and a (linear) symplectic structure (which are required to be compatible in a sense which we will describe). We start by describing the additional complex structure in the rest of this section, and then focus on the symplectic one in the next section.

\medskip

The key notion throughout the work is that of tameness, which is a strengthening of the requirement in \cite{ET_confoliations}.
First, let us introduce the following definition:
\begin{definition}
    \label{def:tamed}
    Let $\xi$ be a hyperplane distribution on $M$. 
    An almost complex structure $J_\xi$ on $\xi$ is called \textbf{tamed by $\CpS_\xi$} if 
    \begin{equation}
        \label{eqn:strictly_tamed}
        \CpS_\xi(v,J_\xi v)\geq 0 \, \text{ with equality if and only if } \, v\in \cK_\xi \, ,
    \end{equation}
    where $\cK_\xi$ is the characteristic distribution of $\xi$.
\end{definition}
\begin{remark}
    \label{rmk:tamed_invariance_Kxi}
    For all $J_\xi$ complex structure on $\xi$, if $v\in \cK_{\xi}$ or $J_\xi v\in \cK_{\xi}$ then necessarily $\CpS_\xi(v,J_\xi v) = 0$.
    Then, the condition above just means that $\CpS_\xi(v,J_\xi v) = 0$ happens for the smallest set of vectors possible, i.e.\ $\cK_\xi$.
    In particular, \eqref{eqn:strictly_tamed} implies 
 that \emph{$\cK_\xi$ is $J_\xi$ invariant}.
\end{remark}
\begin{remark}
    \label{rmk:strictly_tamed_contractible_possibly_empty}
    Note that in general, the space of tamed complex structures can be empty.
    For instance (forgetting for an instant the compatibility of all involved (co-)orientations that we assume implicitly since the footnote right before \Cref{def:char_distr}), this must be the case if $\xi$ is a hyperplane distribution on a connected manifold which is a positive contact structure on some open subset, and a negative contact structure on another open subset.
    Another explicit example, where orientations are all compatible, will be given below in \Cref{exa:non-tame_confoliation}; there, the absence of tamed complex structures on $\xi$ is due to how $\cK_\xi$ changes near a point of locally-non-constant rank.
\end{remark}
Requiring the existence of a tamed almost complex structure leads to the following notion, which is more restrictive than that proposed in \cite{ET_confoliations}\footnote{While Eliashberg--Thurston simply refer to the objects they defined as confoliations, we will always refer to our notion as ``tame confoliation'' in this paper, in order to avoid confusion with \Cref{def:smooth_confoliation}, that does not rely on any additional geometric data.}.
\begin{definition}\label{def:tameconfo}
    A confoliation $\xi$ is \textbf{tame} if there exists an almost complex structure $J_\xi$ on $\xi$ tamed by $\CpS_{\xi}$. 
\end{definition}
\begin{remark}\label{rmk:tameimpliesconfo} 
    Admitting a tamed $J_\xi$ implies that $\xi$ is a confoliation; in fact, this is true already for a $J_\xi$ such as in the definition of \cite{ET_confoliations}, and does not rely on the fact that the characteristic foliation is preserved. 
    Indeed, at any point where $\alpha\wedge (\d\alpha)^n$ is not zero, one can use an oriented complex framing of $\xi$ to deduce that its sign must be positive. Note however that this does not exclude the product $\xi=\eta\oplus TN$ on $M\times N$ of a negative contact structure $\eta$ with a manifold $N$ of positive dimension, as there $\alpha\wedge(\d\alpha)^n$ vanishes identically. In fact, $J_\eta\oplus J_N$ is tame for $J_\eta$ tamed by $\CS_\eta$ and a suitable $J_N$. This example is instead excluded by the symplectic data: if $\omega_N$ is a symplectic form on $N$ and $\mu=\pi_N^*\omega_N$, then $(\mu+t\,\d\alpha)\vert_\xi$ is negatively non-degenerate for $t>0$, so that $(\xi,\CS_{\xi,\mu})$ is not a c-symplectic confoliation.
\end{remark}
Tamed almost complex structures are adapted to the ``non-integrable" directions of $\xi$, forbidding the existence of negative contact ones. Such an almost complex structure also captures the characteristic distribution (see Remark \ref{rmk:tamed_invariance_Kxi}), and thus the geometry of the confoliation. This will be key in order to have more control when using pseudo-holomorphic curve methods.
\begin{example*}
    Let $(M,\eta)$ be a contact manifold, and $(X,J_X)$ be an almost complex one.
    Consider the hyperplane field $\xi = {\pi_M}^* \eta$ on $M\times X$, where $\pi_M\colon M\times X \to M$ is the natural projection.
    For any choice $J_\eta$ of almost complex structure on $\eta$ tamed by $\CS_\eta$, $J_\xi = J_\eta \oplus J_X$ is tamed by $\CpS_\xi = \pi^*_M \CS_\eta$. While it is easy to adapt this product example using cutoff functions to construct, say on $\R^{2n+1}$, further examples which are of non-constant-rank, we do not go into the details as we will give later on, specifically in \Cref{sec:construction}, some more interesting examples on ambient closed manifolds of such objects (and in fact of symplectic confoliations, as defined below in \Cref{def:sympl_confoliation}).
\end{example*}

We will also need the following condition.

\begin{definition}
    \label{def:geodesible}
    A hyperplane field $\xi$ on $M$ is called \textbf{geodesible} if it admits a defining one-form $\alpha$ for which there is a vector field $R$ with $\alpha(R)=1$ and $\iota_R\d\alpha=0$.
\end{definition}

\begin{remark}
    \label{rmk:geodesible}
    The terminology comes from the fact that a vector field $R$ is geodesible, i.e.\ its orbits are geodesics for some Riemannian metric, if and only if there is a one-form $\alpha$ with $\alpha(R)=1$ and $\iota_R\d\alpha=0$.
    Note that if $(\alpha,R)$ is as in \Cref{def:geodesible}, so is $(c\,\alpha,\frac1c R)$ for any constant $c>0$.
    Moreover, the condition of being geodesible is affine in the vector field; hence, given two geodesible vector fields with respect to a certain one-form, one can find a third one coinciding with each of the two on two given disjoint  closed subsets of the ambient manifold.
    We also point out that contact structures (with $R$ the Reeb vector field of $\alpha$), kernels of stable Hamiltonian structures (with $R$ their Reeb vector field), and foliations defined by closed one-forms are geodesible; so are products $\pi_M^*\eta$ of a geodesible $\eta$ on $M$ with any manifold. However, not every confoliation is geodesible. 
\end{remark}

Geodesibility poses restrictions on the holonomy of the characteristic foliation:
\begin{lemma}
    \label{lem:geodesible_confol_holonomy_strict_contacto}
    Let $\xi$ be a geodesible confoliation.
    Then, $\cK_\xi$ has holonomy made of strict contactomorphisms.
\end{lemma}
To be more explicit, $\cK_\xi$ having \textbf{holonomy made of strict contactomorphisms} means that $\xi$ admits a defining one-form $\alpha$ such that 
\[
    \cL_X \alpha =0 \quad
    \text{for all local vector fields } X\in \cK_\xi \, .
\]
The name is justified by the fact that such condition is equivalent to the following more global one: whenever there is a closed loop $\gamma$ in a leaf of $\cK_\xi$ admitting a neighborhood where $\cK_\xi$ has constant rank, then the holonomy of $\cK_\xi$ along $\gamma$ is not just a contactomorphism of the contact structure on a local transversal, but a strict one.

\begin{proof}[Proof of \Cref{lem:geodesible_confol_holonomy_strict_contacto}]
    Let $\alpha$ and $R$ be as in \Cref{def:geodesible}, and $X$ a local vector field tangent to $\cK_\xi$.
    Then $\cL_X\alpha=\iota_X\d\alpha$, and as $X\in\cK_\xi$ we have $\iota_X\d\alpha\vert_\xi=0$, i.e.\ $\iota_X\d\alpha=c\,\alpha$ with $c=\d\alpha(X,R)=-(\iota_R\d\alpha)(X)=0$.
\end{proof}

\begin{remark}\label{rem:geodesible2}
That the holonomy is made of strict contactomorphisms (as opposed to possibly non-strict contactomorphisms for a general confoliation) is rather restrictive for regular foliations in codimension-$1$ due to structural results such as Tischler's theorem\footnote{However, recall that for any regular foliation leaves having trivial holonomy are generic \cite[Theorem 2.3.12]{CanConBook}.}.
However, in higher codimensions strict contactomorphisms are somewhat more abundant, and the fact that the characteristic foliation is \emph{not} regular in our situation allows for much more freedom in the global structure of confoliations even in the case in which the characteristic foliation has trivial holonomy.
    Another point to underline is that the known results  \cite{Hin97} and \cite[Theorem 3.5.1]{ET_confoliations} about (weak) symplectic fillings of $3$-dimensional confoliations that rely on pseudo-holomorphic curves do not use weak maximum principles for confoliations directly, but rather leverage the maximum principle in the contact case through an argument of Gromov compactness for sequences of curves that are $J_n$-holomorphic, for $J_n$ a converging sequence of almost complex structures tamed by the approximating sequence of contact forms $\xi_n$ guaranteed to exist by \cite{ET_confoliations}.
    As an analogue of such an approximation result is still lacking in higher dimensions, we take the approach of focusing on the geodesible case in order to have a weak maximum principle directly for c-symplectic confoliations.
\end{remark}

\medskip

Finally, let us also mention an important difference between three dimensions and the higher-dimensional picture. 
In dimension 3, any hyperplane field inherits a homotopically unique conformal symplectic structure. 
This is not the case for hyperplane fields in higher dimensions, and thus symplectic fillability questions might have to involve fixing additional symplectic data on the hyperplane field. 
We will do this in the next section, but before, we first introduce the following auxiliary notion, which has already appeared in \cite{MNW} in the case of contact structures $\xi$, even if not explicitly given a name:
\begin{definition}
    \label{def:symplectically_compatible}
    Let $\xi$ be a hyperplane distribution on an odd-dimensional manifold $M$, and $\CpS_\xi$ its conformal partially-symplectic class. 
    An alternating two-form $\mu$ on $\xi$ is called \textbf{symplectically compatible} with $\CpS_\xi$ if, at each point $p\in M$, the ray of alternating two-forms $\mu_p + \CpS_\xi(p)$ on $\xi_p$ consists of non-degenerate alternating two-forms.
\end{definition}
More explicitly, the requirement is that $\mu+t\d\alpha$ is a non-degenerate two-form on $\xi$ for all $t>0$.
Note also that this is equivalent to asking that the cone of alternating two-forms $[\mu_p] + \CpS_\xi(p)$ on $\xi_p$ consists of non-degenerate alternating two-forms, where $[\mu_p]$ is the conformal class of $\mu_p$.
 We also point out that in general the restriction of ${\mu}$ to $\cK_\xi$ need \emph{not} be non-degenerate, as the following local example shows. 
\begin{example}
    \label{exa:sympl_compat_not_non-deg_on_K}
    On $\R^5$ with coordinates $(z,x_1,y_1,x_2,y_2)$, consider $\alpha = \d z + x_1 \d x_2$, $\xi=\ker\alpha$, and $\mu=\d x_1 \wedge \d y_1 + \d x_2\wedge \d y_2$.
    Then, for all $t> 0$, $\alpha \wedge (\mu + t\d\alpha)^2 = 2\d z \wedge \d x_1 \wedge \d y_1 \wedge \d x_2 \wedge \d y_2 $, that is non-degenerate.
    However, $\cK_\xi=\langle \partial_{y_1},\partial_{y_2}\rangle$, and $\mu\vert_{\cK_\xi}$ vanishes.
\end{example}

This being said, we have the following nice situation where non-degeneracy is guaranteed.

\begin{remark}
\label{rmk:non-deg_on_K_if_tame}
In the case of a tame confoliation $(\xi,J_\xi)$, if there is a two-form $\mu$ that is non-degenerate on $\xi$ and also tames $J_{\xi}$, then necessarily $\mu$ is symplectically compatible with $\CpS_\xi$, as required in \Cref{def:symplectically_compatible}. 
\end{remark}

This leads to the following definition.
\begin{definition}
\label{def:cotamed_complex_structure}
    Let $\xi$ be a confoliation, and $\omega$ a non-degenerate two-form on $\xi$.
    A complex structure $J_{\xi}$ on $\xi$ is \textbf{cotamed} by $\CpS_\xi$ and $\omega$ if it is tamed by $\CpS_{\xi}$ in the sense of \Cref{def:tamed} and by $\omega$ in the symplectic sense.
\end{definition}
These cotamed almost complex structures will play an important role later.

\bigskip

Ideally, we would like to say that for a hyperplane field, admitting a non-degenerate two-form symplectically compatible with $\CpS_{\xi}$,
and admitting a tamed complex structure are equivalent. This is the case in the context of contact structures \cite{MNW}. In Example \ref{exa:non-tame_confoliation} we will see that this is not true in general for confoliations. Some things are still true, however, and we now explain in more detail what can be said.

First, one can always go from a tamed complex structure to a compatible two-form:
\begin{prop}
    \label{prop:from_alm_cplx_to_sympl}
    Let $\xi$ be a hyperplane field on $M$, and $J_\xi$ a complex structure on $\xi$ tamed by $\CpS_\xi$.
    Then, there is a non-degenerate alternating two-form $\mu$ on $\xi$ that is symplectically compatible with $\CpS_\xi$ and tames $J_\xi$.
    In particular, $\mu\vert_{\cK_\xi}$ is non-degenerate.
\end{prop}

\begin{proof}
    Let $\alpha$ be a defining one-form for $\xi$, and consider an auxiliary metric $g$ on $\xi$.
    Define an alternating two-form $\mu$ on $\xi$ by
    \[
    \mu \coloneqq \frac{1}{2}\big(g(J_\xi\cdot , \cdot) - g(\cdot , J_\xi \cdot)\big) \, .
    \]
    A direct computation gives $\mu(J_\xi\cdot, J_\xi\cdot)=\mu$ and $\mu(\cdot , J_\xi \cdot) = \frac{1}{2}(g + J_\xi^*g)$, which is positive definite.
    Hence, $\mu$ is compatible with $J_\xi$, and in particular it tames $J_\xi$.

    Now, for every $t\geq 0$ and every non-zero $v\in\xi$, we have
    \[
    (\mu + t\,\d\alpha)(v,J_\xi v) = \mu(v,J_\xi v) + t\,\d\alpha(v,J_\xi v) \geq \mu(v,J_\xi v) > 0 \, ,
    \]
    where we used that $\d\alpha(v,J_\xi v)\geq 0$, as $J_\xi$ is tamed by $\CpS_\xi$.
    Therefore, $\mu+t\,\d\alpha$ tames $J_\xi$, and is thus non-degenerate on $\xi$, for all $t\geq 0$.
    This shows that $\mu$ is non-degenerate and symplectically compatible with $\CpS_\xi$ (this is independent of the choice of $\alpha$, as $\CpS_\xi$ only depends on $\xi$).

    Lastly, as $\cK_\xi$ is $J_\xi$-invariant by \Cref{rmk:tamed_invariance_Kxi}, the restriction of $\mu$ to $\cK_\xi$ still tames $J_\xi\vert_{\cK_\xi}$, and is hence non-degenerate.
\end{proof}

{
The following explicit example, due to the anonymous referee, shows that 
the converse \emph{does not hold} in full generality.

\begin{example}
    \label{exa:non-tame_confoliation}
    Consider an auxiliary smooth function $f\colon \R\to \R$ that vanishes on $\R_{\leq 0}$ and is strictly positive on $\R_{>0}$, and let $F\colon \R\to \R$ be given by $F(r)=\int_0^r f(s)\d s$.
    We then consider, on the ambient manifold $M=\R^5$ with coordinates $(z,x_1,x_2,x_3,x_4)$, the one-form $\alpha=\d z + \lambda$, where $\lambda = F(x_1)\d x_2 + F(x_2)\d x_3 + F(x_3)\d x_4$.
    Note that $\d\alpha = f(x_1)\d x_1 \wedge \d x_2 + f(x_2)\d x_2\wedge \d x_3 + f(x_3)\d x_3 \wedge \d x_4$. 
    Then, the two-form $\mu= \d x_1 \wedge \d x_2 + \d x_3 \wedge \d x_4$ is symplectically compatible with $\CpS_\xi$: indeed, for all $t\geq 0$,
    \[
        \alpha \wedge (\mu + t\d \alpha)^2 = 2[1+t(f(x_1)+f(x_3))+t^2f(x_1)f(x_3)] \, \d z \wedge \d x_1 \wedge  \d x_2 \wedge\d x_3 \wedge\d x_4 >0 \, .
    \]
    We now claim that there is no complex structure $J_\xi$ on $\xi$ that preserves $\cK_\xi$ (and thus, that $\xi$ is not tame).
    Indeed, $\cK_\xi$ is given by $\langle \partial_{x_1},\partial_{x_2}\rangle$ at each point 
    in $X:=\{x_3>0, x_1=x_2=0\}$, and by $\langle \partial_{x_1},\partial_{x_4}\rangle$ at each 
    point in $X':=\{x_2>0, x_1=x_3=0\}$. 
    As there are points in $X$ and $X'$ accumulating onto the origin of $\R^5$, if there were a $J_\xi$ preserving $\cK_\xi$ at all points, the complex structure $J_\xi\vert_0 \colon \xi_0\to \xi_0$ induced on $\xi_0=\langle\partial_{x_1},\partial_{x_2},\partial_{x_3},\partial_{x_4}\rangle$ at the origin would have to send $\partial_{x_1}$ to $\langle \partial_{x_1},\partial_{x_2}\rangle \cap \langle \partial_{x_1},\partial_{x_4}\rangle=\langle \partial_{x_1}\rangle$, which is impossible.
\end{example}

\begin{remark}
The converse to \Cref{prop:from_alm_cplx_to_sympl} does hold for c-symplectic confoliations $(\xi,\CS_{\xi,\omega})$ (which we will soon define) where $\xi$ is a contact structure, as shown in \cite[Theorem 2.4]{MNW}.
\end{remark}

}

The behavior exploited in the example above, and which makes the converse of \Cref{prop:from_alm_cplx_to_sympl} fail, is the following: for sequences $p_n$ of points in $M$ converging to a certain $p$ such that $\rk (\cK_\xi)_{p_n}$ is not equal to $\rk (\cK_\xi)_{p}$ for infinitely many $n$'s, the sequence $(\cK_\xi)_{p_n}$ might in general accumulate on several different subspaces of $(\cK_\xi)_{p}$.
This being said, a converse to \Cref{prop:from_alm_cplx_to_sympl} does hold under the hypothesis that one can find a $J_\xi$ as desired near the discontinuity points of the order function $k(p)=\frac{1}{2}[\dim(M)-1-\rk((\cK_\xi)_p)]$ of $\xi$, i.e.\ of $\rk\cK_\xi$:
\begin{lemma}
    \label{lem:from_sympl_to_alm_cplx}
    Let $\xi$ be a hyperplane field on $M$, and $\mu$ a non-degenerate alternating two-form on $\xi$ that is symplectically compatible with $\CpS_\xi$.
    Assume that there is a neighborhood $U$ of the discontinuity points of $\rk\cK_\xi$ and a complex structure $J_U$ on $\xi\vert_U$ that is tamed by $\CpS_\xi\vert_U$.
    Then, for any open neighborhood $O$ of the discontinuity points of $\rk\cK_\xi$ that is compactly contained in $U$ and such that $\mu\vert_{\cK_\xi}$ is non-degenerate on $M\setminus O$, there is $J_\xi$ on $\xi$ extending $J_U\vert_O$ on $O$ and tamed by $\CpS_\xi$, i.e.\ $\xi$ is a tame confoliation.

    Moreover, if $J_U\vert_{\cK_\xi}$ is also tamed by $\mu\vert_U$, $J_\xi$ can also be arranged to be tamed by $\mu$.
\end{lemma}

\begin{proof}
    The function $\rk \cK_\xi$ is locally constant on $M\setminus O$, hence constant on each connected component $C$ of $M\setminus O$; in other words, $\cK_\xi$ is an \emph{honest} foliation near each $C$.
    Moreover, over each such $C$, we have a constant-rank splitting $\xi\vert_C = \cK_\xi \oplus V_C$, where $V_C$ is the $\mu$-orthogonal to $\cK_\xi$, which is a complement as $\mu\vert_{\cK_\xi}$ is non-degenerate on $C$ by assumption.
    As $\d\alpha(\cK_\xi,\xi)=0$ and $\mu(\cK_\xi,V_C)=0$, one has $(\mu+t\,\d\alpha)\vert_\xi=\mu\vert_{\cK_\xi}\oplus(\mu+t\,\d\alpha)\vert_{V_C}$ for all $t\geq0$, so that $\mu\vert_{V_C}$ is symplectically compatible with the non-degenerate form $\d\alpha\vert_{V_C}$.
    We can then define $J_C$ as the direct sum of two complex structures: one on $V_C$ tamed by both $\CpS_\xi\vert_{V_C}$ and $\mu\vert_{V_C}$ (which exists by \cite[Appendix A]{MNW}), and one on $\cK_\xi$ tamed by $\mu\vert_{\cK_\xi}$.
    Using again that $\d\alpha(\cK_\xi,\xi)=0$ and $\mu(\cK_\xi,V_C)=0$, one checks directly that $J_C$ is tamed by both $\CpS_\xi\vert_C$ and $\mu\vert_C$.
    Using \Cref{prop:space_cotaming_Js_contractible}, one can then interpolate between $J_U\vert_O$ and the $J_C$'s (with an interpolation that is compactly supported in $U\setminus \overline{O}$) in order to find $J_\xi$ satisfying all the desired properties.
\end{proof}

A natural question is then that of finding a characterization, or at least a reasonable sufficient assumption, for the presence of a $J_U$ as in the second part of \Cref{lem:from_sympl_to_alm_cplx}, i.e. an almost complex structure in a neighborhood of the set of discontinuity points of $\rk \mathcal{K}_\xi$ tamed by $\CpS_{\xi}$.
The following is, for instance, (almost trivially) a sufficient condition:
\begin{lemma}
    \label{lem:sufficient_condition_J_near_discontinuity_points}
    Let $\xi$ be a hyperplane field on $M$ and $\mu$ be a non-degenerate alternating two-form on $\xi$.
    Consider also a point $p\in M$ of discontinuity for the function $\rk(\cK_\xi)$. 
    Assume that there is a neighborhood $U$ of $p$ and a splitting $\xi|_U= E_1\oplus\ldots\oplus E_N$, where the $E_i$'s are \emph{constant-rank} distributions over $U$ equipped with conformal symplectic structures $\CS_{i}$ such that:
    \begin{enumerate}
        \item the $E_i$'s are pairwise $\CpS_\xi$- and $\mu$-orthogonal;
        \item for every $q\in U$, $(\cK_\xi)_q$ is sum of some of the $(E_i)_q$'s;
        \item\label{item:CS_on_Ei} if $p_n \to p$ in $M$ and $(E_i)_{p_n}$ is not a subspace of $(\cK_\xi)_{p_n}$, then $\CpS_\xi\vert_{(E_i)_{p_n}}$ converges to $(\CS_i)_p$ on $(E_i)_p$;
        \item $\mu$ is symplectically compatible with $\CS_i$ on the $E_i$'s.
    \end{enumerate}
    Then, there is a complex structure $J_U$ on $\xi\vert_U$ tamed by $\CpS_\xi\vert_U$, i.e.\ $\xi\vert_U$ is a tame confoliation, and tamed by $\mu$.
\end{lemma}

\begin{proof}
    According to the results about almost complex structures cotaming non-degenerate alternating two-forms on vector spaces in \cite[Appendix A]{MNW}, one can find complex structures $J_i$'s on each of the $E_i$ which are tamed by both $\mu\vert_{E_i}$ and $\CS_i$ at the same time.

    By the properties of the splitting, it is not hard to check explicitly that $J_U := J_1\oplus \ldots \oplus J_N$ is tamed by $\CpS_\xi\vert_U$ and tamed by $\mu$.
    This concludes the proof.
\end{proof}
\begin{remark}
    \label{rmk:char_distr_is_closed}
    The characteristic distribution $\cK_\xi$ is ``closed'', in the sense that if $p_n \to p$ in $M$ and $v_n\in (\cK_\xi)_{p_n}$ is of norm $1$, then the sequence of the $v_n$'s has an accumulation point that (is a vector over $p$ and) lies inside $(\cK_\xi)_{p}$.
    The splitting assumption above is morally just an additional control on how this accumulation behavior occurs at discontinuity points of $\rk \cK_\xi$.
    In particular, \Cref{item:CS_on_Ei} implies that, even if for any defining one-form $\alpha$ for $\xi$ the differential $\d\alpha\vert_{(E_i)_{p_n}}$ goes to $0$, whenever $(E_i)_p\subset \cK_\xi$ then $\d\alpha\vert_{(E_i)_{p_n}}$ has a well-defined limit as a conformal symplectic structure, i.e.\ there is a sequence of constants $c_n>0$ such that $c_n\d\alpha\vert_{(E_i)_{p_n}}$ converges to a well-defined alternating two-form on $(E_i)_p$.
\end{remark}

%%%%%%%%%%%%%%%%%%%%%%%%%%%%%%%%%%%%%%%%%%%%%%%%%%%%%%%%%%%%%%%%%%%%%%

\subsection{Symplectic data on confoliations}
\label{sec:symplectic_confoliations}

We proceed to define several kinds of symplectic data that one can equip confoliations with.

\medskip

The first and weakest notion is motivated by the fact that we want to look at confoliations on hypersurfaces in symplectic manifolds, and we need to ensure some additional compatibility between the two structures.
This can be formalized using the notion of ``c-symplectic structures'', defined as follows.

\begin{definition}
    \label{def:c-sympl_class}
    Let $\xi=\ker\alpha$ be a hyperplane distribution on $M$, and $\mu$ an alternating two-form on $\xi$ which is symplectically compatible with $\CpS_\xi$.
    We call \textbf{symplectic cone directed by $[\mu]$ and $\CpS_\xi$}, and denote it by $\CS_{\xi,\mu}$, the open cone of non-degenerate alternating two-forms of the form $f\,\mu + g\, \d\alpha|_{\xi}$, for any $f,g\colon M\to \R_{>0}$.
    We will also say that $\CS_{\xi,\mu}$ is a \textit{c-symplectic structure} on $\xi$.
\end{definition}

\begin{definition}
    \label{def:cotamed_sympl_cone}
    Let $\xi$ be a hyperplane distribution. A symplectic cone $\CS_{\xi,\mu}$ is \textbf{tame} if for any element $\omega= f \mu +g \d\alpha|_{\xi}$ of the cone, there is an almost complex structure $J_\xi$ on $\xi$ cotamed by $\CpS_\xi$ and $\omega$.
\end{definition}
Note that the existence of a tame symplectic cone implies that $\xi$ is, in particular, a tame confoliation.
\begin{remark}
\label{rmk:def_tame_sympl_cone}
In practice, a way to show that a symplectic cone is tame is to find a complex structure $J$ on $\xi$ that is at the same time tamed by $\CpS_{\xi}$ and also by the partial symplectic conformal class $[\mu]$, in the sense that $\mu(v,Jv)|_p \geq 0$ with equality exactly if $v\in \ker \mu|_p$. 
\end{remark}

 \begin{remark}
 \label{rmk:def_tame_sympl_cone_bis}
 Let $\CS_{\xi,\mu}$ be a tame symplectic cone on $\xi$.
 If $\omega\in \CS_{\xi,\mu}$, by \Cref{def:cotamed_sympl_cone} there is a complex structure $J$ on $\xi$ cotamed by $\CpS_\xi$ and $\omega$. 
 Then not only the symplectic sub-cone $\CS_{\xi,\omega}\subset \CS_{\xi,\mu}$ is obviously tame just by the fact that $\CS_{\xi,\mu}$ is tame, but $J$ itself cotames $\xi$ and \emph{every} element $\omega'\in \CS_{\xi,\omega}$. We point out that, in our study of fillings below (where one can attach a piece of symplectization on top of the boundary), we will in a sense be interested at the ``behavior of the cone at infinity''; hence, in practice, up to passing to a sub-cone, when $\mu$ is non-degenerate one can simply think that the stronger assumption in \Cref{rmk:def_tame_sympl_cone} is satisfied, i.e.\ that there is a complex structure $J$ cotamed by $\CpS_\xi$ and $\mu$ directly.
 \end{remark}

\begin{remark}
    \label{rmk:strictly_cotamed_contractible_possibly_empty}
    {Similarly to what was said in \Cref{rmk:strictly_tamed_contractible_possibly_empty}, the space of almost complex structures that are cotamed by a non-degenerate two-form and $\CpS_\xi$ can potentially be empty, as it happens in \Cref{exa:non-tame_confoliation}. This being said, recall there is a sufficient condition for non-emptiness of this space in \Cref{lem:from_sympl_to_alm_cplx} above.}
\end{remark}

With these notions, we can define the first type of fiberwise symplectic data on a confoliation that we are interested in:
\begin{definition}
    \label{def:c-sympl_confoliation}
    A \textbf{cone almost-symplectic confoliation}, or \textbf{c-symplectic confoliation} in short\footnote{We name it cone almost-symplectic instead of cone-symplectic since the latter might suggest some integrability condition (as e.g.\ in the case when one speaks of a symplectic foliation, where the two-form is required to be closed when restricted to each leaf). 
    To ease notation, we drop the ``almost" and use the shorter c-symplectic throughout the rest of the paper, even though in general we make no closedness requirements; we hope this will not confuse the reader.}, is a pair $(\xi,\CS_{\xi,\mu})$ of a confoliation $\xi$ and a symplectic cone $\CS_{\xi,\mu}$ (as in Definition \ref{def:c-sympl_class}) on it 
    {such that $\CS_{\xi,\mu}\vert_{\cK_{\xi}}$ is non-degenerate}\footnote{
    {While this condition on non-degeneracy of $\CS_{\xi,\mu}\vert_{\cK_{\xi}}$ is strictly speaking not necessary for some of the properties that we will show in the following sections, it is necessary for some of them, and it automatically holds for the tame case by \Cref{rmk:non-deg_on_K_if_tame}. We hence include it in this definition directly.}
    }.
    Moreover, if $\CS_{\xi,\mu}$ is tame, the c-symplectic confoliation is called \textbf{tame}.
\end{definition}
{More explicitly, the condition that $\CS_{\xi,\mu}\vert_{\cK_{\xi}}$ is non-degenerate amounts to the fact that $\mu\vert_{\cK_\xi}$ is non-degenerate, as $\CpS_\xi$ vanishes on $\cK_\xi$ (by definition).} 

The following geometric property of a c-symplectic confoliation will be useful later on.
\begin{definition}\label{def:confo_aspherical}
    A c-symplectic confoliation $(\xi,\CS_{\xi,\mu})$ is \textbf{aspherical} if for any smooth map $u\colon S^2\to M$ such that $\operatorname{Im}(\d u)\subset \cK_{\xi}$, the following holds: if $u^*\mu\geq 0$ then $u^*\mu =0$.
\end{definition}
Note that this only depends on the conformal class of $\mu\vert_{\cK_\xi}$, hence on $\CS_{\xi,\mu}$.

\medskip

Going back to symplectic data on confoliations, one can also fix a specific symplectically compatible non-degenerate two-form on $\xi$, and then speak of an almost-symplectic confoliation. 
One can also add integrability conditions to this fiberwise symplectic data. 
For instance, we can speak of symplectic confoliation if the two-form restricted to each leaf of the characteristic foliation is closed (as it happens for symplectic foliations), or of strong symplectic confoliation if there is an additional global condition generalizing the definition of strong symplectic foliations, considered by many other authors, e.g.\ \cite{Mit18,Mar13,TorThesis,TouThesis,PreVen,Ven,MTdPP18,GirTou,GirTou24,GNT}.
In other words, we can give the following definition:
\begin{definition}
    \label{def:sympl_confoliation}
    A \textbf{symplectic confoliation} is the data $(\xi,\omega)$ of a hyperplane field $\xi$ and a two-form $\omega$ on $\xi$ such that:
    \begin{itemize}
        \item[-] $(\xi,\CS_{\xi,\omega})$ is a c-symplectic confoliation;
        \item[-] at every point $p\in M$ near which $\cK_\xi$ has constant rank, $\d(\omega\vert_{L_p})=0$ on $TL_p$, where $L_p$ is a local leaf of $\cK_\xi$ passing through $p$.
    \end{itemize}
    A symplectic confoliation $(\xi,\omega)$ is moreover called \textbf{strong} if there is a two-form $\Omega$ on $M$ which is closed (as form on $M$), its restriction to $\xi$ is symplectically compatible with $\omega$ and $\CpS_{\xi}$, and such that $\Omega\vert_{\cK_\xi}=\omega\vert_{\cK_\xi}$. 
    One defines similarly \textbf{tame symplectic} and \textbf{tame strong symplectic confoliations}.
\end{definition}
\noindent
The notion of strong symplectic confoliation can be understood as an extension not only of strong symplectic foliations (hence of taut $3$-dimensional foliations) but also of taut $3$-dimensional confoliations defined in \cite{ET_confoliations}.
One could moreover call \textbf{taut} any high-dimensional confoliation $\xi$ for which there exists a closed codegree-one form that is non-vanishing along $\xi$; strong symplectic confoliations are then in particular taut.

%%%%%%%%%%%%%%%%%%%%%%%%%%%%%%%%%%%%%%%%%%%%%%%%%%%%%%%%%%%%%%%%%%%%%%

\subsection{Symplectic fillability of confoliations with symplectic data}
\label{sec:sympl_fillability_confoliations}

We start by giving an analog of the notion of weak fillability for contact structures from \cite{MNW} in our confoliated setting. 
\begin{definition}
    \label{def:symplectic_filling_confoliation}
    A symplectic manifold $(W,\Omega)$ with non-empty boundary $M=\partial W$ is said to be a \textbf{symplectic filling} of a confoliation $\xi$ on $M$ if the two-form $\omega=\Omega\vert_M$ satisfies that $\omega|_{\xi}$ is non-degenerate and that $(\xi, \CS_{\xi,\omega})$ defines a c-symplectic confoliation. 
    Whenever $\partial W=M$ is disconnected, the symplectic filling will also be said to be a \textbf{symplectic semi-filling} of each connected component of $M$, while the latter will be said to be \textbf{symplectically co-fillable}. 
\end{definition}
Notice that, in particular, $\omega\vert_{\cK_\xi}$ is non-degenerate, by \Cref{def:c-sympl_confoliation}.

\begin{remark}
\label{rmk:semi-filling_convention}
 Concerning the number of boundary components that a filling is allowed to have, the literature seems to follow multiple conventions.
 Sometimes, fillings are required to have a single boundary component, and those with disconnected boundaries are called \emph{co-fillings}.
 Here we stick to the convention in the above definition.
\end{remark}

For our purposes, which are to give intrinsic statements depending only on the data at the boundary, we have to specify a priori some symplectic data on the confoliation at the boundary. 
We hence define the following more tractable notions from a symplectic viewpoint, which will be the ones used in the rest of the paper:

\begin{definition}
\label{def:symplectic_filling_sympl_confoliation_different_versions}
    Let $M$ be equipped with a c-symplectic confoliation $(\xi,\CS_{\xi,\mu})$. A symplectic manifold $(W,\Omega)$ with non-empty connected\footnote{We stick here to the case of fillings for simplicity; completely analogous notions can be defined in the case of semi-fillings.} boundary $M=\partial W$, with $\omega=\Omega|_{TM}$, is:
    \begin{enumerate}
        \item a \textbf{weak symplectic filling} of $(\xi,\CS_{\xi,\mu})$ if $\omega|_{\cK_\xi}=f\mu|_{\cK_{\xi}}$ for some positive function $f$, $\omega|_{\xi}$ is non-degenerate and $[\omega]$ together with $\CS_{\xi,\mu}$ direct a symplectic cone in $\xi$;
        \item a \textbf{quasi-strong symplectic filling} of $(\xi,\CS_{\xi,\mu})$ if $\omega\vert_\xi\in\CS_{\xi,\mu}$.
    \end{enumerate}
    If we work instead with a strong symplectic confoliation $(\xi,\eta)$, we can then speak of a strong symplectic filling.
    \begin{enumerate}
        \item[3.] $(W,\Omega)$ is a \textbf{strong symplectic filling} of $(\xi,\eta|_{\xi})$, where $\eta$ is closed in $M$, if $\omega=c\eta + \d\alpha$ for some defining form $\alpha$ of $\xi$ and positive constant $c>0$.
    \end{enumerate}
    We will call it just a \textbf{symplectic filling} whenever its type among these three is clear from the context or does not need to be specified.
\end{definition}

\noindent
Each item in \Cref{def:symplectic_filling_sympl_confoliation_different_versions} implies those that are above. For the sake of being explicit, we recall from \Cref{def:c-sympl_class} that $\omega\vert_\xi\in \CS_{\xi,\mu}$ simply means that, for any auxiliary defining one-form $\alpha$ for $\xi$, there are functions $f,g\colon M\to\R_{>0}$ such that $\omega\vert_\xi=f\mu + g\d\alpha$. 
We point out that if $\xi$ is a contact structure then a weak symplectic filling of $(\xi, \CS_{\xi,0})$ in the confoliated sense is the same as a weak symplectic filling in the contact sense \cite{MNW} if $\dim M\geq 5$. 
Similarly, a quasi-strong (respectively, strong) symplectic filling of $(\xi,\CS_{\xi,0})$ {(respectively, of $(\xi,0))$} is a strong symplectic filling of $\xi$ in the contact sense if $\dim M\geq 5$; in the quasi-strong case, this follows from \cite[Lemma 2.1]{McD}. 
In dimension three, for $\xi$ a contact structure, strong fillability of {$(\xi, 0)$} (in the confoliated sense) is equivalent to strong fillability of $\xi$ in the contact sense, while weak symplectic and quasi-strong symplectic fillability of $(\xi, \CS_{\xi,0})$ (in the confoliated sense) both coincide with weak fillability of the 3D contact structure $\xi$.

\begin{example}
 Let $(M,\xi)$ be a contact manifold and $(W,\Omega)$ be a weak (respectively strong) symplectic filling of $(M,\xi)$. Let $(N,\omega_N)$ be another closed symplectic manifold. Then $(W\times N, \Omega + \omega_N)$ is a weak symplectic filling (respectively, a strong symplectic filling) of $\xi\oplus TN$ equipped with $\CS_{\xi\oplus TN, \omega_N}$ {(respectively, equipped with $\omega_N$)}.
\end{example}

\medskip

As in the contact case \cite[Lemma 2.6]{MNW}, and with exactly the same proof (which works more generally in the framed Hamiltonian setting), one naturally has the following normal form near a c-symplectic confoliated boundary:
\begin{lemma}
    \label{lem:weak_sympl_fill_near_boundary}
    Let $(W,\Omega)$ be a weak symplectic filling of a c-symplectic confoliation $(\xi,\CS_{\xi,\mu})$ on $M=\partial W$.
    Then, for any auxiliary choice of defining one-form $\alpha$ for $\xi$, $M$ admits a collar neighborhood $(-\varepsilon,0]\times M$ inside $W$ where 
    \[
    \Omega = \d(t\alpha) + \Omega\vert_M \, .
    \]
\end{lemma}

We also have the following analogue of \cite[Lemma 2.9]{MNW} in our confoliated setup, which will allow us to arrange the necessary almost complex structures in \Cref{lem:deformation_lemma} below.
 
\begin{lemma}
    \label{lem:cotamed_for_T_big_enough}
    Let $\xi$ be a tame confoliation on a compact $M$, with $J_\xi$ a tamed complex structure, and $\Omega_M$ be a two-form on $M$ such that $\Omega_M\vert_{\cK_\xi}$ is non-degenerate and tames $J_\xi\vert_{\cK_\xi}$.
    Then, for any given defining one-form $\alpha$ for $\xi$, there is a large $T>0$ such that $J_\xi$ is tamed by $(\Omega_M+T\,\d\alpha)\vert_\xi$.
\end{lemma}

\begin{proof}
    Assume by contradiction that the conclusion is false, i.e.\ that for every natural number $n>0$ there is a non-zero vector $v_n\in \xi$ such that 
    \[
    (\Omega_M + n \, \d\alpha )(v_n , J_\xi v_n)\leq 0 \, .
    \]
   After normalizing $v_n$ with respect to any auxiliary Riemannian metric, up to passing to a subsequence, we can assume that $v_n\to v_\infty \neq 0$ when $n\to \infty$.
    By rescaling the left-hand side of the above inequality by $1/n$, we deduce that $\d\alpha(v_\infty , J_\xi v_\infty) \leq 0$.
    As $J_\xi$ is tamed by $\CpS_\xi$, this implies that $v_\infty \in \cK_\xi$.
    Fix now any complementary $\nu$ to $\cK_\xi$ inside $\xi$; such a $\nu$ is to be intended as a generalized (i.e.\ not constant rank) distribution on $M$.
    (Note also that $\nu$ is not closed as a subset of $TM$, but this is not an issue for what follows.)
    We then decompose the vectors as $v_n = u_n + w_n$ with $u_n \in \cK_\xi$ and $w_n\in \nu$. 
    Using that vectors in $\cK_\xi$ annihilate $\d\alpha|_\xi$ and that $\cK_\xi$ is $J_\xi$-invariant, we then compute
    \begin{equation*}
                (\Omega_M +  n \, \d\alpha )(v_n , J_\xi v_n)  =  \; 
        \Omega_M (v_n , J_\xi v_n)  
        +  n \, \d\alpha (w_n , J_\xi w_n) \; .
    \end{equation*}
    The first term is positive for $n$ big enough, since we know that $\Omega_M (v_\infty , J_\xi v_\infty)>0$ using that $J_\xi$ is tamed by $\Omega_M|_{\cK_\xi}$.
    The second term is moreover non-negative for all $n$. 
    In particular, for $n$ big enough, $(\Omega_M +  n \, \d\alpha )(v_n , J_\xi v_n) >0$; this contradicts the assumption that this quantity is $\leq 0$ for all $n>0$, thus concluding the proof.
\end{proof}

\bigskip

We now give an analogue of \cite[Lemma 2.10]{MNW}, which we will later use to change the form of the symplectic structure near the confoliated boundary:

    \begin{lemma}
    \label{lem:deformation_lemma}
    Let $\xi=\ker\alpha$ be a confoliation on a closed manifold $M^{2n+1}$ and $\Omega_M$ be a closed two-form on $M$, non-degenerate on $\xi$ and on $\cK_\xi$, symplectically compatible with $\CpS_\xi$. 
    Consider also $\Omega_M'$ another closed two-form such that $\Omega_M' = \Omega_M + \d\beta$, for some one-form $\beta$ on $M$ with compact support contained in an open subset $U\subset M$ on which the order of $\xi$ is locally constant, and such that 
    \begin{equation}
        \label{eqn:assumption_deformation_lemma}
        (\Omega_M + s\d\beta)\vert_{\cK_{\xi}} 
        \text{ is non-degenerate for all } s\in[0,1] \, .
    \end{equation}
    \noindent
    Then, for every $T>0$ big enough, there is a symplectic form $\Omega$ on $[0,T]\times M$ which restricts to $\d(t\alpha)+\Omega_M$ near $\{0\}\times M$ and to $\d(t\alpha)+\Omega_M'$ near $\{T\}\times M$, and such that $\Omega\vert_{\{t\}\times M}$ on each level $\{t\}\times M$ induces a symplectic cone $\CS_{\xi,\Omega\vert_{\{t\}\times M}}$.

    \noindent 
    If moreover there is a complex structure $J_\xi$ on $\xi$ that is cotamed by $\CpS_\xi$ and $\Omega_M$, and in addition is weakly tamed by $\d\beta\vert_{\cK_\xi}$ (i.e.\ $\d\beta(v,J_\xi v )\geq 0$ for all $v\in \cK_\xi$), 
    then, up to taking $T$ larger, $J_\xi$ is also cotamed by $\CpS_\xi$ and $\Omega\vert_{\{T\}\times M}$.
\end{lemma}

Notice that it follows from the assumption \eqref{eqn:assumption_deformation_lemma} that $\Omega'_M$ is non-degenerate on $\cK_\xi$.
We also point out that it follows from the conclusion that the hyperplane distribution $\xi$ on each level set $\{t\}\times M$ has an induced c-symplectic confoliation structure given by the symplectic cone $\CS_{\xi,\Omega\vert_{\{t\}\times M}}$.

\begin{proof}
    Consider the two-form $\Omega=\d(t\alpha)+\Omega_M + \d(\rho(t) \beta)$ on $[0,T]\times M$, with $\rho\colon [0,T]\to [0,1]$ a non-decreasing function equal to $0$ near $0$ and to $1$ near $T$.
    We then have 
    $\Omega  = \d t \wedge (\alpha + \dot\rho \beta) + t \d\alpha + \Omega_M + \rho \d\beta$, which 
    hence brings to
    \begin{align*}
    \Omega^{n+1} = & \; 
    (n+1)\d t \wedge (\alpha + \dot\rho\beta) \wedge (t\d\alpha + \rho \d\beta + \Omega_M)^{n}
    \\
    = & \;\underbrace{(n+1)\d t \wedge \alpha  \wedge (t\d\alpha + \rho \d\beta + \Omega_M)^{n}}_{\Theta_0} 
    + \underbrace{(n+1) \dot\rho \d t \wedge  \beta \wedge (t\d\alpha + \rho \d\beta + \Omega_M)^{n}}_{\Theta_1} \, .
    \end{align*}

        We first arrange $\Theta_0>0$.
    Outside $\supp\beta$, $\Theta_0=(n+1)\,\d t\wedge\alpha\wedge(t\d\alpha+\Omega_M)^n$ is positive for all $t\geq0$, because $\Omega_M$ is non-degenerate on $\xi$ and symplectically compatible with $\CpS_\xi$.
    On $\supp\beta$, which is a compact subset of $U$, the order of $\xi$ is locally constant and takes finitely many values, each on a compact subset. 
    On such a subset, where the order of $\xi$ is $k$ (i.e.\ $\cK_\xi$ has rank $2n-2k$), we have, as $\alpha\wedge(\d\alpha)^{k+1}=0$,
    \[
        \alpha\wedge(\tau\d\alpha+\Omega_M+r\,\d\beta)^n = \binom{n}{k}\,\tau^k\,\alpha\wedge(\d\alpha)^k\wedge(\Omega_M+r\,\d\beta)^{n-k} + O(\tau^{k-1}) \, .
    \]
    Since $\d\alpha\vert_\xi$ has rank $2k$ and kernel $\cK_\xi$, when evaluating the leading coefficient on a basis of $\xi$ adapted to a splitting $\xi=\nu\oplus\cK_\xi$, the factor $(\d\alpha)^k$ can only be fed with vectors in $\nu$, and hence $(\Omega_M+r\,\d\beta)^{n-k}$ only with vectors in $\cK_\xi$; thus, the leading coefficient is non-zero if and only if $(\Omega_M+r\,\d\beta)\vert_{\cK_\xi}$ is non-degenerate.
    By \eqref{eqn:assumption_deformation_lemma}, this coefficient is then positive for all $r\in[0,1]$ (it is so for $r=0$ and never vanishes), and hence bounded from below by compactness.
    It follows that there is $t_0>0$ such that $\alpha\wedge(\tau\d\alpha+\Omega_M+r\,\d\beta)^n>0$ for all $\tau\geq t_0$ and $r\in[0,1]$.
    Choosing $\rho$ with support in $[t_0,T]$ then guarantees $\Theta_0>0$ everywhere.
    
    Now, $\Theta_1$ vanishes outside $\supp\beta$, and on $\supp\beta$ it grows at most one degree faster in $t$ than $\Theta_0$, since $\d\alpha$ has rank at most $2k+2$ at points of order $k$. 
    Hence, it is enough to choose $0\leq\dot\rho\leq\frac{\eta}{1+t}$ for some small enough $\eta>0$ in order to guarantee that $\Theta_0$ dominates $\Theta_1$, and hence $\Omega^{n+1}>0$; this is compatible with $\rho=1$ near $T$ as soon as $T$ is big enough.

    The property about the symplectic cones on each $t$-slice follows from the fact that $\alpha\wedge(\Omega\vert_{\{t\}\times M}+s\,\d\alpha)^n>0$ for all $s\geq0$ (as either $\rho(t)=0$ or $t+s\geq t_0$), together with \eqref{eqn:assumption_deformation_lemma}.

    \medskip

    Lastly, for the second part of the statement, note that $J_\xi$ is tamed by $\CpS_\xi$, and that $J_\xi\vert_{\cK_\xi}$ is tamed by $(\Omega_M+\d\beta)\vert_{\cK_\xi}$, as it is tamed by $\Omega_M$ and weakly tamed by $\d\beta\vert_{\cK_\xi}$.
    The fact that $J_\xi$ is cotamed by $\CpS_\xi$ and $\Omega|_{\{T\}\times M}=\Omega_M + \d\beta + T\d\alpha$ for a large enough $T$ then follows from \Cref{lem:cotamed_for_T_big_enough} applied to the two-form $\Omega_M + \d\beta$ and to $J_\xi$.
    This concludes the proof.
\end{proof}

\bigskip 

We now end the section with some further definitions, which will not be used in this paper but which are natural and serve to set down analogues of notions that are well-known in the contact case.

For instance, analogously to \Cref{def:symplectic_filling_sympl_confoliation_different_versions}, one can naturally define the notion of symplectic cobordism between two c-symplectic confoliations:
\begin{definition}
\label{def:symplectic_cobord_sympl_confoliation}
    A symplectic manifold $(W,\Omega)$ with non-empty boundary $\partial W = M_+ \sqcup \overline{M_-}$ is said to be a \textbf{quasi-strong symplectic cobordism} from a c-symplectic confoliation $(\xi_-,\CS_{\xi_-,\mu_-})$ on $M_-$ to a c-symplectic confoliation $(\xi_+,\CS_{\xi_+,\mu_+})$ on $M_+$ if $\Omega\vert_{\xi_\pm}\in\CS_{\xi_\pm,\mu_\pm}$ at each respective boundary component. 
    \textbf{Weak} and \textbf{strong symplectic cobordisms} are defined analogously.
\end{definition}

Lastly, we explicitly point out a further notion strictly related to symplectic fillings and cobordisms, even though strictly speaking it will not be used in what follows:

\begin{definition}
    \label{def:c-symplectic_confoliated_hypersurface}
    Let $M$ be a hypersurface in a symplectic manifold $(W,\Omega)$, and $(\xi,\CS_{\xi,\mu})$ a c-symplectic confoliation on $M$.
    Then $(M,\xi,\CS_{\xi,\mu})$ is called a \textbf{c-symplectic confoliated hypersurface} in $(W,\Omega)$ if $\omega\vert_{\xi} \in \CS_{\xi,\mu}$, where $\omega=\Omega|_M$.
\end{definition}

%%%%%%%%%%%%%%%%%%%%%%%%%%%%%%%%%%%%%%%%%%%%%%%%%%%%%%%%%%%%%%%%%%%%%%

\subsection{A construction of c-symplectic confoliations}
\label{sec:construction}

In this section, we give a proof of \Cref{prop:open_book_construction}, except for its \Cref{item:open_book_construction__deformation}, whose proof will be given later in \Cref{sec:contact_approx}. 
We give a general construction, inspired by \cite[Section 3.3]{C0}, using some functions that need to satisfy certain conditions. One can check that such functions can be chosen to satisfy these conditions, but to avoid leaving this to the reader, we exhibit a simple choice of functions that works right after the proof.

\medskip

Let $(W,\Omega_W)$ be a compact symplectic manifold of dimension $2n$, which is a weak symplectic filling of a confoliation $\xi_B$ on $B=\partial W$. Namely, the pair $(\xi_B=\ker \alpha_B, \CS_{\xi_B,\omega_B})$ is a c-symplectic confoliation, where $\omega_B=\Omega_W\vert_{\xi_B}$. Since it is tame, up to attaching to $W$ a small collar as in \Cref{lem:weak_sympl_fill_near_boundary} (which replaces $\omega_B$ by $\omega_B+\eta\,\d\alpha_B$ for a small $\eta>0$), let $J_B$ be a complex structure on $\xi_B$ cotamed by $\omega_B$ and $\CpS_{\xi_B}$. Denote further by $\Omega_B = \Omega_W\vert_B$ the restriction of $\Omega_W$ to the boundary $B$ of $W$. 

Let $\varphi\colon W\rightarrow W$ be a symplectomorphism of $W$ which is the identity near $B$.  
The abstract open book $(W,\varphi)$ defines a manifold $M$ of dimension $2n+1$, equipped with an embedded open book with binding $B$ and fibration $\pi\colon M \setminus B \to S^1$. 
We will construct a pair $(\alpha,\Omega)$ on $M$ such that $(\xi=\ker \alpha, \CS_{\xi,\Omega})$ defines a tame c-symplectic confoliation. 
Moreover, the construction has an additional property: if $(\alpha_B,\Omega_B)$ is a stable Hamiltonian structure, then so is $(\alpha,\Omega)$. In this last case, the confoliation is, in particular, geodesible.

\bigskip

First, on the mapping torus piece $W_\varphi = \R\times W / \sim$, where the identification is defined as $(\psi,p)\sim (\psi-1, \varphi(p))$, we simply consider $(\alpha = \d \psi , \Omega = \Omega_W)$, where $\d\psi$ is the pullback of the canonical one-form on $S^1$ via the natural projection $W_\varphi\to S^1$ induced by the projection of $\R\times W \to \R$.
Note that $\xi=\ker\d\psi$ is the tangent space to the pages, so that $\CpS_\xi$ vanishes and $\cK_\xi=\xi$ there. Hence, any complex structure on $\xi$ tamed by $\Omega_W$ is also tamed by $\CpS_\xi$.

Note that, as $\varphi$ is the identity near $\partial W=B$, a neighborhood of the boundary of $W_\varphi$ can be identified with $S^1 \times (-\varepsilon,0]_s\times B$.
Under this natural identification, according to \Cref{lem:weak_sympl_fill_near_boundary} (up to taking $\varepsilon$ smaller, so that $\varphi$ is the identity on the collar given by the lemma), we have that the just-defined $\Omega=\Omega_W$ becomes simply $\Omega_B + \d (s \alpha_B)$.

Now, in order to see this ``from the binding's perspective'', we make the following change of coordinates, where the factors $[\delta, \delta+\varepsilon) \times S^1$ in the source are to be interpreted as a neighborhood of $\partial D^2_\delta$ inside the $2$-disk $D^2_\delta$ of radius $\delta$:
\[
[\delta,\delta+\varepsilon)_r\times S^1_\theta \times B \to S^1 \times (-\varepsilon,0]\times B \, ,
\quad 
(r,\theta,p) \mapsto (\theta , -r+\delta, p) \, .
\]
In the source of this orientation-preserving diffeomorphism, the pair $(\alpha,\Omega)$ then becomes $(\alpha = \d\theta, \Omega = \Omega_B + \d ((\delta-r)\alpha_B))$.
Let us denote $\Omega_B' = \Omega_B + \delta\, \d\alpha_B$, so that the pair can be more simply rewritten as
$(\alpha = \d\theta, \Omega= \Omega_B' - \d(r\alpha_B))$.
Note that $J_B$ tames $\Omega_B'-r\,\d\alpha_B=\Omega_B+(\delta-r)\,\d\alpha_B$ for all $r\in [0,\delta]$, because $J_B$ is cotamed by $\omega_B$ and $\CpS_{\xi_B}$. To lighten the notation, we will denote $\Omega'_B$ by $\Omega_B$.

\medskip

We now want to extend this pair over a neighborhood $U=B\times D_\delta^2$ of the binding $B\subset M$.
For this, choose polar coordinates $(r,\theta)$ on $D^2_\delta$, that naturally extend the coordinates used in the source $[\delta,\delta+\varepsilon)_r\times S^1_\theta \times B$ of the diffeomorphism above.

Consider the one-form
\[\alpha= f(r) \alpha_B + g(r) \d\theta,\]
where $f,g$ are non-increasing and non-decreasing functions, respectively, such that:
\begin{itemize}
\item[-] $f(r)=1$ near $r=0$, $f(r)=0$ near $r=\delta$,
\item[-] $g\geq 0$, $g(r)=0$ near $r=0$, $g(r)=1$ near $r=\delta$,
\item[-] $f^2+g^2\neq 0$ for all $r$,
\item[-] $f$ and $g$ are constant where $fg=0$, and $E= fg'-gf'$ is positive where $fg\neq 0$,
\item[-] $fg/E$ extends smoothly by zero to the region where $fg=0$,
\end{itemize}
and consider
\begin{align*}
\Omega&=\Omega_B + h(r)\d r\wedge \d\theta - \d(l(r) \alpha_B)\\
	&= \Omega_B + h(r)\d r\wedge \d\theta - l'\, \d r \wedge \alpha_B - l(r)\d\alpha_B,
\end{align*}
where $l$ is a non-decreasing function and $h$ a non-negative function such that 
\begin{itemize}
\item[-] $h(r)=r$ near $r=0$, $h(r)=0$ near $r=\delta$,
\item[-] $l(r)=0$ near $r=0$, $l(r)=r$ near $r=\delta$,
\item[-] (non-degeneracy along $\xi$) $F= fh+gl'>0$ for $r\in(0,\delta)$. 
\end{itemize}
Note that since $l(r)\in [0,\delta]$, $J_B$ tames $\Omega_B-l\,\d\alpha_B$ for all $r$.

\medskip

For each $r$, we let $R_B=R_B(r)$ be the vector field on $B$ determined by 
$$\alpha_B(R_B)=1, \qquad \text{and} \qquad \iota_{R_B}(\Omega_B-l\,\d\alpha_B)=0,$$
which is well defined as $(\Omega_B-l\,\d\alpha_B)\vert_{\xi_B}$ is non-degenerate. (Notice that if $(\alpha_B,\Omega_B)$ is a stable Hamiltonian structure, $R_B$ is its Reeb vector field.) Then
\[
\xi=\xi_B\oplus \langle \partial_r, Y\rangle \, ,
\]
where $Y=f\partial_\theta-gR_B$. A direct computation shows the following:
\begin{itemize}
    \item[-] the splitting is $\Omega$-orthogonal, with $\Omega\vert_{\xi_B}=\Omega_B-l\,\d\alpha_B$ and $\Omega(\partial_r,Y)=F$;
    \item[-] $\d\alpha\vert_{\xi_B}=f\,\d\alpha_B$, $\d\alpha(\partial_r,Y)=E$, and the only term mixing the two factors is $\d\alpha(v,Y)=-fg\,\d\alpha_B(v,R_B)$ for $v\in\xi_B$.
\end{itemize}
To remove the mixing term, we consider instead
\[
H=\Big\{v+\frac{fg}{E}\,\d\alpha_B(v,R_B)\,\partial_r \;\Big\vert\; v\in\xi_B\Big\}\subset\xi \, ,
\]
which coincides with $\xi_B$ where $fg=0$ and is $\d\alpha$-orthogonal to $\langle \partial_r, Y \rangle$. Hence, identifying $H$ with $\xi_B$ via the projection along $\partial_r$, the restriction $\d\alpha\vert_\xi$ is the direct sum of $f\,\d\alpha_B$ on $H$ and $\d\alpha\vert_{\langle \partial_r, Y\rangle}$, which equals $E/F$ times the area form $\Omega\vert_{\langle \partial_r, Y\rangle}$.

We define $J$ as $J_B\oplus\hat J$ with respect to $\xi=H\oplus \langle \partial_r, Y\rangle$, where $\hat J(\partial_r)= \kappa(r)^{-1}Y$, with $\kappa>0$ a function equal to $r$ near $r=0$ (so that $J$ extends smoothly over the binding) and to $1$ where $fg\neq 0$.
As $J_B$ is tamed by $\CpS_{\xi_B}$, the above block form of $\d\alpha\vert_\xi$ directly shows that $J$ is tamed by $\CpS_\xi$.
As for $\Omega$, where $fg=0$ we have $H=\xi_B$, so $J$ tames $\Omega$ by the $\Omega$-orthogonality of $\xi_B\oplus \langle \partial_r, Y\rangle$, as $J_B$ tames $\Omega_B-l\,\d\alpha_B$ and $F>0$.
Where $fg\neq 0$, $J$ is a small perturbation of $J_B\oplus\hat J$ on $\xi_B\oplus\langle\partial_r,Y\rangle$ as soon as $fg/E$ is small. We claim that $\frac{fg}{E}$ can be made arbitrarily small while keeping $F$ bounded away from zero (see \Cref{ex:functions}). Since taming is an open condition, $J$ tames $\Omega$ in this region as well.

Extending $J$ over the mapping torus $W_\varphi$ to a complex structure on $\xi$ tamed by $\Omega_W$, we get a complex structure $J$ on $\xi$ cotamed by $\CpS_\xi$ and $\Omega$ on all of $M$.
It follows that $\Omega\vert_\xi+t\,\d\alpha\vert_\xi$ tames $J$ for all $t\geq0$, so that $\Omega\vert_\xi$ is symplectically compatible with $\CpS_\xi$, $\Omega\vert_{\cK_\xi}$ is non-degenerate by \Cref{rmk:non-deg_on_K_if_tame}, and $J$ is cotamed by $\CpS_\xi$ and every element of $\CS_{\xi,\Omega}$. 
In other words, $(\xi,\CS_{\xi,\Omega})$ is a tame c-symplectic confoliation.

Lastly, if $(\alpha_B, \Omega_B)$ is a stable Hamiltonian structure, then $\ker \Omega= \langle R\rangle$ with $R=hR_B + l'\partial_\theta$, and $(\alpha,\Omega)$ is a stable Hamiltonian structure if and only if $\iota_R \d\alpha=0$, i.e.\ $hf'+l'g'=0$. 
A way to achieve this is to impose the following conditions on $f$ and $g$, as the parameter $r$ increases:
\begin{enumerate}
\item first, while $f$ and $g$ are constant equal to respectively $1$ and $0$, $h$ and $l$ are made equal to $\frac{1}{2}$ and $\frac{1}{2} r$, respectively;
\item then, $f$ is made equal to $0$ as prescribed by the previous conditions, and $g$ is chosen equal to $1-f$ in this region;
\item third, $h$ and $l$ are made equal to $0$ and $r$, respectively, so that in particular $l'>0$.
\end{enumerate}
This finishes the proof.
\hfill
\qedsymbol

\begin{example}\label{ex:functions}
    An example of a choice of functions that satisfy all the conditions in the previous proof, including the ones making the pair $(\alpha, \Omega)$ a stable Hamiltonian structure, is the following. Choose in $(0,\delta)$ numbers $\delta_1<a<b<\delta_2$. 
    Let $\phi\colon\R\to[0,1]$ be a non-increasing function with $\phi=1$ on $(-\infty,0]$, $\phi=0$ on $[1,\infty)$, $\phi'<0$ on $(0,1)$, and such that $\phi(1-\phi)/(-\phi')$ extends smoothly by zero to $\R$ (e.g.\ $\phi(x)=\psi(1-x)/(\psi(x)+\psi(1-x))$ with $\psi(x)=\exp(-e^{1/x})$ for $x>0$ and $\psi=0$ otherwise).
    We take $f(r)=\phi\big(\tfrac{r-a}{b-a}\big)$ and $g=1-f$, so that $E=fg'-gf'=-f'>0$ on $(a,b)$, which is exactly where $fg\neq 0$, and
    \[
    \frac{fg}{E}=(b-a)\,\frac{\phi(1-\phi)}{-\phi'}\Big(\frac{r-a}{b-a}\Big) \, ,
    \]
    which can be made arbitrarily small by choosing $b-a$ small enough.
    We choose a non-negative function $h$ such that $h=r$ for $r\leq \delta_1$, $h$ is positive in $(0,b]$, $h(r)=\tfrac{1}{2}$ for $r\in[a,b]$ and $h=0$ for $r\geq \delta_2$. Finally, we choose a non-decreasing function $l$ such that $l(r)=0$ near $0$, $l(r)=\tfrac{1}{2}r$ for $r\in [a,b]$, $l'(r)>0$ for $r\geq b$ and $l(r)=r$ for $r\geq \delta_2$.
    Then $F=fh+gl'$ equals $h>0$ for $r\in(0,a]$, $\tfrac12$ on $[a,b]$ (independently of $b-a$), and $l'>0$ for $r\geq b$, and $hf'+l'g'=\tfrac12(f'+g')=0$ on $[a,b]$, while $f'=g'=0$ elsewhere. Lastly, the function $\kappa$ can be taken equal to $r$ near $0$ and to $1$ for $r\geq a$.
    
\end{example}

%%%%%%%%%%%%%%%%%%%%%%%%%%%%%%%%%%%%%%%%%%%%%%%%%%%%%%%%%%%%%%%%%%%%%%
%%%%%%%%%%%%%%%%%%%%%%%%%%%%%%%%%%%%%%%%%%%%%%%%%%%%%%%%%%%%%%%%%%%%%%

\section{Restrictions on $J$-curves from cotameness}
\label{sec:restrictions_J-hol_from_strict_cotameness}

%%%%%%%%%%%%%%%%%%%%%%%%%%%%%%%%%%%%%%%%%%%%%%%%%%%%%%%%%%%%%%%%%%%%%%

\subsection{Weakly pluri-subharmonic functions and maximum principles}
\label{sec:weak_plurisubharm_max_princ}

In this section, we explain how maximum principles for $J$-holomorphic curves with or without boundary can be obtained from what we call \emph{tame} weakly pluri-subharmonic functions.
This is a straightforward extension to the setting that will later suit us of the presentation given in \cite[Section 3.1, page 192]{Book_Contact_Symplectic} for the weakly pluri-subharmonic case.

\bigskip 

Let $(X,J)$ be an almost complex manifold.
Given a smooth function $\psi\colon X\to \R$, let $\d^\C\psi \coloneqq \d\psi\circ J $.
Recall that $\psi$ is said to be a \textbf{weakly pluri-subharmonic function} if, for every $v\in TX$, $- \d\d^\C\psi (v,Jv)\geq 0$.
(If $X$ is $2$-dimensional, the prefix ``pluri-'' will be omitted.)
In other words, this means that the symmetric $(0,2)$-tensor field $g_{\psi,J}$ on $X$ given by
\[
g_{\psi,J}(v,w)= \frac{1}{2}(- \d\d^\C\psi (v,Jw) - \d\d^\C\psi (w,Jv))
\]
is positive semi-definite.

This classical notion of weak pluri-subharmonicity is, however, not enough for our purposes.
In short, the motivation is that it does not allow any topological control over the pseudo-holomorphic curves that are contained in level sets of $\psi$ (cf.\ \Cref{rmk:lack_max_principle_non_cotamed} below).
This is, in fact, the main motivation for introducing the notion of tamed complex structures in \Cref{def:tamed}; we give here the analogous stronger notion of weak pluri-subharmonicity:

\begin{definition}
    \label{def:tame_weakly_subharmonic}
    Let $(X,J)$ be an almost complex manifold. A weakly pluri-subharmonic function $\psi\colon X \to \R$ is called \textbf{tame weakly pluri-subharmonic} if $J$ is \textbf{tamed} by $-\d\d^\C\psi$ in the sense that, for every regular point $x$ of $\psi$ and every $v\in T_xX$, we have
    \[
    -\d\d^\C\psi(v,Jv)\geq 0 \, ,
    \quad
    \text{with equality if and only if }
    v\in (\cK_{\xi_M})_x \, ,
    \]
    where $M=\psi^{-1}(\psi(x))$, and $\cK_{\xi_M}$ is the characteristic distribution of the hyperplane field $\xi_M= \ker\big((-\d^\C\psi)\vert_{TM}\big)$ on (the regular part of) $M$.
\end{definition}

\noindent
We point out that, as in \Cref{rmk:tamed_invariance_Kxi}, this implies that $\cK_{\xi_M}$ is $J$-invariant.

\medskip

We will now need two classical yet fundamental results about weakly subharmonic functions on Riemann surfaces with possibly non-empty boundary. 
We give here just the statements; for their proofs, we invite the reader to consult for instance \cite[Theorems 3.2 and 3.3, pages 194--196]{Book_Contact_Symplectic}.
\begin{theorem}[Weak maximum principle]
    \label{thm:weak_maximum_principle}
    Let $(\Sigma,j)$ be a compact connected Riemann surface, and $\phi\colon \Sigma\to\R$ a weakly subharmonic function. 
    If $\phi$ attains its maximum at an interior point $z\in \Sigma\setminus \partial \Sigma$, then $\phi$ is necessarily constant.
\end{theorem}
\begin{theorem}[Boundary point lemma]
    \label{thm:boundary_point_lemma}
    Let $(\Sigma,j)$ be a compact connected Riemann surface with non-empty boundary, and $\phi\colon \Sigma\to\R$ a weakly subharmonic function. 
    If $\phi$ takes its maximum value at a point $z\in\partial \Sigma$, then $\phi$ is either constant or satisfies that, for any outward-pointing vector $v$ in $T_{z}\Sigma$, $\d_z\phi(v)>0$.
\end{theorem}

The two above results considerably restrict the behavior of $J$-holomorphic curves in the presence of a tame weakly pluri-subharmonic function.
More precisely, we have the following two consequences:
\begin{lemma}
    \label{lem:maximum_principle_curve}
    Let $(X,J)$ be an almost complex manifold and $\psi\colon X\to \R$ a tame weakly pluri-subharmonic function.
    Let also $\Sigma$ be a connected Riemann surface, and $u\colon \Sigma \to X$ a $J$-holomorphic map such that $\psi\circ u$ attains its maximum in the interior of $\Sigma$.
    Then, the image of $u$ is contained in a level set $M$ of $\psi$. More precisely, $u$ is tangent\footnote{As $\cK_{\xi_M}$ is a singular distribution, to avoid any ambiguity we specify once and for all that, by ``\textbf{$u$ is tangent at each point to $\cK_{\xi_M}$}'' we mean more precisely the following: for all $q\in \Sigma$ such that $u(q)$ is a regular point of $\psi$, the image of $\d_q u$ is included in $(\cK_{\xi_M})_{u(q)}$.} at each point to the characteristic distribution $\cK_{\xi_M}$ of the hyperplane field $\xi_M=\ker\big((-\d^\C\psi)\vert_{TM}\big)$.
\end{lemma}

\noindent
We point out that the conclusion of the lemma, of course, includes the case where $u$ is a constant map.

\bigskip

\begin{proof}
    Using $J$-holomorphicity of $u$, one can easily compute
    \begin{equation}
    \label{eqn:maximum_principle_bis}
    -u^*\d\d^\C\psi = - \d u^*(\d\psi \circ J) =
    -\d (\d\psi\circ J\circ\d u)=
    -\d (\d\psi\circ\d u \circ j) =
    - \d \d^j (\psi\circ u) \, . 
    \end{equation}

    On the other hand, using again that $u$ is $J$-holomorphic and also that $\psi$ is weakly pluri-subharmonic, one gets the following: for all non-zero $v\in T_q\Sigma$,
    \begin{equation}
    \label{eqn:maximum_principle}
    \begin{split}
    (-u^*\d\d^\C\psi)_q(v,j_qv) &= 
    (-\d\d^\C\psi)_{u(q)}(\d_q u (v), \d_q u(j_q v))
    \\ &= 
    (-\d\d^\C\psi)_{u(q)}(\d_q u (v), J_{u(q)} \d_q u(v)) \geq 0 \, .
    \end{split}
    \end{equation}
    Moreover, if $u(q)$ is a regular point of $\psi$, since $J$ is tamed by $-\d\d^\C\psi$, the equality happens exactly when $\d_q u(v)$ (or equivalently all the image of $\d_q u$, by $J$-invariance of $\cK_{\xi_M}$) is included in $\cK_{\xi_M}$, where $M$ is the level set of $\psi$ through $u(q)$. 
    (Note that this includes trivially the points $q$ such that $\d_q u = 0$.)
    
    The two above facts then mean that $\psi\circ u$ is weakly subharmonic on $\Sigma$.
    Applying \Cref{thm:weak_maximum_principle} on a closed disk around any interior point where $\psi\circ u$ attains its maximum shows that the set of such points is open; as it is also closed and $\Sigma$ is connected, $\psi\circ u$ must be constant, i.e.\ the image of $u$ is contained in a level set $M$ of $\psi$.  
    
    Now, since $\psi\circ u$ is constant, \eqref{eqn:maximum_principle_bis} gives $u^*\d\d^\C\psi=0$, so that, by \eqref{eqn:maximum_principle} and the discussion of the equality case right after said equation, the image of $\d_q u$ is contained in $\cK_{\xi_M}$ for all $q\in\Sigma$ such that $u(q)$ is a regular point of $\psi$.
    This concludes the proof.
\end{proof}

\begin{lemma}
    \label{lem:boundary_point_lemma_curve}
    Let $(X,J)$ be an almost complex manifold and $\psi\colon X\to \R$ a tame weakly pluri-subharmonic function.
    Let also $\Sigma$ be a connected Riemann surface with non-empty boundary, and $u\colon \Sigma \to X$ a $J$-holomorphic map such that $\psi\circ u$ attains its maximum at $p\in\partial\Sigma$.
    Then, one of the following alternatives holds:
    \begin{itemize}
        \item[-] either $\d_p(\psi\circ u) (v)>0$ for every $v\in T_p\Sigma$ pointing outward;
        \item[-] or the image of $u$ is contained in a level set $M$ of $\psi$;
        more precisely, $u$ is tangent at each point to the characteristic distribution $\cK_{\xi_M}$ of the hyperplane field $\xi_M=\ker\big((-\d^\C\psi)\vert_{TM}\big)$.
    \end{itemize}
\end{lemma}
As for the previous lemma, we point out that the second case in the conclusion of the lemma again includes, of course, the case where $u$ is a constant map.

The proof of \Cref{lem:boundary_point_lemma_curve} is analogous to the one of \Cref{lem:maximum_principle_curve}, but using \Cref{thm:boundary_point_lemma} (on a closed half-disk around $p$) instead of \Cref{thm:weak_maximum_principle}; the details are left to the reader.

\begin{remark}
\label{rmk:lack_max_principle_non_cotamed}
    In the case where $\psi\colon X\to \R$ is only weakly pluri-subharmonic and not tame, the conclusion in \Cref{lem:maximum_principle_curve,lem:boundary_point_lemma_curve} is weaker.
    Namely, one cannot conclude anymore that $u$ needs to be tangent to the distribution $\cK_{\xi_M}$, but only to the level sets of $\psi$.
    This is bad news if one wants to have any topological control over what the pseudo-holomorphic curves do inside the level sets of $\psi$, which is what we will need later in the setup of symplectic fillings of confoliations.
\end{remark}

\subsection{Adapted $J$'s near the boundary of fillings}

Let $(W,\Omega)$ be a weak symplectic filling of a tame c-symplectic confoliation $(M=\partial W, \xi , \CS_{\xi,\mu})$, with $M$ compact.
We assume moreover that $\xi$ is geodesible, and we fix $\alpha$ and $R$ as in \Cref{def:geodesible}.

According to \Cref{lem:weak_sympl_fill_near_boundary}, for such a choice of defining one-form $\alpha$ for $\xi$ and denoting $\omega:=\Omega\vert_{\partial W}$,  we have a tubular neighborhood of $M=\partial W$ in $(W,\Omega)$ of the form
\[
    ((-\varepsilon,0]_t\times M , \Omega=\omega+\d(t\alpha)) \, .
\]
We also assume that there is a complex structure $J_\xi$ on $\xi$ cotamed by $\CpS_\xi$ and $\omega\vert_\xi$.

The following lemma guarantees that we can then apply the weak maximum principle from \Cref{lem:maximum_principle_curve,lem:boundary_point_lemma_curve} to the $t$-coordinate:

\begin{lemma}
    \label{lem:t_is_weak_plurisubharm}
    In the above situation, on the (enlarged) symplectic manifold $((-\varepsilon,\delta)_t\times M , \Omega=\omega+\d(t\alpha))$, with $\varepsilon<1$, consider an almost complex structure $J$ such that $J\vert_\xi=J_\xi$ and $J\partial_t=\frac{1}{t+1}R$, where $R$ is the vector field above, extended so that $\d t(R)=0$.
    Then, $-\d^\C t$ is equal to $(t+1)\alpha$, and the function $\psi(t,p)=t$ is a tame weakly pluri-subharmonic function, with $\ker((-\d^\C \psi)\vert_{\{t\}\times M})=\xi$ for all $t\in(-\varepsilon,\delta)$.
    Moreover, up to attaching a long enough cylindrical end to $W$ and shrinking $\varepsilon$, $J$ is tamed by $\Omega$.
\end{lemma}

The choice above for $J\partial_t$ corresponds to the situation where $J$ is invariant under translation of the $s=\log(t+1)$ coordinate.
(The $t$-invariant choice $J\partial_t=R$ would still make $\psi$ weakly pluri-subharmonic, but not tame, as $\d\alpha(R,JR)=0$ while $R\notin\cK_\xi$.)

\begin{proof}
    The proof is a direct computation using the definition of $J$.
    Denote $h(t)=t+1$, so that $J\partial_t=\frac{1}{h}R$ and $JR=-h\partial_t$.
    We have $\d^\C t(\partial_t)=\d t (J \partial_t)=\frac{1}{h}\d t (R)=0$, and $\d^\C t(R)=\d t(JR)=-h\,\d t(\partial_t)=-h$.
    Similarly, for all $Z\in \xi$, we have
    $\d^\C t(Z)=\d t (J Z)=0$ as $JZ=J_\xi Z\in \xi$.
    In other words, the one-forms $\d^\C t$ and $-h\alpha$ evaluate equally on a basis of tangent vectors, i.e.\ they are equal. 

    Let us now check that $\psi=t$ is tame weakly pluri-subharmonic. 
    We have $-\d\d^\C t=\d(h\alpha)=\d t\wedge\alpha+h\,\d\alpha$.
    For $v=Z+a\partial_t+bR$ with $Z\in\xi$ and $a,b\in\R$, we have $Jv=J_\xi Z+\frac{a}{h}R-bh\partial_t$, and hence, using that $\iota_R\d\alpha=0$ and that $\d\alpha$ has no $\d t$-component,
    \[
        -\d\d^\C t(v,Jv)=\frac{a^2}{t+1}+(t+1)\,b^2+(t+1)\,\d\alpha(Z,J_\xi Z) \, .
    \]
    As $J_\xi$ is tamed by $\CpS_\xi$, this is non-negative, and it vanishes if and only if $a=b=0$ and $Z\in\cK_\xi$, i.e.\ if and only if $v\in\cK_\xi$. 

    For the last part, fix a $J_\xi$-invariant metric $|\cdot|$ on $\xi$.
    As $J_\xi$ is tamed by $\omega\vert_\xi$ and $\d\alpha(Z,J_\xi Z)\geq0$ for all $Z\in\xi$, we have $(\omega+t\,\d\alpha)(Z,J_\xi Z)\geq\omega(Z,J_\xi Z)$ for all $t\geq0$; hence, by compactness of $M$ and up to shrinking $\varepsilon$, there is a constant $m>0$ such that $(\omega+t\,\d\alpha)(Z,J_\xi Z)\geq m|Z|^2$ for all $Z\in\xi$ and $t\in(-\varepsilon,\delta)$.
    Let also $C_0>0$ be such that $|\omega(Z,R)|\leq C_0|Z|$ for all $Z\in\xi$.
    With $v$ as above, since $\Omega=\d t\wedge\alpha+\omega+t\,\d\alpha$ and $\iota_R\d\alpha=0$, we get
    \[
        \Omega(v,Jv)=\frac{a^2}{h}+hb^2+(\omega+t\,\d\alpha)(Z,J_\xi Z)+\frac{a}{h}\,\omega(Z,R)-b\,\omega(J_\xi Z,R) \, .
    \]
    This is bounded from below by $\big(m-\frac{C_0^2}{2h}\big)|Z|^2$, which is positive for $Z\neq0$ as soon as $C_0^2/m\leq 1-\varepsilon$; for $Z=0$ and $(a,b)\neq(0,0)$, it equals $\frac{a^2}{h}+hb^2>0$.

    Now, attaching the cylindrical end $([0,T]\times M,\omega+\d(t\alpha))$ to $W$ along $\partial W=\{0\}\times M$ and writing a collar of the new boundary $\{T\}\times M$ as above, the triple $(\omega,\alpha,R)$ is replaced by $(\omega+T\,\d\alpha,(1+T)\alpha,\frac{1}{1+T}R)$, which still satisfies $\alpha(R)=1$ and $\iota_R\d\alpha=0$, and for which $J_\xi$ is still cotamed by $\CpS_\xi$ and the restriction to $\xi$ of the new boundary form.
    Under this replacement, the constant $m$ does not decrease, as $\d\alpha(Z,J_\xi Z)\geq0$, while $C_0$ gets divided by $1+T$, as $(\omega+T\,\d\alpha)(Z,R)=\omega(Z,R)$ by $\iota_R\d\alpha=0$.
    Hence, for $T$ large enough, we have $C_0^2/m\leq1-\varepsilon$, and $J$ is tamed by $\Omega$.
    This concludes the proof.
\end{proof}

%%%%%%%%%%%%%%%%%%%%%%%%%%%%%%%%%%%%%%%%%%%%%%%%%%%%%%%%%%%%%%%%%%%%%%
%%%%%%%%%%%%%%%%%%%%%%%%%%%%%%%%%%%%%%%%%%%%%%%%%%%%%%%%%%%%%%%%%%%%%%

\section{Confoliated bLobs, normal forms, and curves near singularities} \label{sec:confoliated_bLob}

In this section, we introduce the most relevant geometric object of this work: the confoliated bordered Legendrian open book, which is a certain $(2n+1-k)$-dimensional submanifold $\tilde N$ in $M$, with $0<k\leq n$, with an induced foliation that has certain singularities. Once defined, we establish normal forms for the involved geometric structures near the singularities of this foliation in a hypothetical symplectic filling of the confoliated manifold. We conclude the section by using these normal forms to study the behavior of pseudo-holomorphic curves in a neighborhood of the singularities in the symplectic filling.

\subsection{Confoliated bLobs}\label{sec:confoliated_defi}

{We start by recalling the following notion:}

\begin{definition}
    \label{def:bordered_open_book}
    Let $N$ be a manifold with non-empty boundary. A \textbf{bordered open book decomposition} of $N$ is a pair $(B,\theta)$ of a codimension $2$ closed submanifold $B$ of $N$, called the \textbf{binding}, together with a fibration $\theta\colon N \setminus B \to S^1$ such that its fibers meet $\partial N$ transversely, and that, on a normal neighborhood $U=D^2\times B\subset N$ of $B$ with coordinates $(x,y,b)$, one has $\theta^{-1}(\theta_0)\cap U=\{x=r\cos\theta_0,y=r\sin\theta_0 \mid r> 0\}$. 
    The fibers of $\theta$ in $N\setminus B$ (and sometimes also their closure in $N$ obtained by adding $B$, with an abuse of notation) are called \textbf{pages} of the bordered open book.
\end{definition}
Note that the above implies in particular that $\ker(\d\theta)\vert_U = \ker( \frac{1}{2}(xdy-ydx))$. In a contact manifold, one has the notion of a bordered Legendrian open book.

\begin{definition}[\cite{MNW}]
    \label{def:contact_bLob}
    Let $(V,\eta)$ be a contact manifold of dimension $2m+1$.
    A compact $(m+1)$-dimensional submanifold $N_c\subset V$ with non-empty boundary is called a
    \textbf{bordered Legendrian open book}, or \textbf{contact bLob} for short, if it admits a
    bordered open book decomposition $(B_c,\theta)$ as in \Cref{def:bordered_open_book} such
    that:
    \begin{itemize}
        \item[-] all the pages of $(B_c,\theta)$ are Legendrian,
        \item[-] the boundary $\partial N_c$ is Legendrian.
    \end{itemize}
    We call $B_c$ the \textbf{binding} of $N_c$.
\end{definition}

\begin{remark}
    \label{rmk:contact_bLob}
    As pointed out in \cite{MNW}, the binding $B_c$ is then automatically isotropic, its
    tangent space being contained in the tangent spaces of the closures of the pages. The
    singular distribution $\eta\cap TN_c$ is integrable, defining a foliation of $N_c$ which is
    regular exactly on $N_c\setminus(B_c\cup\partial N_c)$, where it coincides with
    $\ker\d\theta$.
    In other words, a contact bLob is a maximally foliated submanifold whose induced foliation
    is singular along $B_c$ and $\partial N_c$, cf.\ \cite[Section 1]{Nie06}.
\end{remark}

We now introduce the object of main interest, which will serve as an obstruction to fillability.

\begin{definition}\label{def:confoliated_bLob}
    Let $(\xi,\CS_{\xi,\mu})$ be a c-symplectic confoliation on a manifold $M$ of dimension $2n+1$. A \textbf{confoliated bLob} $\tilde N\subset M$ is a $2n+1-k$ dimensional submanifold $\tilde N$ of $M$, where $0<k\leq n$, with non-empty boundary, with the following properties.
    \begin{enumerate}
        \item\label{item:confoliated_bLob_is_product} (Product structure) A neighborhood of $\tilde N$ is diffeomorphic to $U\times Y$, where $Y$ is a closed connected manifold of dimension $2(n-k)$. Under this identification, the leaves of $\cK_\xi$ are given by $\{p\}\times Y$, with $p\in U$, and $\tilde N$ is given by $N_c\times Y$ for some $k+1$ dimensional submanifold with boundary $N_c$ of $U$.
        \item\label{item:contact_blob_factor} (Contact bLob) For every $w\in Y$, the submanifold $N_c\times \{w\}$ is a contact bLob for the contact structure obtained by restricting $\xi$ to $U\times \{w\}$.
        \item \label{item:weak_exact_Lagrangian} (Weakly exact Lagrangian) There exists a closed submanifold $L\subset Y$ of dimension $n-k$ such that, for every $p\in N_c$, the submanifold $\{p\}\times L$ is Lagrangian in $\{p\}\times Y$ endowed with the non-degenerate two form induced by $\mu$. Furthermore, for any map $u\colon D^2\to M$ tangent to $\{p\}\times Y$ and whose boundary takes values in $\{p\}\times L$, if $u^*\mu\geq 0$ then $u^*\mu=0$.
    \end{enumerate}
    We call $N=N_c\times L$ the \textbf{core} of the confoliated bLob. Let $B_c$ and $\partial N_c$ denote the binding and boundary of the contact bLob $N_c$. Then the \textbf{binding} of the bLob (respectively, its core) is $B_c\times Y$ (respectively $B_c\times L$). Their \textbf{boundary} is $\partial N_c\times Y$ (respectively $\partial N_c\times L$).
\end{definition}

\begin{remark}
    In \Cref{item:contact_blob_factor} of \Cref{def:confoliated_bLob}, the contact structure
    induced on the slice $U\times\{w\}$ is independent of $w$ under the identification
    $U\times\{w\}\cong U$: indeed, $\xi$ is invariant under the flows of vector fields tangent
    to $\cK_\xi$ (\Cref{lem:char_distr_involutive_and_invariant}), which by
    \Cref{item:confoliated_bLob_is_product} are the vertical ones, so that $\xi$ descends to a
    contact structure $\xi_U$ on $U$ with $\xi=\pi_U^*\xi_U$.
    In particular the induced bordered open book on $N_c$, and hence its binding $B_c$ and its
    boundary $\partial N_c$, do not depend on $w$ either, and it is enough to check
    \Cref{item:contact_blob_factor} for a single $w\in Y$.
\end{remark}

From now on and throughout the rest of the paper, when we pick any neighborhood of a bLob $\tilde N$, it will always be assumed to be a split neighborhood of the form $U\times Y$.

\begin{example}\label{exa:product_confoliated_blob}
    The first, and most relevant for us, example can be found in the product of an overtwisted contact manifold $(M,\xi_M=\ker \alpha)$ with a symplectic manifold $(W,\omega)$ admitting a weakly exact Lagrangian $L$.  In $M$, there is a contact bLob $N_c$ (for example, a Plastikstufe) by \cite{BEM15}. 
    Consider in $M\times W$ the confoliation $\xi=\ker \alpha$ endowed with the compatible two-form $\mu=\omega$. Then $\tilde N=N_c\times W$ is a confoliated bLob of $(\xi, \CS_{\xi,\mu})$. Note that $\xi$ is geodesible.
\end{example}

\begin{example}
    \label{exa:implanting_confoliated_bLob}
    Let $(M,\xi,\CS_{\xi,\omega})$ be any c-symplectic confoliated manifold admitting a confoliated bLob, and $(N,\eta,\CS_{\eta,\mu})$ another confoliated manifold of the same dimension.
    Assume moreover that
    both the c-symplectic confoliations are honest contact near two given points $p\in M$ and $q\in N$, namely, near those points we have $\omega=0$ and $\mu=0$. Then, performing a contact connected sum near those points gives a c-symplectic confoliation on $M\# N$ which admits a confoliated bLob. Note that if $\xi$ and $\eta$ are geodesible, then so is the resulting confoliation.
\end{example}

If a c-symplectic confoliation admits a symplectic filling, in particular, there is a closed two-form that is non-degenerate along $\xi$ and the confoliation is strong symplectic in the sense of \Cref{def:sympl_confoliation}.
In the context of obstructing fillability, we will need this closed two-form to satisfy a technical condition, analogous to the exactness condition along a contact bLob \cite{MNW}.

\begin{definition}\label{def:transversely-exact_bLob}
Let $(\xi,\CS_{\xi,\mu})$ be a c-symplectic confoliation on $M$, and $\tilde N\cong N_c\times Y$ a confoliated bLob. Let also $\Omega_M$ be a closed two-form on $M$ such that $\Omega_M|_{\cK_{\xi}}=f\mu|_{\cK_{\xi}}$ for some positive function $f\in C^\infty(M)$, and such that $[\Omega_M]$ and $\CS_{\xi,\mu}$ direct a symplectic cone in $\xi$. Then, we say that $\tilde N$ is \textbf{transversely-exact with respect to $\Omega_M$} if in a sufficiently small split neighborhood $V=U\times Y$ of $\tilde N$ there exists a one-form $\gamma$ on $V$ with $\gamma|_{\cK_\xi}=0$ such that
\[
    \iota_v(\Omega_M-\d\gamma)=0 \quad \text{for every } v \text{ tangent to the slices } U\times\{w\},\ w\in Y .
\]
Equivalently, $\Omega_M|_V=\pi_Y^*\omega_Y+\d\gamma$ for a symplectic form $\omega_Y$ on $Y$. We set $\omega_\cK:=\pi_Y^*\omega_Y$.
\end{definition}

In particular, $\omega_{\cK}$ induces on $\cK_\xi$ the same conformal class as $\mu\vert_{\cK_\xi}$. Moreover, $\omega_\cK$ is then automatically closed.
The following is a direct consequence of \Cref{lem:deformation_lemma}.
\begin{lemma}
\label{lem:confol_bLob_can_assume_omega_only_char_fol}
    Let $(\xi,\CS_{\xi,\mu})$ be a c-symplectic confoliation, and let $\tilde N$ be a confoliated bLob that is transversely-exact with respect to a closed two-form $\Omega_M$ on $M$.
    Then, there are $T>0$ large enough and a symplectic form $\Omega$ on $[0,T]\times M$ which restricts to $\d(t\alpha)+\Omega_M$ near $\{0\}\times M$ and to $\d(t\alpha)+\omegacK$ in $\{T\}\times V=\{T\}\times U \times Y\subset \{T\}\times M$, where $V$ is a small enough split neighborhood of $\tilde N$, and for which each level $\{t\}\times M$ has an induced symplectic cone $\CS_{\xi,\Omega\vert_{\{t\}\times M}}$.  Furthermore, if $(\xi, \CS_{\xi,\mu})$ is tame, then there is a $J_\xi$ cotamed by $\CpS_{\xi}$ and $\Omega|_{\{T\}\times M}$. 
\end{lemma}
\begin{remark}
    \label{rmk:still_bLob_after_deformation}
    The submanifold $\{T\}\times \tilde N$ in the slice $\{T\}\times M$ as in the conclusion of the lemma is a (transversely-exact) confoliated bLob for $\xi$ equipped with the symplectic cone $\CS_{\xi,\Omega\vert_{\{T\}\times M}}$.
\end{remark}
\begin{proof}
Let $V=U\times Y$ and $\gamma$ be respectively the open set and the one-form given by the transversely-exact property of $\tilde N$ as in \Cref{def:transversely-exact_bLob}, i.e.\ such that $\Omega_M\vert_V = \omega_\cK + \d\gamma$ and $\gamma=0$ on $\cK_\xi\vert_V$.
Consider then the two-form $\Omega_M'=\Omega_M-\d(\chi\gamma)$, where $\chi$ is a cutoff function that is equal to $1$ near $\tilde N$ and to $0$ on the complement of a very small neighborhood of $\tilde N$ in $V$, so that $\Omega'_M=\omegacK$ near $\tilde N$ and $\Omega'_M=\Omega_M$ outside $V$.
As $\gamma$ vanishes on $\cK_\xi$ by assumption and $\cK_\xi$ is integrable, the restriction of $\d(\chi\gamma)= \d\chi \wedge \gamma + \chi \d\gamma$ to $\cK_\xi$ vanishes identically (recall that $\cK_\xi$ is regular over $V$). Recall as well that up to attaching a first collar using \Cref{lem:cotamed_for_T_big_enough}, we can assume that there is a $J_\xi$ cotamed by $\CpS_{\xi}$ and $\Omega_M$. We are then under the assumptions of \Cref{lem:deformation_lemma}, with $\beta = -\chi\gamma$, whose support is compactly contained in $V$, where the order of $\xi$ is constant. 
\end{proof}

We will say that \textbf{a symplectic filling $(W,\Omega)$} of a c-symplectic confoliated manifold \textbf{has a transversely-exact confoliated bLob $\tilde N$} if the confoliated bLob $\tilde N$ is transversely-exact with respect to $\Omega|_{TM}$.
In that context, up to applying \Cref{lem:confol_bLob_can_assume_omega_only_char_fol}, we will assume that $\Omega$ has a normal form of the type $\d\hat \alpha + \omegacK + \d(t\hat \alpha)$ on $(-\varepsilon,0]\times V$, where $\hat \alpha=T\alpha$ is a defining form of $\xi$, where $V$ is a (split) neighborhood of $\tilde N$ in $M$ and $(-\varepsilon,0]\times V$ lies inside a collar neighborhood $(-\varepsilon,0]\times M$ of $M=\partial W$.
The advantage of this normal form is that we are able to work with an explicit model near the confoliated bLob which is independent of the symplectic form on the filling.
\begin{remark}
\label{rem:strong_fillings_transversely_exact}
Let $(\xi,\eta)$ be a strong symplectic confoliation on $M$, with $\eta$ a closed two-form on $M$, and let $\tilde N$ be a confoliated bLob of the c-symplectic confoliation $(\xi,\CS_{\xi,\eta})$.
Assume that, on a split neighborhood $V=U\times Y$ of $\tilde N$, the form $\eta$ is basic with respect to the projection $\pi_Y\colon V\to Y$, i.e.\ $\eta\vert_V=\pi_Y^*\eta_Y$ for some two-form $\eta_Y$ on $Y$; note that $\eta_Y$ is then closed, and symplectic because $\eta\vert_{\cK_\xi}$ is non-degenerate.
Then $\tilde N$ is transversely-exact for any strong symplectic filling $(W,\Omega)$ of $(\xi,\eta)$.
Indeed, by definition $\Omega\vert_M=c\,\eta+\d\alpha$ for some constant $c>0$ and some defining one-form $\alpha$ of $\xi$, so that on $V$
\[
    \Omega\vert_M=\pi_Y^*(c\,\eta_Y)+\d\alpha \, ,
\]
and $\alpha$ vanishes on $\cK_\xi\subset\xi$.
In particular, $\Omega\vert_M$ is already of the form provided by \Cref{lem:confol_bLob_can_assume_omega_only_char_fol}, with $\omega_\cK=\pi_Y^*(c\,\eta_Y)$.
\end{remark}

To finish this section, we give a simple criterion for certain bLobs on quasi-strong symplectic fillings to be transversely-exact. It will be used later in the applications of our main theorem. 
\begin{lemma}\label{lem:astrfill}
    Let $\tilde N$ be a confoliated bLob in $(M,\xi, \CS_{\xi,\mu})$, with $\dim M= 2n+1$, and let $V=U\times Y$ be a split neighborhood of $\tilde N$. Suppose that near $\tilde N$ we have both that $2n-\operatorname{rank}\cK_\xi>2$ and that $\mu$ is basic with respect to the projection $\pi_Y\colon V\to Y$, i.e.\ $\mu|_V=\pi_Y^*\mu_Y$ for some two-form $\mu_Y$ on $Y$. Then any quasi-strong symplectic filling $(W,\Omega)$ of $(M,\xi,\CS_{\xi,\mu})$ has $\tilde N$ as a transversely-exact bLob.
\end{lemma}
\begin{proof}
    We can assume that $U$ is connected, by working on each connected component of $V$ separately. Recall that $\xi|_V=\pi_U^*\xi_U$ for a contact structure $\xi_U$ on $U$, and let $\alpha=\pi_U^*\alpha_U$ for a contact form $\alpha_U$ defining $\xi_U$. By the definition of quasi-strong filling, we have
    $$ \Omega\vert_M= f \mu + g\,\d\alpha + \alpha\wedge \gamma,$$
    where $f,g$ are positive functions and $\gamma$ is some one-form. Instead of changing the defining form of $\xi$ (which would destroy the fact that $\alpha$ is pulled back from $U$), we write $g\,\d\alpha=\d(g\alpha)-\d g\wedge\alpha$, so that
    $$ \Omega\vert_M= f \mu + \d(g\alpha) + \alpha\wedge \gamma', \qquad \gamma'\coloneqq\gamma+\d g.$$
    Let $\beta$ denote the closed form $\Omega|_V-\d(g\alpha)=f\mu+\alpha\wedge\gamma'$. If $\cF$ denotes the foliation by $U$-slices, let $R$ be a leafwise Reeb field of $\alpha$: it satisfies $$\iota_R\d\alpha=0, \quad \iota_R\alpha=1$$
    and $$\iota_R\mu=0,$$
    since $R$ is tangent to $\cF$ and $\mu$ is basic. More generally, $\iota_X\mu=0$ and $\iota_X\d\mu=0$ for any $X$ tangent to $\cF$, and $\iota_Z\alpha=0$, $\iota_Z\d\alpha=0$ for any $Z\in\cK_\xi$. We can assume that $\gamma'(R)=0$, by possibly replacing $\gamma'$ by $\gamma'-\gamma'(R)\alpha$. Differentiating the expression of $\beta$, using that it is closed, we have
    $$0=\d f\wedge \mu + f\,\d\mu + \d\alpha\wedge \gamma' - \alpha \wedge \d\gamma'.$$
    Along the foliation, this gives
    $$0 = \d\alpha\wedge \gamma' |_{\cF}- \alpha\wedge \d\gamma' |_{\cF}.$$
    Contracting with the foliated Reeb field, we get
    $$ \d\gamma'|_{\cF}=\alpha \wedge \iota_R \d\gamma' |_{\cF}.$$
    This implies that $\d\gamma' \wedge \alpha|_{\cF}=0$, and thus we conclude that $\d\alpha\wedge \gamma'|_{\cF}=0$. Since $\d\alpha$ is non-degenerate in $\xi\cap \cF$, we must have $\gamma'|_{\xi\cap \cF}=0$ (as long as $2n-\operatorname{rank}\cK_\xi>2$, this is the only step where this assumption is needed). Together with $\gamma'(R)=0$, this gives $\gamma'|_{\cF}=0$.

    We now show that $\gamma'$ also vanishes on $\cK_\xi$. Let $Z\in\cK_\xi$ be a vector tangent to the $Y$ factor. Contracting the differentiated identity with $Z$ and restricting along the foliation, the first two terms vanish (as $\mu$, $\d\mu$ and their contractions with $Z$ vanish when evaluated on a vector tangent to $\cF$), and we get
    $$0=\gamma'(Z)\,\d\alpha|_{\cF}+\alpha\wedge\iota_Z\d\gamma'|_{\cF}.$$
    Contracting again with the foliated Reeb field gives $\iota_Z\d\gamma'|_{\cF}=\alpha\wedge\iota_R\iota_Z\d\gamma'|_{\cF}$, so that $\alpha\wedge\iota_Z\d\gamma'|_{\cF}=0$, and thus $\gamma'(Z)\,\d\alpha|_{\cF}=0$. As $\d\alpha|_\cF$ vanishes nowhere, we must have $\gamma'(Z)=0$. Since $TV=T\cF\oplus\cK_\xi$, this shows that $\gamma'=0$, i.e.
    $$\Omega|_V=f\mu+\d(g\alpha).$$
    Lastly, $f\mu=\beta$ is closed, so that $0=\d f\wedge\mu+f\,\d\mu$. Contracting with any $X$ tangent to $\cF$ gives $X(f)\,\mu=0$; as $\mu$ does not vanish (it is non-degenerate on $\cK_\xi$), $f$ is constant along the leaves of $\cF$, i.e.\ $f=\pi_Y^*f_Y$ for a positive function $f_Y$ on $Y$. Hence
    $$\Omega|_V=\pi_Y^*(f_Y\mu_Y)+\d(g\alpha),$$
    where $f_Y\mu_Y$ is closed (as $\pi_Y^*$ is injective) and symplectic on $Y$. Since $g\alpha$ vanishes along $\cK_\xi$, this shows that $\tilde N$ is transversely-exact for the filling.
\end{proof}

\subsection{Model neighborhoods}\label{sec:model_neighborhoods}

We now introduce model neighborhoods (in a hypothetical symplectic filling of the confoliation) of the binding and the boundary of a bLob, to later show that sufficiently small neighborhoods of arbitrary bLobs are isomorphic to these ones.
\subsubsection{Model near the binding}\label{sec:model_near_binding}

Consider $B_c\times Y$, where $B_c$ is a binding of the contact bLob $N_c$. We now define a model neighborhood of it, which is essentially a product of the contact model in \cite[Section 4]{MNW} times $Y$.
Consider the fiber bundle
    \[
        \R^2 \times \R^2 \times T^*B_c \times Y \, .
    \]
We will also equip it with the one-form $\lambda:=u\d v-v\d u + x \d y - y\d x + \lambda_c$, where $\lambda_c$ is the canonical Liouville one-form in $T^*B_c$. The form $\lambda$ is Liouville on any slice $\R^2\times \R^2 \times T^*B_c\times \{w\}$. 
Let $\omega_0=\pi_Y^*\omega_Y$ be the pullback of a symplectic form $\omega_Y$ on $Y$, whose characteristic leaves are $\R^2\times \R^2 \times T^*B_c\times \{w\}$.

Assume further that $L$ is Lagrangian in $Y$ for $\omega_Y$.

Lastly, we will consider the almost complex structure given by 
$$J_{std}=i\oplus i \oplus J_c\oplus J_Y,$$ 
where $i$ is the standard complex structure on $\R^2$, $J_c$ on $T^*B_c$ is compatible with $\d\lambda_c$ and such that $B_c$ is totally real, with $\d f_c\circ J_c=-\lambda_c$ for $f_c$ as below, and $J_Y$ is tamed by the restriction of $\omega_0$ to each $Y$ slice. Note that $i\oplus i\oplus J_c$ is tamed by $\d\lambda$, and so $J_{std}$ is tamed by $\Omega_{std}=\d\lambda + \omega_0$.

\begin{remark}\label{rem:standard_near_Lagrangian}
    As $\omega_0=\pi_Y^*\omega_Y$ and $L$ is Lagrangian in $(Y,\omega_Y)$, Weinstein's neighborhood theorem gives a neighborhood $V_L$ of $L$ in $Y$ and a symplectomorphism from $(V_L,\omega_Y)$ onto a neighborhood of the zero section of $(T^*L,\d\lambda_{can})$, sending $L$ to the zero section.
    Since the space of complex structures tamed by a symplectic form is non-empty and contractible, we may choose $J_Y$ so that, under this identification, it agrees on a smaller neighborhood of $L$ with the standard complex structure of $T^*L$ induced by an auxiliary metric on $L$.
\end{remark}

On $\R^2\times\R^2\times T^*B_c\times Y$, consider
\begin{equation}
\label{eq:weakharmonicbinding}
    F(u,v;x,y;q,p_c;w)= \tfrac{1}{2}\left(u^2+v^2+x^2+y^2\right)+f_c(p_c) \, ,
\end{equation}
where $f_c(p_c)=\tfrac12\|p_c\|^2$ for an auxiliary metric on $B_c$, so that
$\d f_c\circ J_c=-\lambda_c$ (see \cite[Appendix B]{Nie06}).
Then
\[
    -\d F\circ J_{std}=\lambda \, ,
\]
and $F$ is tame weakly pluri-subharmonic with respect to $J_{std}$, since
$-\d^{\C}F=\lambda$ satisfies that $\d\lambda$ tames $i\oplus i\oplus J_c$ and annihilates the $Y$-directions.
Note that $F$ does not depend on the coordinate $w\in Y$.

Set $M_0\coloneqq F^{-1}(1/2)$ and $\alpha_0\coloneqq\lambda\vert_{M_0}$.
Then $\xi_0\coloneqq\ker\alpha_0$ is a confoliation on $M_0$ whose characteristic foliation is
exactly given by the slices $\{p\}\times Y$. Moreover $\xi_0$ is the hyperplane field of $J_{std}$-complex tangencies of $M_0$, and
$J_{\xi_0}\coloneqq J_{std}\vert_{\xi_0}$ is cotamed by $\CpS_{\xi_0}$ and $\mu_0$, where
$\mu_0 \coloneqq\omega_0\vert_{M_0}$.

The model confoliated bLob is
\begin{equation}
\label{eqn:bLob_in_local_model_binding}
    \tilde N_0 \coloneqq
    \{(u,0;x,y;q,0;w)\ \vert\ \ u\geq 1-\delta',\ u^2+x^2+y^2=1, w\in Y\} \, ,
\end{equation}
with core $N_0=\tilde N_0\cap\{w\in L\}$, binding $B_0=\{u=1\}\cap\tilde N_0$, and
$\tilde N_0$ is a confoliated bLob for $(M_0,\xi_0,\CS_{\xi_0,\mu_0})$, which is moreover transversely-exact.

\subsubsection{Model near the boundary}
\label{sec:model_near_boundary}

The boundary of the bLob is $\partial N_c\times Y$, and we now define a model neighborhood of it, again essentially a product of the contact model of \cite[Section 4]{MNW} with $Y$.

\medskip

Recall that $\partial N_c$ is naturally a mapping torus.
Indeed, the restriction to $\partial N_c$ of the fibration $\theta\colon N_c\setminus B_c\to S^1$
defining the bordered open book is itself a fibration
\[
    \theta_\partial \coloneqq \theta\vert_{\partial N_c}\colon \partial N_c \longrightarrow S^1 \, ,
\]
whose fiber we denote by $F_0$. If $\phi_\partial: F_0\longrightarrow F_0$ denotes the first return map of
$\theta_\partial$, then
\[
    \partial N_c \cong \faktor{F_0\times [0,1]}{(z,1)\sim(\phi_\partial(z),0)} \, .
\]
Lifting the first return map to a symplectomorphism of $T^*F_0$, we get a $T^*F_0$-bundle over $S^1$. Consider the fiber bundle $E_c$ with projection
\begin{align*}
    \pi: E_c=\R^2\times \R \times \left(\faktor{T^*F_0\times [0,1]}{\sim} \right) &\longrightarrow \R^2 \times T^*S^1\\
    (x,y;\tau; q,p;s)&\longmapsto (x,y; s,\tau),
\end{align*}
which comes equipped with a fiber-wise canonical Liouville form $\lambda_c$. Lastly, we consider the total space $E=E_c\times Y$.

\medskip

We equip $E$ with the one-form
\[
    \lambda \coloneqq (x\,\d y - y\,\d x) + \tau\,\d s + \lambda_c \, ,
\]
where $s$ denotes the coordinate of the mapping torus, i.e.\ the pullback under $\pi$ of the
base coordinate of $T^*S^1$, and $\lambda_c$ is the fiberwise canonical Liouville form.
Note that $\lambda$ is a Liouville form on each slice $E_c\times\{w\}$, $w\in Y$.
Let $\omega_0=\pi_Y^*\omega_Y$ be the pullback of a symplectic form on $Y$, whose characteristic leaves are the slices $E_c\times\{w\}$. Assume further that $L\subset Y$ is Lagrangian in $Y$.

Lastly, we consider the almost complex structure
\[
    J_{std}= i\oplus i \oplus J_c\oplus J_Y \, ,
\]
where the first $i$ is the standard complex structure on the $\R^2_{x,y}$-factor, the second is
the standard structure on $T^*S^1$, determined by $i\,\partial_\tau=\partial_s$ and
$i\,\partial_s=-\partial_\tau$, $J_c$ is a complex structure on the vertical bundle
$\ker \d\pi\cong T^*F_0$ compatible with $\d\lambda_c$, and $J_Y$ is tamed by $\omega_0$ on
each $Y$-slice.
Note also that $i\oplus i\oplus J_c$ is tamed by $\d\lambda$, so
that $J_{std}$ is tamed by $\Omega_{std}=\d\lambda+\omega_0$.

\medskip

Consider now the function
\begin{equation}
\label{eq:weakharmonicboundary}
    G\colon E\longrightarrow \R \, ,
    \qquad
    G(x,y;\tau;q,p;w)=\tfrac12\left(x^2+y^2+\tau^2\right)+f_c(p) \, ,
\end{equation}
where $f_c(p)=\tfrac12\|p\|^2$ is defined using a metric on the vector bundle
$T^*F_0\to F_0$, so that $\d f_c\circ J_c=-\lambda_c$ (see \cite[Appendix B]{Nie06}).
Then
\[
    -\d G\circ J_{std}=\lambda \, ,
\]
and, as in \cite[Lemma 3.7]{MNW}, $G$ is tame weakly pluri-subharmonic with respect to $J_{std}$ on
a neighborhood of the zero section. Indeed $-\d\d^\C G=\d\lambda$ restricted to
$E_c$-slices simplifies there to
$2\,\d x\wedge\d y+\d\tau\wedge\d s-\d\d^{\,c}f_c$, which is positive on complex lines, while
$-\d\d^\C G$ annihilates the $Y$-directions since $G$ does not depend on $w$.

Set $M_0 \coloneqq G^{-1}(1/2)$ and $\alpha_0\coloneqq\lambda\vert_{M_0}$.
Then $\xi_0\coloneqq\ker\alpha_0$ is a confoliation on $M_0$ whose characteristic foliation
is exactly given by the slices $\{p\}\times Y$.
Moreover $\xi_0$ is the field of $J_{std}$-complex tangencies of $M_0$, and
$J_{std}\vert_{\xi_0}$ is cotamed by $\CpS_{\xi_0}$ and every element in $\CS_{\xi_0,\mu_0}$ with
$\mu_0\coloneqq\omega_0\vert_{M_0}$.

The model confoliated bLob near its boundary is the image of the embedding
\begin{equation}
\label{eqn:bLob_in_local_model_boundary}
\begin{aligned}
    (-\varepsilon,0]\times \faktor{\left([0,1]\times F_0\right)}{\sim}\times Y
    &\longrightarrow E \\
    (\tau;s,q;w) &\longmapsto \left(\sqrt{1-\tau^2},\,0;\,\tau;\,q,0, s;\,w\right) \, ,
\end{aligned}
\end{equation}
whose boundary is $\{\tau=0\}$ and on which the induced singular foliation is given by
$\tau\,\d s$, as in the proof of \cite[Lemma 4.7]{MNW}, where $\tau$ is the coordinate of a collar
$(-\varepsilon,0]\times\partial N_c$ of $\partial N_c$ inside $N_c$, so that $\tau\leq 0$ on
the model bLob and $\tau=0$ exactly on its boundary.
Its core is obtained by restricting $w$ to $L$.

%%%%%%%%%%%%%%%%%%%%%%%%%%%%%%%%%%%%%%%%%%%%%%%%%%%%%%%%%%%%%

\subsection{Normal forms}
\label{sec:normal_forms}

We now establish that one can indeed work in the previously introduced models. For the upcoming statements, we let $(\xi,\CS_{\xi,\mu})$ be a c-symplectic confoliation on a manifold $M$ of dimension $2n+1$, and $\tilde N\subset M$ a bLob which is transversely-exact for a symplectic filling $(W,\Omega)$. 
We also assume that $\xi$ is geodesible, and fix a defining one-form $\alpha$ and a vector field $R$ as in \Cref{def:geodesible}.
As in \Cref{lem:geodesible_confol_holonomy_strict_contacto}, we have $\cL_v\alpha=0$ for every $v\in\cK_\xi$; in particular, on a split neighborhood $U\times Y$ of $\tilde N$, whose leaves $\{p\}\times Y$ are compact and connected, we have $\alpha=\pi_U^*\alpha_U$ for a contact form $\alpha_U$ on $U$.
Moreover, as the conditions $\alpha(R)=1$ and $\iota_R\d\alpha=0$ are affine in $R$, up to interpolating with a cutoff function we can assume that, near $\tilde N$, $R$ is the lift of the Reeb vector field of $\alpha_U$. 
Lastly, up to multiplying $\alpha_U$ by a positive function supported near $B_c\cup\partial N_c$ (and modifying $R$ accordingly, so that it is still the lift of the Reeb vector field of $\alpha_U$), we can assume that the contactomorphisms from neighborhoods of $B_c$ and $\partial N_c$ in $U$ to the models of \Cref{sec:model_near_binding,sec:model_near_boundary}, given by \cite{Nie06} and \cite[Lemma 4.6]{MNW} and used below, pull back $\alpha_0$ to $\alpha_U$.

%%%%%%%%%%%%%%%%%%%%%%%%%%%%%%%%%%%%%%%%%%%%%%%%%%%%%%%%%%%%

\subsubsection{Normal form near the binding}

We now show that a split neighborhood of the binding of a bLob can be identified with a neighborhood of the binding on the model.
\begin{lemma}
    \label{lem:normal_form_J_near_core}
    Up to possibly attaching a topological cylindrical end (as in
    \Cref{lem:confol_bLob_can_assume_omega_only_char_fol}), there are a neighborhood $O$ of the
    binding $B_c\times Y$ of $\tilde N$ inside the filling $(W,\Omega)$, an
    $\Omega\vert_O$-tamed almost complex structure $J_O$ on $O$, and a pseudo-holomorphic
    embedding
    \[
        \Psi_O\colon (O,J_O)\longrightarrow
        \left(\R^2\times\R^2\times T^*B_c\times Y,\ J_{std}=i\oplus i\oplus J_c\oplus J_Y\right)
    \]
    onto its image, such that the following hold.
    \begin{enumerate}
        \item\label{item:normal_form_J_manifold}
        $\Psi_O(M\cap O)=M_0\cap \Psi_O(O)$, where $M_0=F^{-1}(1/2)$ with $F$ as in
        \eqref{eq:weakharmonicbinding}; moreover $\Psi_O$ matches $\xi$ with $\xi_0$, $\cK_\xi$
        with the foliation by the slices $\{p\}\times Y$, and $\omegacK$ with $\omega_0$.

        \item\label{item:normal_form_J_near_core_image}
        The image of $\Psi_O$ is the compact domain with corners
        \[
            U_O= F^{-1}\left([0,1/2]\right)\cap\{u\geq 1-\delta\} \, ,
            \qquad \partial U_O=\partial_+U_O\cup\partial_-U_O \, ,
        \]
        where $\partial_+U_O=U_O\cap F^{-1}(1/2)$ and $\partial_-U_O\subset\{u=1-\delta\}$, for
        some $0<\delta<\delta'$ with $\delta'$ as in \eqref{eqn:bLob_in_local_model_binding}.

        \item\label{item:normal_form_J_near_core_bLob}
        $\Psi_O(\tilde N\cap O)=\tilde N_0\cap U_O$ and $\Psi_O(N\cap O)=N_0\cap U_O$, with
        $\tilde N_0$ and $N_0$ as in \eqref{eqn:bLob_in_local_model_binding}; in particular
        $\Psi_O(B_c\times Y)=B_0$ and $\Psi_O(B_c\times L)=B_0\cap\{w\in L\}$.

        \item\label{item:normal_form_J_near_core_nbhd_bdry}
        Let $U\times Y$ be a split neighborhood of the confoliated bLob $\tilde N$. Then, there is an open neighborhood $U_B$ of the binding of the contact bLob $N_c$ in $U$ such that an open neighborhood $V_O$ of $O$ inside $W$ is diffeomorphic to $(-\varepsilon,0]\times U_B \times Y$, with the following properties. First, in this neighborhood $\Omega=\omegacK+\d\alpha+\d(t\alpha)$, with $\alpha$ a pullback from the projection to $U_B$ and $\omegacK$ a pullback from the projection to $Y$. Second, the almost complex structure $J_O$ is the restriction to $O$ of an $\Omega$-tamed almost complex structure $J_{V_O}$ which is split with respect to the $(-\varepsilon,0]\times U_B$ and $Y$ factors, has complex tangencies $\xi$ along $M\cap V_O$, and such that $-\d(\d t\circ J_{V_O})$ is basic with respect to the $Y$-factor and tames $J_{V_O}$ restricted to the $(-\varepsilon,0]\times U_B$ factor. 
    \end{enumerate}
\end{lemma}

{
\begin{proof}
 By \Cref{item:confoliated_bLob_is_product} of \Cref{def:confoliated_bLob}, a neighborhood
    of $\tilde N$ in $M$ is identified, say by some diffeomorphism $\varphi$, with $U\times Y$, with the leaves of $\cK_\xi$ given by
    the slices $\{p\}\times Y$ and $\tilde N=N_c\times Y$.
    Recall that $\xi$ is invariant under the flows of
    vector fields tangent to $\cK_\xi$, so it descends to a
    hyperplane field (in fact, a contact structure) $\xi_U$ on $U$ with $\xi=\pi_U^*\xi_U$, where the order of $\xi$ along $\tilde N$ is $k$, so $\dim U=2k+1$.
    By \Cref{item:contact_blob_factor} of \Cref{def:confoliated_bLob}, $N_c\subset U$ is a
    contact bLob for $\xi_U$, with binding $B_c$.
    As pointed out at the beginning of \Cref{sec:normal_forms}, the fixed defining form $\alpha$ is of the form $\pi_U^*\alpha_U$ for a contact form $\alpha_U$ on $U$ defining $\xi_U$, and near $\tilde N$ the vector field $R$ is the lift of the Reeb vector field of $\alpha_U$.

    By the transverse exactness of $\tilde N$ with
    respect to $\Omega$, up to attaching a cylindrical end as in
    \Cref{lem:confol_bLob_can_assume_omega_only_char_fol} we may assume that on a neighborhood
    of $\tilde N$ in $M$ we have
    \[
        \Omega\vert_M=\d\alpha+\omegacK \, ,
    \]
    where $\alpha$ now denotes a constant multiple of our previously fixed defining form, with $\omegacK$ closed, symplectic on each slice $\{p\}\times Y$, and with characteristic foliation given by the slices $U\times\{w\}$.
    In particular, $\omegacK=\pi_Y^*\omega_Y$, and, by
    \Cref{item:weak_exact_Lagrangian} of \Cref{def:confoliated_bLob}, $\{p\}\times L$ is Lagrangian in $(\{p\}\times Y,\omegacK)$ for every $p\in U$.
    \medskip

By \Cref{lem:weak_sympl_fill_near_boundary} and the above, a collar neighborhood of $M$ in
    $W$ over a neighborhood of $\tilde N$ splits as
    \begin{equation}\label{eq:symp_split}
        \left((-\varepsilon,0]\times U\right)\times Y \, ,
        \qquad
        \Omega=\d\left((1+t)\alpha_U\right)+\omegacK \, ,
    \end{equation}
    where we have used that $\alpha=\pi_U^*\alpha_U$ is pulled back from the contact factor, so
    that $\d\alpha+\d(t\alpha)$ is as well.
    In particular $\left((-\varepsilon,0]\times U,\ \d((1+t)\alpha_U)\right)$ is a piece of the
    symplectization of $(U,\alpha_U)$, of dimension $2k+2$.

    We can then apply to it the normal form for a contact bLob near its binding of
    \cite[Section 4]{Nie06} (cf.\ also the proof of \cite[Theorem 4.4]{MNW}): there are a neighborhood $O_c$ of $B_c$ inside $(-\varepsilon,0]\times U$ and an embedding 
    $$\psi\colon O_c\to\R^2\times\R^2\times T^*B_c$$ with $\psi^*\d\lambda=c\,\d((1+t)\alpha_U)$, such that:
    \begin{itemize}
                \item[-] $\psi\left(\left(\{0\}\times U\right)\cap O_c\right)\subset F^{-1}(1/2)$, and
        $\psi^*\alpha_0=c\,\alpha_U$ there, for a constant $c>0$ (by the choice of $\alpha_U$ at the beginning of \Cref{sec:normal_forms}, $c$ stands for the constant rescaling of $\alpha$ above),
        \item[-] $\psi_*((1+t)\partial_t)=X_F$, where $X_F=\frac12(u\partial_u+v\partial_v+x\partial_x+y\partial_y)+\nabla f_c$ is the Liouville vector field of $\lambda$, i.e.\ $\iota_{X_F}\d\lambda=\lambda$,
        \item[-] $\psi\left(N_c\cap O_c\right)$ is contained in the model contact bLob
        $\{(u,0;x,y;q,0)\mid u\geq 1-\delta',\ u^2+x^2+y^2=1\}$, and $\psi(B_c)=\{u=1\}$,
        \item[-] $\psi$ is pseudo-holomorphic for a suitable $\d((1+t)\alpha_U)$-tamed almost
        complex structure on $O_c$ and for $i\oplus i\oplus J_c$ on the model.
    \end{itemize}
    Here $\lambda$ and $F$ are as in \Cref{sec:model_near_binding}.
    \medskip
    
    Now set
    \[
        \Psi\coloneqq\left(\psi\times\operatorname{id}_Y\right)\circ\varphi \, ,
    \]
    defined on a neighborhood of $B_c\times Y$ in $W$.
    It sends $\tilde N$ to $\tilde N_0$ and $N$ to $N_0$, matches $\xi=\pi_U^*\xi_U$ with
    $\xi_0$ and $\cK_\xi$ with the foliation by the $Y$-slices, and, by the first bullet above,
    sends $M\cap \image(\Psi)$ into $M_0=F^{-1}(1/2)$.
    We let $\omega_0\coloneqq(\Psi^{-1})^*\omegacK$, which is closed, symplectic on
    each $Y$-slice, and has the $\R^2\times\R^2\times T^*B_c$-slices as characteristic leaves. Let $J_Y$ be a complex structure on the $Y$-directions tamed by $\omega_0$ on each slice,
    chosen as in \Cref{rem:standard_near_Lagrangian}. Lastly, we set
    $$J_{std}=i\oplus i\oplus J_c\oplus J_Y$$ and $$J_O\coloneqq\Psi^*J_{std}.$$
    Since $J_{std}$ is tamed by $\d\lambda+\omega_0$, it follows that $J_O$ is tamed by $\Omega\vert_O$.

    \medskip
    Lastly, choose $0<\delta<\delta'$ small enough that $U_O= F^{-1}([0,1/2])\cap\{u\geq 1-\delta\}$ is
    contained in $\image(\Psi)$, and set
    $$O\coloneqq\Psi^{-1}(U_O)$$ 
    and 
    $$\Psi_O\coloneqq\Psi\vert_O.$$
    \medskip

    For Item \ref{item:normal_form_J_near_core_nbhd_bdry}, notice that, up to shrinking $\varepsilon$ and choosing $U_B\subset U$ a small enough open neighborhood of $B_c$ so that $(-\varepsilon,0]\times U_B\subset O_c$, the open set
    \[
        V_O\coloneqq(-\varepsilon, 0]\times U_B\times Y
    \]
    is contained in the domain of $\Psi$, and, up to shrinking $\delta$ further, it contains $O$. As we claimed, the form $\Omega$ writes on $V_O$ as in Equation \ref{eq:symp_split}, i.e.\ $\Omega=\omegacK+\d\alpha+\d(t\alpha)$ with $\alpha=\pi_U^*\alpha_U$ and $\omegacK$ pulled back from $Y$. Moreover, $J_O$ is the restriction to $O$ of $J_{V_O}\coloneqq(\Psi|_{V_O})^*J_{std}=\psi^*(i\oplus i\oplus J_c)\oplus J_Y$, which is split with respect to the $(-\varepsilon,0]\times U_B$ and $Y$ factors, as $\Psi=(\psi\times\operatorname{id}_Y)\circ\varphi$. Lastly, as $t$ does not depend on the $Y$-factor, $-\d(\d t\circ J_{V_O})$ is basic with respect to the $Y$-factor, and $-\d(\d t\circ J_{V_O})$ tames the induced almost complex structure in $(-\varepsilon,0]\times U_B$. The latter property comes directly from the contact framework in \cite{Nie06, MNW}. 
\end{proof}

\subsubsection{Normal form near the boundary}

As for the boundary, we have the following.

\begin{lemma}
    \label{lem:normal_form_J_boundary}
    Up to possibly attaching a topological cylindrical end (as in
    \Cref{lem:confol_bLob_can_assume_omega_only_char_fol}), there are a neighborhood $O'$ of
    $\partial N_c\times Y$ inside the filling $(W,\Omega)$, an $\Omega\vert_{O'}$-tamed almost
    complex structure $J_{O'}$ on $O'$, and a pseudo-holomorphic embedding
    \[
        \Psi_{O'}\colon (O',J_{O'})\longrightarrow
        \left(E=E_c\times Y,\ J_{std}=i\oplus i\oplus J_c\oplus J_Y\right)
    \]
    onto its image, such that the following hold.
    \begin{enumerate}
        \item\label{item:normal_form_J_boundary_manifold}
        $\Psi_{O'}(M\cap O')=M_0\cap\image(\Psi_{O'})$, where $M_0=G^{-1}(1/2)$ with $G$ as in
        \eqref{eq:weakharmonicboundary}; moreover $\Psi_{O'}$ matches $\xi$ with $\xi_0$,
        $\cK_\xi$ with the foliation by the slices $\{p\}\times Y$, and $\omegacK$ with
        $\omega_0$.

        \item\label{item:normal_form_J_near_boundary_image}
        The image of $\Psi_{O'}$ is the compact domain with corners
        \[
            U_{O'}=G^{-1}\left([0,1/2]\right)\cap\{x\geq 1-\delta\} \, ,
            \qquad \partial U_{O'}=\partial_+U_{O'}\cup\partial_-U_{O'} \, ,
        \]
        where $\partial_+U_{O'}=U_{O'}\cap G^{-1}(1/2)$ and
        $\partial_-U_{O'}\subset\{x=1-\delta\}$, for some $\delta>0$ with
        $\delta<1-\sqrt{1-\varepsilon^2}$, $\varepsilon$ being as in
        \eqref{eqn:bLob_in_local_model_boundary}.

        \item\label{item:normal_form_J_near_boundary_bLob}
        $\Psi_{O'}$ maps $\tilde N\cap O'$ onto the intersection with $U_{O'}$ of the model
        confoliated bLob \eqref{eqn:bLob_in_local_model_boundary}, and $N\cap O'$ onto its core;
        in particular $\Psi_{O'}(\partial N_c\times Y)\subset\{\tau=0\}$.

        \item\label{item:normal_form_J_near_boundary_nbhd_bdry}
        The analogue of \Cref{item:normal_form_J_near_core_nbhd_bdry} of
        \Cref{lem:normal_form_J_near_core} holds, with $O'$ and $U_\partial$ in place of $O$ and $U_B$.
    \end{enumerate}
\end{lemma}

 \begin{proof}
    The first part of the argument is identical to that in the proof of
    \Cref{lem:normal_form_J_near_core}, and we only recall it briefly.
    By \Cref{item:confoliated_bLob_is_product} of \Cref{def:confoliated_bLob}, a neighborhood
    of $\tilde N$ in $M$ is identified, say by some diffeomorphism $\varphi$, with $U\times Y$,
    with the leaves of $\cK_\xi$ given by the slices $\{p\}\times Y$ and $\tilde N=N_c\times Y$;
    the confoliation $\xi$ descends to a contact structure $\xi_U=\ker\alpha_U$ on the
    $(2k+1)$-dimensional $U$ with $\xi=\pi_U^*\xi_U$, and by
    \Cref{item:contact_blob_factor} of \Cref{def:confoliated_bLob} the submanifold
    $N_c\subset U$ is a contact bLob for $\xi_U$, now regarded near its boundary $\partial N_c$.
    As before, the fixed defining form $\alpha$ is of the form $\pi_U^*\alpha_U$ for a contact form $\alpha_U$ on $U$, and near $\tilde N$ the vector field $R$ is the lift of the Reeb vector field of $\alpha_U$.

    By the transverse exactness of $\tilde N$ with respect to $\Omega$, up to attaching a
    cylindrical end as in \Cref{lem:confol_bLob_can_assume_omega_only_char_fol} we may assume
    that $\Omega\vert_M=\d\alpha+\omegacK$ on a neighborhood of $\tilde N$ in $M$, where $\alpha$ is now a constant multiple of the previously fixed defining form, with
    $\omegacK$ closed, symplectic on each slice $\{p\}\times Y$, and with characteristic
    foliation given by the slices $U\times\{w\}$; in particular, $\{p\}\times L$ is Lagrangian
    in $(\{p\}\times Y,\omegacK)$ for every $p\in U$.
    Consequently, by \Cref{lem:weak_sympl_fill_near_boundary}, a collar neighborhood of $M$ in
    $W$ over a neighborhood of $\tilde N$ splits as
    \[
        \left((-\varepsilon,0]\times U\right)\times Y \, ,
        \qquad
        \Omega=\d\left((1+t)\alpha_U\right)+\omegacK \, ,
    \]
    the first factor being a piece of the symplectization of $(U,\alpha_U)$, of dimension
    $2k+2$.

    \medskip

    We now apply to it the normal form for a contact bLob near its \emph{boundary}, given by
    \cite[Lemma 4.6 and the proof of Lemma 4.7]{MNW}.
    There, the germ uniqueness of \cite[Lemma 4.6]{MNW} is used to produce a contactomorphism
    between $\partial_+U_{E_c}\coloneqq U_{E_c}\cap G^{-1}(1/2)$ and a neighborhood of
    $\partial N_c$ in $U$, matching the induced singular foliations, which is then extended to
    a symplectic embedding of $U_{E_c}$ into the ambient symplectic manifold sending
    $\partial_+U_{E_c}$ into the boundary; here $U_{E_c}\subset E_c$ is a small neighborhood of
    $\{(1,0)\}\times\faktor{(\R\times F_0)}{\sim}$.
    Inverting this embedding onto its image, we obtain a neighborhood $O'_c$ of $\partial N_c$
    inside $(-\varepsilon,0]\times U$ and a symplectomorphism
    \[
        \psi\colon \left(O'_c,\ c \d((1+t)\alpha_U)\right)\longrightarrow
        \left(E_c,\ \d\lambda\right)
    \]
    onto its image, such that:
    \begin{itemize}
                \item[-] $\psi\left(\left(\{0\}\times U\right)\cap O'_c\right)\subset G^{-1}(1/2)$, and
        $\psi^*\alpha_0=c\,\alpha_U$ there, for a constant $c>0$ (by the choice of $\alpha_U$ at the beginning of \Cref{sec:normal_forms}, where $c$ stands for the constant rescaling of $\alpha$ above),
        \item[-] $\psi_*((1+t)\partial_t)=X_G$, where $X_G=\frac12(x\partial_x+y\partial_y)+\tau\partial_\tau+\nabla f_c$ is the Liouville vector field of $\lambda$, i.e.\ $\iota_{X_G}\d\lambda=\lambda$; in other words, $\psi(t,m)=\phi_{X_G}^{\log(1+t)}(\psi(0,m))$, so that $\psi^*\lambda=c\,(1+t)\alpha_U$,
        \item[-] $\psi\left(N_c\cap O'_c\right)$ is contained in the model contact bLob
        \[
            \left\{\left(\sqrt{1-\tau^2},\,0;\,\tau;\,q,0\right)\ \middle|\
            -\varepsilon< \tau\leq 0\right\} \, ,
        \]
        and $\psi(\partial N_c)\subset\{\tau=0\}$.
    \end{itemize}
    Here $E_c$, $\lambda$ and $G$ are as in \Cref{sec:model_near_boundary}.
    The only structural difference with respect to the binding case is that the fiber factor of
    $E_c$ is the mapping torus $\faktor{T^*F_0\times[0,1]}{\sim}$ associated with the first
    return map of $\theta_\partial\colon\partial N_c\to S^1$, instead of the honest cotangent
    bundle $T^*B_c$.

    \medskip

    Now set
    \[
        \Psi\coloneqq\left(\psi\times\operatorname{id}_Y\right)\circ\varphi \, ,
    \]
    defined on a neighborhood of $\partial N_c\times Y$ in $W$.
    It sends $\tilde N$ to the model confoliated bLob
    \eqref{eqn:bLob_in_local_model_boundary} and $N$ to its core, matches $\xi$ with $\xi_0$
    and $\cK_\xi$ with the foliation by the $Y$-slices, and, by the first bullet above, sends
    $M\cap\image(\Psi)$ into $M_0=G^{-1}(1/2)$.
    We let $\omega_0\coloneqq(\Psi^{-1})^*\omegacK$, which is closed, symplectic on each
    $Y$-slice, and has the $E_c$-slices as characteristic leaves.
    Let $J_Y$ be a complex structure on the $Y$-directions tamed by $\omega_0$ on each slice,
    chosen as in \Cref{rem:standard_near_Lagrangian}.
    Lastly, we set
    $$J_{std}=i\oplus i\oplus J_c\oplus J_Y$$
    and
    $$J_{O'}\coloneqq\Psi^*J_{std} \, .$$
    Since $J_{std}$ is tamed by $\d\lambda+\omega_0$, it follows that $J_{O'}$ is tamed by
    $\Omega\vert_{O'}$.

    \medskip

    Lastly, choose $\delta>0$ with $\delta<1-\sqrt{1-\varepsilon^2}$ and small enough that
    $U_{O'}=G^{-1}([0,1/2])\cap\{x\geq 1-\delta\}$ is contained in $\image(\Psi)$, and set
    $$O'\coloneqq\Psi^{-1}(U_{O'})$$
    and
    $$\Psi_{O'}\coloneqq\Psi\vert_{O'} \, .$$
\medskip

For Item \ref{item:normal_form_J_near_boundary_nbhd_bdry}, notice that, up to shrinking $\varepsilon$ and choosing $U_\partial \subset U$ a small enough open neighborhood of $\partial N_c$ so that $(-\varepsilon,0]\times U_\partial \subset O_c'$, the open set
    \[
        V_{O'} \coloneqq(-\varepsilon, 0]\times U_\partial \times Y
    \]
    is contained in the domain of $\Psi$, and, up to shrinking $\delta$ further, it contains $O'$. As we claimed, the form $\Omega$ writes on $V_{O'}$ as in Equation \ref{eq:symp_split}, i.e.\ $\Omega=\omegacK+\d\alpha+\d(t\alpha)$ with $\alpha=\pi_U^*\alpha_U$ and $\omegacK$ pulled back from $Y$. Moreover, $J_{O'}$ is the restriction to $O'$ of $J_{V_{O'}}\coloneqq(\Psi|_{V_{O'}})^*J_{std}=\psi^*(i\oplus i\oplus J_c)\oplus J_Y$, which is split with respect to the $(-\varepsilon,0]\times U_\partial$ and $Y$ factors, as $\Psi=(\psi\times\operatorname{id}_Y)\circ\varphi$. As for the binding case, one has that $-\d(\d t\circ J_{V_{O'}})$ is basic with respect to the $Y$-factor, and tames the induced almost complex structure on $(-\varepsilon,0]\times U_\partial$.
\end{proof}

\subsection{Pseudo-holomorphic curves near the singularities}
\label{sec:pseudo-hol_curves_near_singularities}

We will now use the special model almost complex structures near the binding and boundary of the confoliated bLob to find an explicit Bishop family of pseudo-holomorphic disks near the binding, as well as to prove the triviality of pseudo-holomorphic curves entering a small neighborhood of the binding and the boundary of the bLob.

%%%%%%%%%%%%%%%%%%%%%%%%%%%%%%%%%%%%%%%%%%%%%%%%%%%%%%%%%%%%%%%%%%%%%%

\subsubsection{Bishop family near the core's binding}
\label{sec:bishop_family}

According to \Cref{lem:normal_form_J_near_core}, we are allowed to work on the almost complex
manifold
\[
    \left(\R^2\times\R^2\times T^*B_c\times Y,\ J_{std}=i\oplus i\oplus J_c\oplus J_Y\right)
\]
in a small enough neighborhood of the image of the binding $B_c\times Y$ under $\Psi_O$.
We use this normal form to describe the desired local Bishop family of $J_{std}$-holomorphic
disks near the binding.

Let $\delta>0$ be small enough.
For all $s_0\in[1-\delta,1)$, $q_0\in B_c$ and $w_0\in L$, we have a $J_{std}$-holomorphic disk
\begin{equation}
\begin{aligned}
\label{eqn:bishop_disc_good_coordinates}
u_{s_0,q_0,w_0}
\colon D^2 &\longrightarrow \R^2\times\R^2\times T^*B_c\times Y \, \\
\quad z &\longmapsto
\left(s_0\,;\ \sqrt{1-s_0^2}\cdot z\,;\ q_0,0\,;\ w_0\right) \, ,
\end{aligned}
\end{equation}
where both $\R^2$-factors are seen here as copies of $\C$, and we use a complex coordinate. In particular, $s_0$ is to be interpreted as a real point of the first factor. The map is $J_{std}$-holomorphic because it is holomorphic in the second $\C$-factor and
constant in all the others.
Moreover, it has totally real boundary conditions by the choice $q_0\in B_c$ and $w_0\in L$.

Note that
\[
    F\circ u_{s_0,q_0,w_0}(z)\leq \tfrac12 \, ,
\]
with equality if and only if $|z|=1$. In particular the disk is contained in
$\{F\leq 1/2\}$ and touches $M_0=F^{-1}(1/2)$ exactly along its boundary.
Moreover its $u$-coordinate is constantly equal to $s_0\geq 1-\delta$, so that
$\image(u_{s_0,q_0,w_0})\subset U_O\subset\image(\Psi_O)$ for all parameters as above. Notice that $\partial u_{s_0,q_0,w_0}$ takes values
in
\[
    \left\{(s_0,0;x,y;q,0;w)\ \big\vert\ s_0^2+x^2+y^2=1,\ w\in L\right\}
    \ \subset\ N_0 \, ,
\]
i.e.\ in the core $N_0=\tilde N_0\cap\{w\in L\}$ of the model confoliated bLob. We stress that it is the core, and not $\tilde N_0$, which provides the totally real boundary condition here.

\begin{definition}
    \label{def:bishop_family}
    The \textbf{Bishop family} of pseudo-holomorphic disks \textbf{stemming from
    $B_c\times L$} is the family of $J_{std}$-holomorphic disks
    $\left(u_{s_0,q_0,w_0}\right)_{(s_0,q_0,w_0)\in[1-\delta,1)\times B_c\times L}$
    with values in
    $\image(\Psi_O)\subset \R^2\times\R^2\times T^*B_c\times Y$.
    Each disk in the Bishop family is called a \textbf{Bishop disk}.
\end{definition}

\begin{remark}
    \label{rmk:dimension_bishop_family}
    The Bishop family is parametrized by a manifold of dimension
    \[
        1+\dim B_c+\dim L=1+(k-1)+(n-k)=n \, ,
    \]
    so that, once a marked point on the boundary of the domain is added, one obtains
    $n+1=\dim N$.
\end{remark}

\begin{prop}
    \label{prop:bishop_discs_regular}
    Each Bishop disk $u_{s_0,q_0,w_0}$ is Fredholm regular, and its index is
    \[
    \ind(u_{s_0,q_0,w_0}) = n+3 \, .
    \]
\end{prop}
 \begin{proof}
    Consider a Bishop disk $u_{s_0,q_0,w_0}$.
    Fix a chart $U_{q_0}\cong\R^{k-1}$ of $B_c$ centered at $q_0$, so that $T^*B_c$ is
    identified over it with $T^*\R^{k-1}$.
    Moreover, by Remark \ref{rem:standard_near_Lagrangian}, there is a neighborhood of $(q_0,0;w_0)$, equipped with $\d\lambda_c+\omega_0$, that can be identified with a neighborhood of the zero section of
    \[
        T^*\R^{k-1}\times T^*\R^{\,n-k} = T^*\left(\R^{k-1}\times\R^{\,n-k}\right)
        = T^*\R^{\,n-1} \, ,
    \]
    with the standard symplectic form. Using this identification, we can rewrite the map with a simpler target manifold
    \begin{equation}
    \begin{aligned}
        u_{s_0,q_0,w_0}\colon D^2 &\longrightarrow D^2\times D^2\times T^*\R^{\,n-1} \, \\
        z &\longmapsto \left( s_0;\ \sqrt{1-s_0^2}\cdot z;\ 0\right) \, ,
    \end{aligned}
    \end{equation}
    where the last $0$ is the constant map to $(q_0,w_0)$ in the zero section of
    $T^*\R^{\,n-1}$, and the boundary condition is the zero section itself, times the model
    contact bLob in the two $D^2$-factors. This local description of a neighborhood of a given Bishop disk is exactly the same
    (almost complex structure and boundary condition included) as the one given in the contact
    case in \cite[Propositions 8 and 9]{Nie06}.
    In particular, the computation of the relative Maslov index and the proof of Fredholm
    regularity can then be carried out exactly as there, giving $\ind(u_{s_0,q_0,w_0})=n+3$.
\end{proof}

\subsubsection{Controlled behavior of disks near singularities}
\label{sec:behavior_near_singularities}

We describe here how the normal forms described in \Cref{sec:model_near_binding,sec:model_near_boundary} guarantee a nice control on $J$-holomorphic disks whose image reaches a neighborhood of $B$ or of $\partial N$.

\medskip

We first describe the case of the binding:

\begin{lemma}
    \label{lem:uniqueness_lemma}
    Let $O$ be the neighborhood from \Cref{sec:bishop_family} where the Bishop family of
    disks is constructed.
    Any non-constant $J$-holomorphic disk $w\colon (D^2,i)\to (W,J)$ with boundary lying on the core
    $N=N_c\times L\subset \partial W=M\subset W$ deprived of the binding $B$ and winding around $B$ once, and whose image meets $O$, coincides up to
    reparametrization in the domain with a Bishop disk $u_{s_0,q_0,w_0}$.
\end{lemma}
\begin{proof}
        We identify $O$ with $U_O$ via $\Psi_O$ of \Cref{lem:normal_form_J_near_core}, so that we
    work in
    \[
        \R^2_{u,v}\times\R^2_{x,y}\times T^*B_c\times Y \, ,
        \qquad
        U_O=F^{-1}\left([0,1/2]\right)\cap\{u\geq 1-\delta\} \, ,
    \]
    with $J=J_{std}=i\oplus i\oplus J_c\oplus J_Y$ there.
    We write $\pi_u,\pi_v,\pi_x,\pi_y$ for the four coordinate projections, $\pi_c$ for the
    projection to $T^*B_c$ and $\pi_Y$ for the one to $Y$, and set $\Sigma_O\coloneqq
    w^{-1}(U_O)$.\medskip

    First, we readily point out that the image of $\d w$ cannot be contained in $\cK_{\xi}$. Indeed, suppose $\d w$ takes values in $\cK_\xi$. Then $w$ has image in a single leaf
    $\{p\}\times Y$ and boundary in $\{p\}\times L$, and $w^*\mu\geq 0$ since $J_Y$ is tamed by
    $\mu\vert_{\{p\}\times Y}$. By
    \Cref{item:weak_exact_Lagrangian} of \Cref{def:confoliated_bLob} we get $w^*\mu=0$, hence
    $w$ is constant. We may therefore assume from now on that this is not the case. 
    \medskip

    By \Cref{lem:boundary_point_lemma_curve} applied (on a neighborhood of $U_O$ in $\R^2_{u,v}\times \R^2_{x,y}\times T^*B_c\times Y$) to the tame weakly pluri-subharmonic function
    $F$ from \eqref{eq:weakharmonicbinding},
    and since $w$ is not tangent to $\cK_\xi$, the curve $w$ is transverse to
    $\partial_+U_O=U_O\cap F^{-1}(1/2)$ and in fact we even have $w^{-1}(M_0)=\partial D^2$.
    Moreover, by Sard's theorem and after shrinking $\delta$, we may assume that $1-\delta$ is a
    regular value of $\pi_u\circ w$, so that $w$ is also transverse to
    $\partial_-U_O\subset\{u=1-\delta\}$. The set $U_O$ is a compact neighborhood of the binding $B_c\times Y$ of the bLob and, consequently, the set $\Sigma_O$ is a compact subsurface with boundary and
    corners of $D^2$, with
    \[
        \partial\Sigma_O=\partial_+\Sigma_O\sqcup\partial_-\Sigma_O \, ,
        \qquad
        \partial_+\Sigma_O=w^{-1}(\partial_+U_O)\subset \partial D^2 \, .
    \]

    Let us now show that the maximum of $\pi_u\circ w$ lies on $\partial_+\Sigma_O$. As $J$ preserves the splitting and restricts to the standard structure on the first
    $\R^2$-factor, the function $\pi_u\circ w$ is harmonic, so on each connected component
    $\Sigma'$ of $\Sigma_O$ it attains its maximum on $\partial\Sigma'$ unless it is constant.
    Now $U_O\subset\{u\geq 1-\delta\}$ with equality exactly on $\partial_-U_O$, so
    $\pi_u\circ w$ attains on $\partial_-\Sigma'$ its minimum over
    $\overline{\Sigma'}$. If its maximum were attained there, or if $\pi_u\circ w$ were
    constant, we would get $\pi_u\circ w\equiv 1-\delta$ on $\Sigma'$, contradicting the
    transversality to $\partial_-U_O$. In particular $\partial_+\Sigma'\neq\emptyset$ for every component $\Sigma'$, and the
    maximum of $\pi_u\circ w$ on $\overline{\Sigma'}$ is attained at some
    $p_0\in\partial_+\Sigma'\subset\partial D^2$.

    Since $w$ is not tangent to $\cK_{\xi}$, \Cref{thm:boundary_point_lemma} implies that either $\pi_u\circ w$ is constant, or $w$ has a non-vanishing outward derivative at $p_0$. In the second case, the fact that $\pi_v\circ w$ vanishes on $\partial_+ \Sigma_{O}$ together with the Cauchy--Riemann equation leads to a contradiction. This shows that $\pi_u\circ w$ must be constant. Arguing analogously, one deduces that $\pi_v \circ w$ is constant and thus must be equal to $0$. This shows that the image of $w$ is completely inside $O$.
    \medskip
    
    We now show that the cotangent component is constant. Consider $w_c\coloneqq\pi_c\circ w$, a $J_c$-holomorphic map into $T^*B_c$ with boundary in
    the zero section, by \eqref{eqn:bLob_in_local_model_binding}.
    The function $f_c\circ w_c$ is subharmonic and vanishes on
    $\partial D^2$, hence by the maximum principle it vanishes identically, i.e.\ $w$ has values in
    the zero section of $T^*B_c$.
    Consequently
    \[
        0\ \leq\ \int_{D^2}w_c^*\,\d\lambda_c \ =\ \int_{\partial D^2}w_c^*\lambda_c \ =\ 0 \,.
    \]
    Since $J_c$ is tamed by $\d\lambda_c$, we deduce that $w_c\equiv (q_0,0)$ for some $q_0\in B_c$.
    \medskip
    
    Lastly $w_Y\coloneqq\pi_Y\circ w$ is a $J_Y$-holomorphic disk in $(Y,\omega_Y)$ with boundary on $L$. Using that $J_Y$ is tamed by $\omega_p$ and \Cref{item:weak_exact_Lagrangian} of \Cref{def:confoliated_bLob}, we deduce that $w_Y\equiv l_0$ is constant, for some $l_0\in L$. 
    We have thus shown that only the $x$ and $y$ coordinates of $w$ are non-constant, and one concludes as in the contact case in \cite{Nie06,NieNotes} that $w$ must be a Bishop disk up to reparametrization.
\end{proof}

We now aim to show that the pseudo-holomorphic disks we consider have boundary that stays away from $\partial N$. 
More precisely, we have the following: 
\begin{lemma}
    \label{lem:boundary_is_wall}
    Let $O'$ and $J_{O'}$ be respectively the neighborhood and the almost complex structure on it given by \Cref{lem:normal_form_J_boundary}.
    Then, any $J$-holomorphic map
    \[
    w\colon (D^2,S^1,i)\to (W,N,J)
    \]
    whose image intersects $O'$ must be constant.
\end{lemma}

\begin{proof}
    To prove that $w$ is contained in $O'$ it is enough to prove that the $x$-coordinate of $w$ is constant, which prevents $w$ from ever leaving $O'$. This is argued exactly as in the proof of \Cref{lem:uniqueness_lemma}, with the function $G$ of \eqref{eq:weakharmonicboundary} in place of $F$, and with $\pi_x,\pi_y$ in place of $\pi_u,\pi_v$. Hence the image of $w$ is in $O'$, and, in particular, its boundary lies in $N\cap O'$.
    \medskip

    Consider $\hat w\coloneqq \pi\circ \pi_{E_c}\circ w$, where $\pi_{E_c}$ is the projection from $E_c\times Y$ to $E_c$, a holomorphic map into
    $\R^2\times T^*S^1$ with respect to the standard structure.
    By \eqref{eqn:bLob_in_local_model_boundary}, $\partial\hat w$ takes values in
    \[
        \pi\left(N\cap O'\right)=
        \left\{(x,y;s,\tau)\ \big\vert\ y=0,\ x=\sqrt{1-\tau^2}\right\} \, ,
    \]
    so that in particular the $y$-coordinate of $\hat w$ vanishes on $\partial D^2$.
    As $\pi_y$ is harmonic for the standard structure, $\pi_y\circ\hat w$ vanishes identically,
    and the Cauchy--Riemann equations then give that $\pi_x\circ \hat w\equiv x_0$ is constant.
    By the description of $\partial \hat w$ above, $\tau$ is then constantly equal to
    $-\sqrt{1-x_0^2}$ along $\partial D^2$; since $\pi_\tau$ is also harmonic and its maximum
    and minimum on $\partial D^2$ agree, $\pi_\tau\circ\hat w$ is constant as well, and once
    more the Cauchy--Riemann equations give that $\pi_s\circ\hat w$ is constant.
    Hence $\hat w$ is constant.

    Therefore $w$ has image in
    $\{(x_0,0;\tau_0)\}\times T^*F_0|_{s_0}\times Y $, and its boundary lies in the
    intersection of that fiber with $N\cap O'$, namely in $\{0\}\times L$, by
    \eqref{eqn:bLob_in_local_model_boundary}.
    Write $w=(w_c,w_Y)$ accordingly.
    On the one hand, $w_c$ is a $J_c$-holomorphic map into $T^*F_0$ with boundary in the zero
    section, and $J_c$ is compatible with $\d\lambda_c$. By Stokes,
    \[
        0\ \leq\ \int_{D^2}w_c^*\,\d\lambda_c\ =\ \int_{\partial D^2}w_c^*\lambda_c\ =\ 0 \, ,
    \]
    which implies that $w_c$ is constant.
    On the other hand, $w_Y$ is then a $J_Y$-holomorphic disk contained in a single slice
    $\{p\}\times Y$, with boundary on $\{p\}\times L$; as $J_Y$ is tamed by
    $\omega_0\vert_{\{p\}\times Y}$ we have $w_Y^*\omega_0\geq 0$, so
    \Cref{item:weak_exact_Lagrangian} of \Cref{def:confoliated_bLob} gives
    $w_Y^*\omega_0=0$ and hence $w_Y$ constant.
\end{proof}

%%%%%%%%%%%%%%%%%%%%%%%%%%%%%%%%%%%%%%%%%%%%%%%%%%%%%%%%%%%%%%%%%%%%%%
%%%%%%%%%%%%%%%%%%%%%%%%%%%%%%%%%%%%%%%%%%%%%%%%%%%%%%%%%%%%%%%%%%%%%%

\section{Proof of the main result}
\label{sec:proof_obstruction_fillability}

\subsection{Arranging nice boundary data}
\label{sec:arranging_boundary_data}

Consider now the hypothetical semi-positive weak symplectic filling $(W,\Omega)$ of $M$ such that $\tilde N$ is transversely-exact with respect to $\Omega$, and write $\omega$ for the restriction of $\Omega$ to $M$. 
Let also $\alpha$ and $R$ be the defining one-form of $\xi$ and the vector field fixed at the beginning of \Cref{sec:normal_forms}. Recall that $\alpha(R)=1$ and $\iota_R\d\alpha=0$, that $\alpha$ is the pullback of a contact form on $U$ near $\tilde N$, and that $R$ is the lift of the corresponding Reeb vector field there, so that $\iota_R\omegacK=0$ near $\tilde N$ as well.

Up to attaching a long enough cylindrical end, we can assume two things. First, by \Cref{lem:cotamed_for_T_big_enough}, we can assume that there is $J_\xi$ that is cotamed by $\omega|_{\xi}$ and by $\CpS_\xi$. Secondly, according to \Cref{item:normal_form_J_near_core_nbhd_bdry} of \Cref{lem:normal_form_J_near_core} and \Cref{item:normal_form_J_near_boundary_nbhd_bdry} of \Cref{lem:normal_form_J_boundary}, 
there is a normal neighborhood $(-\varepsilon,0]\times M$ of $M=\partial W$ in $W$, and neighborhoods 
$O\Subset V_B$ and $O'\Subset {V_\partial}$ of $B$ and $\partial \tilde N$ respectively, with the following properties. The open neighborhoods $V_B$ and $V_\partial$ are
of the form $(-\varepsilon,0]\times U_B\times Y$ and $(-\varepsilon,0]\times U_\partial \times Y$, $\Omega$ writes as $\omega_\cK+ \d\alpha+\d (t\alpha)=\omega+\d(t\alpha)$, and we have almost complex structures $J_{B}$ and $J_{{\partial}}$ with the following properties. 
First, they are split with respect to the $(-\varepsilon,0]\times U_B$ (respectively, $(-\varepsilon,0]\times U_\partial$) and $Y$ factors. 
Secondly, we know that the collar coordinate $t$ is such that $-\d(\d t\circ J_{B})$ and $-\d(\d t\circ J_{\partial})$ are $Y$-basic and tame the induced almost complex structure in $(-\varepsilon,0]\times U_B$ and $(-\varepsilon,0]\times U_\partial$. 
In particular, the function $t$ is tame weakly pluri-subharmonic, and even though the complex tangencies along a level set of $t$ different from $0$ are not necessarily $\xi$, they are a hyperplane field $\xi_t$ such that $\cK_{\xi_t}=\cK_{\xi}$.

We will first extend the collar identification to all of $M$ as follows. Consider a smooth identification of a neighborhood of the boundary $M=\partial W$ in $W$ with $(-\varepsilon,0]\times M$, so that it extends the identifications of ${V_B}$ and ${V_\partial}$ with respectively $(-\varepsilon,0]\times (M\cap {V_B})$ and $(-\varepsilon,0]\times (M\cap {V_\partial})$. 
This model neighborhood then comes with two natural symplectic forms, namely $\Omega$ (induced from $W$) and the symplectic form $\omega+\d(t\alpha)$ (recall that $\omega=\omega_\cK+\d\alpha$ in ${V_B}$ and ${V_\partial}$). Moreover, as the identification extends those available on ${V_B}$ and ${V_\partial}$, according to the properties of the latter, these two symplectic structures agree on $(-\varepsilon,0]\times (M\cap {V_B})$ and $(-\varepsilon,0]\times (M\cap {V_\partial})$.
Their restriction to $T(\{0\}\times M)$ also agrees, and they both induce the same orientation on $\xi=\ker\alpha$.
This implies that the linear path between both forms is in fact among non-degenerate forms at all points of $M$. Moser's path method (cf.\ again \cite[Exercise 3.4.17]{McDuffSalamon_Intro}) relative to $M$ and ${V_B},{V_\partial}$, allows us to find a neighborhood of $M$ of the form (up to shrinking $\varepsilon$)
$$(-\varepsilon,0]\times M, \qquad \text{with} \qquad \Omega=\omega+\d(t\alpha),$$
that coincides with the previously existing identifications on $(-\varepsilon,0]\times(M\cap {V_B})$ and $(-\varepsilon,0]\times(M\cap {V_\partial})$.
\medskip

Following the contact case as treated in \cite{NieNotes} (with due adaptations), we will now extend the almost complex structures $J_B$ and $J_\partial$ on $O\subset V_B$ and $O'\subset V_\partial$ to all of the collar neighborhood $(-\varepsilon,0]\times M$ in a very precise way that will make the function $\psi\coloneqq e^{Ct}-1$ of the collar coordinate, for $C>0$ large enough, tame weakly pluri-subharmonic. For this, let now $K_B\subset {V_B}$ and $K_\partial\subset {V_\partial}$ be open neighborhoods of respectively $O$ and $O'$ in $W$, with closures contained in $V_B$ and $V_\partial$ respectively. Fix also a complex structure $J_{\xi}$ on $\xi$ in all $M$ cotamed by $\CpS_\xi$ and $\omega\vert_\xi$, which exists by assumption. Without loss of generality, up to using Proposition \ref{prop:space_cotaming_Js_contractible} and \Cref{lem:cotamed_for_T_big_enough} (i.e.\ attaching a longer cylindrical end), and up to slightly shrinking $V_B$ and $V_\partial$, we can assume that $J_{\xi}$ is split on the split neighborhood $V$ of $\tilde N$, which contains $V_B\cup V_\partial$.
Recall that $J_B$ and $J_\partial$ are split with respect to a symplectization of a contact hypersurface $(-\varepsilon,0]\times U_B$ (and $(-\varepsilon,0]\times U_\partial$) and $Y$. Hence, up to shrinking all the neighborhoods whenever necessary to keep the taming conditions (which are open in the split factor where the almost complex structure is modified), and using the fact that the complex tangencies of $J_B$ and $J_\partial$ along $M$ are $\xi$, we can use the standard contractibility of tame almost complex structures on the contact factor to find an almost complex structure $J_{mod}$ in $V_B\cup V_\partial$ with the following properties:
\begin{itemize}
    \item[-] $J_{mod}$ is tamed by $\Omega$ and split (in the sense we have been using), with a constant $Y$-factor,
    \item[-] $J_{mod}|_{O}=J_B$ and $J_{mod}|_{O'}=J_\partial$,
    \item[-] along $\{t=0\}$, the complex tangencies of $J_{mod}$ are $\xi$,
    \item[-] in $V_B\setminus K_B$ and $V_\partial \setminus K_\partial$, $J_{mod}$ preserves $\xi$ on each level set of $t$, with $J_{mod}|_{\xi}=J_{\xi}$, and $J_{mod}((1+t)\partial_t)=R$.
\end{itemize}

In particular, by the argument of \cite[Page 200]{NieNotes} applied on the contact factor, for $C$ large enough and up to shrinking $\varepsilon$, the two-form $-\d(\d\psi\circ J_{mod})$ is basic with respect to the $Y$-factor and tames $J_{mod}$ on the $(-\varepsilon,0]\times U_B$ (and $(-\varepsilon,0]\times U_\partial$) factor. 
Away from $V_B\cup V_\partial$, $J_{mod}$ is extended by the almost complex structure as in \Cref{lem:t_is_weak_plurisubharm} to an almost complex structure $J_1$, tamed by $\Omega$, in the whole collar of $M$. The collar coordinate $t$ is then tame weakly pluri-subharmonic away from $V_B\cup V_\partial$, and hence so is $\psi$ for every $C>0$. Indeed, for any function $\phi$ with $\phi'>0$ and $\phi''\geq0$ one has $-\d\d^\C\phi(t)=\phi'(t)\,(-\d\d^\C t)+\phi''(t)\,\d t\wedge(-\d^\C t)$, where the last term is non-negative on complex lines, as $\d t\wedge(-\d^\C t)(v,Jv)=\d t(v)^2+\d t(Jv)^2$,
and vanishes on $\cK_\xi\subset\ker \d t\cap\ker\d^\C t$. 
Moreover, $\phi$ and $t$ have the same level sets and the same induced hyperplane fields on them. 

Summarizing, the function $\psi$ is tame weakly pluri-subharmonic with respect to $J_1$ in the collar. Its level sets are those of $t$; the complex tangencies of the level $\{\psi=c\}$ are $\xi$ for $c=0$ and away from $V_B\cup V_\partial$, while on $V_B\cup V_\partial$ they are a possibly different hyperplane field $\xi_c$, which however still satisfies $\cK_{\xi_c}=\cK_\xi$, given by the $Y$-leaves, as $-\d(\d\psi\circ J_1)$ is basic with respect to the $Y$-factor and tames $J_1$ on the contact factor there.

The conclusions of \Cref{lem:uniqueness_lemma,lem:boundary_is_wall}, w.r.t.\, $O$ and $O'$ hold for the $J_1$ we just constructed.

Lastly, notice that $J_1$ with all the desired properties could have been constructed directly on a slightly enlarged collar neighborhood of the form $(-\varepsilon,\delta)\times M$ for a small enough $\delta$. The conclusions of \Cref{lem:maximum_principle_curve,lem:boundary_point_lemma_curve} hold, with respect to $\psi$, for any $J_1$-holomorphic curve intersecting this collar.

%%%%%%%%%%%%%%%%%%%%%%%%%%%%%%%%%%%%%%%%%%%%%%%%%%%%%%%%%%%%%%%%%%%%%%
%%%%%%%%%%%%%%%%%%%%%%%%%%%%%%%%%%%%%%%%%%%%%%%%%%%%%%%%%%%%%%%%%%%%%%

\subsection{Regularity with cylindrical cotamed almost complex structures}
\label{sec:regularity_cotamed}

In this section we discuss how to achieve regularity for certain pseudo-holomorphic curves via 
generic perturbations of the almost complex structure. 

The standard generic 
result from e.g.\, \cite[Chapter 3]{McDuffSalamon_Intro} does not unfortunately immediately apply to our case.
Indeed, we would rather like to perturb among almost complex structures which are cylindrical and cotamed 
near the c-symplectic confoliated boundary. 
In the contact case, the fact that this can be done without issues ultimately boils down to the
fact that perturbations of pluri-subharmonic functions are still pluri-subharmonic.
In our
confoliated setup, we deal with weakly pluri-subharmonic functions, and perturbations might unfortunately destroy
this property.

However, not all is lost, and regularity for a certain class of curves can still be achieved by 
generic perturbations in the class of cylindrical and cotamed (possibly with respect to other c-symplectic confoliations) almost complex structures.
More precisely, we will use the following result, which we formulate in abstract terms (and specialize to our situation right below).

\begin{theorem}
    \label{thm:generic_regularity_cotamed}
    Let $(W,\Omega)$ be a (possibly open) symplectic manifold.
    Suppose there is an open subset $U\subset W$ of the form $U= Z \times Z'$, where $Z'$ is a closed manifold (and $Z$ is possibly open), in such a way that
    \[  
        \Omega\vert_U = \omega+\omega'
    \]
    with $\omega$ and $\omega'$ symplectic forms respectively on $Z$ and $Z'$.
    Let also $\psi$ be a regular function defined on an open subset of $U$, only depending on the coordinates on $Z$, and $J_{model}$ be an almost complex structure defined on a closed subset $K$ of $W$.
    Lastly, we then denote by $\cJ$ the space of $\Omega$-tamed almost complex structures $J$ on $W$ satisfying the following:
    \begin{enumerate}
        \item coincide with $J_{model}$ on $K$;
        \item\label{item:thm_regularity_J_preserves_first_factor} on $U$, preserve the tangent spaces $TZ'$ to the second factor; 
        \item\label{item:thm_regularity__cotamed_J} on $U$, $\d\psi\circ J$ is basic with respect to the foliation given by the tangent spaces $TZ'$ to the second factor, and $-\d(\d\psi\circ J)\vert_{TZ}$ tames the induced $J$ on $TZ$.
    \end{enumerate}
    Then, there is a residual subset $\cJ_{reg}\subset \cJ$ such that the following property holds.
    For all $J\in \cJ_{reg}$, any $J$-holomorphic curves $u\colon D^2\to W$ satisfying at least one of the following properties is Fredholm regular:
    \begin{enumerate}[i.]
        \item\label{item:thm_regularity__usual_curves} $u$ has an injective point in $W\setminus (K\cup\overline{U})$;
        \item\label{item:thm_regularity__unusual_curves} $u$ has an injective point $z_0$ such that $\d_{z_0}u$ has image not contained in $TZ'$ and that $u(z_0)$ is contained in $U\setminus K$, in such a way that $z_0$ is the unique point that is mapped to $\{q_0\}\times Z'\subset U=Z\times Z'$, where $u(z_0)=(q_0,q_0')\in Z\times Z'=U$. 
    \end{enumerate}
\end{theorem}

Before dealing with the proof, we first describe how 
\Cref{thm:generic_regularity_cotamed} can be applied to our situation. 
The open subset (named $U$ in the statement, but we do not name it here to keep the notation  $U$ for the contact factor of a split neighborhood of $\tilde N$) of the symplectic manifold $W$ is given by $(-\varepsilon,\delta)\times V$ (recall that the compact filling $W$ had been slightly enlarged by adding an open $\delta$-neighborhood of the boundary), where $V=U\times Y$ is a split neighborhood of $\tilde N$, containing also the compact neighborhood $O\cup O'$ of $B$ and $\partial \tilde N$; the splitting $Z\times Z'$ is given by $Z=(-\varepsilon,\delta)\times U$ and $Z'=Y$, with $\psi$ the tame weakly pluri-subharmonic function from Section \ref{sec:arranging_boundary_data}; lastly, the closed set $K$ in the statement is the union of the compact neighborhood $O\cup O'$ of $B$ and $\partial \tilde N$ and $[-\varepsilon,\delta)\times(M\setminus V)$, where $J_{model}$ is $J_1$ on $O\cup O' \cup [-\varepsilon,\delta)\times \left(M\setminus V\right)$. We will then work with a $J\in\cJ_{reg}$ as in \Cref{thm:generic_regularity_cotamed} in the following sections.
The important properties of such a $J$ are that the weak maximum principle is satisfied in $(-\varepsilon,\delta)\times M$, and that regularity of curves holds for disks which satisfy Items \ref{item:thm_regularity__usual_curves} or \ref{item:thm_regularity__unusual_curves} of \Cref{thm:generic_regularity_cotamed} when the relevant injective point is mapped to $\left((-\varepsilon,\delta)\times V\right)\setminus (O\cup O')$.

\medskip

We are now ready to explain the proof of \Cref{thm:generic_regularity_cotamed}.
In fact, the proof follows from the same universal moduli space argument described in great detail in the classical case of \cite[Proposition 3.2.1 and Remark 3.2.3]{McDSalBook}. 
We hence describe only the required modifications.

\begin{proof}[Sketch of proof of \Cref{thm:generic_regularity_cotamed}]
    Universal regularity for curves as in \Cref{item:thm_regularity__usual_curves} can be argued exactly as in \cite{McDSalBook}.
    As far as curves as in \Cref{item:thm_regularity__unusual_curves} are concerned, the additional requirements in \Cref{item:thm_regularity__cotamed_J} restrict the space of available linear deformations of the almost complex structures.
    This results in a loss of regularity for certain curves, but not all. Here are additional details.

    Following the argument (and notations, with slight adaptations) in \cite[Proposition 3.2.1]{McDSalBook}, the key property to show is that, given any $J\in \cJ$, any curve $u$ and any injective point $z_0$ as in \Cref{item:thm_regularity__unusual_curves}, and any non-zero section $\eta\in L^q(\Lambda^{0,1}(\Sigma)\otimes_J u^*TW)$, there is $Y\in T_J\cJ$ such that 
    \begin{equation}
    \label{eqn:section_eta}
        \langle \eta, Y(u)\circ \d u\circ j \rangle >0 \, .      
    \end{equation}

    In order to understand whether such a $Y$ exists, let us first describe $T_J\cJ$ better.
    First, as done in \cite{McDSalBook}, we identify $Y\in T_J\cJ$ with a section of the bundle of endomorphisms of $TW$ anticommuting with $J$.
    Now, note that the requirement in \Cref{item:thm_regularity__cotamed_J} is an open condition, hence this condition does not restrict the set of allowed tangent vectors $Y$'s (provided that the $J$ satisfies it, which is true if it belongs to $\cJ$).
    Moreover, the requirement that $\d\psi\circ J$ is basic with respect to the foliation given by the $Z'$ factor depends only on how $J$ (and hence the infinitesimal change $Y$) behaves in $U$, as $U$ is saturated with respect to leaves of such foliations.
    In other words, we can restrict to the ambient symplectic manifold $U=Z\times Z'$ without loss of generality.
    
    According to \Cref{item:thm_regularity_J_preserves_first_factor},
    $J$ is in lower diagonal block form with respect to the identification $TU=TZ\oplus TZ'$, and up to changing this identification to another of the form $E\oplus TZ'$ with a $J$-invariant complement bundle $E$ to $TZ'$ in $TU$ invariant under the holonomy of the foliation given by the second factor, we can arrange that $J$ is in diagonal block form $J=J_0\oplus J_0'$ with respect to the splitting $TU=E\oplus TZ'$.

    Recall now from e.g.\ \cite[Proposition 2.6.4]{McDuffSalamon_Intro} that there is an explicit identification of a neighborhood of $J$ in the space of all almost complex structures tamed by $\Omega$ with a certain space of endomorphisms 
    \[
        \cY := \{Y \in \Gamma(End(TW)\to W) \mid Y J + J Y =0 , \;  \Vert Y_p\Vert<1\; \forall p\in W \,\}\, ,
    \]
    where $\Vert Y_p \Vert= \sup_{v\in T_pW}\frac{\vert Y_p(v)\vert}{\vert v\vert}$ is the operator norm.
    More precisely, the identification is given by 
    \begin{equation}
    \label{eqn:identification_tangent_alm_cplx_str}
        \Phi(Y) = J(\Id + Y)^{-1}(\Id-Y) \, .
    \end{equation}
    Note that $\cY$ is an open subset of a linear space, hence $T_0\cY$ is naturally identified with $\{Y\in \Gamma(End(TW))\mid YJ+JY=0 \}$ itself.
    We now want to better understand which $Y\in T_0\cY$ are sent by $\d\Phi$ to $T_J\cJ$.
    (We will not determine the complete preimage $\d\Phi^{-1}(T_J\cJ)$, but we will describe a sufficient condition on $Y\in T_0\cY$ so that it belongs to such a preimage. This will be enough for our purposes.)

    Then, we write $Y$ in block matrix form as well, and the condition $J (TZ')=TZ'$ immediately gives that $Y$ must preserve $TZ'$ as well, hence, we can write it in blocks as
    \[
        Y = \begin{pmatrix}
            A & 0 \\
            C & D
        \end{pmatrix} \, .
    \]
    Moreover, anticommutativity with respect to $J$ forces the following conditions:
    \[
        AJ_0 + J_0A=0 \, , 
        \quad
        CJ_0 + J_0'C = 0 \, ,
        \quad
        DJ_0'+J_0'D=0\, .
    \]
    (Note that we do not ask $J$ to be $\Omega$-compatible, so we do not have requirements on symmetry of $A$ and $D$. Moreover, symmetry for $C$ does not a priori make sense, as $Z$ and $Z'$ are allowed to have different dimensions.)
    
    Note that the condition on $\d\psi\circ J$ being basic with respect to the foliation given by $TZ'$ means that $A$ and $C$ cannot be any pair of sections of $End(E)$ and $Hom(E,TZ')$ respectively. 
    However, if these are chosen to be independent of the $Z'$-coordinates (i.e.\ induced as a pullback under the natural projection $Z\times Z'\to Z$) then certainly we will have $\d\Phi(Y)\in T_J\cJ$. 
    Indeed, the latter holds if $\Phi(tY)\in \cJ$ for $t\geq 0$ small enough, but this holds by the explicit formula in  \eqref{eqn:identification_tangent_alm_cplx_str}, as it can be seen as follows. 
    First, $\Phi(tY)$ is again lower block-triangular, and its $E$-block is only a function of $A$ (in particular, the $E$-block is independent of $Z'$-coordinates).
    Since $\d\psi$ annihilates $TZ'$ and is basic, therefore $\d\psi\circ \Phi(tY)$ only depends on the $E$-block and is basic. We have that $TZ'$ is preserved by $\Phi(tY)$ since $J$, $Y$ and $\Id$ preserve it, and lastly the condition $-\d(\d\psi\circ J)\vert_{TZ}$ taming the induced $J$ on $TZ$ is open hence holds for $t\geq 0$ very small. (Note that the condition that $\Phi(Y)$ coincides with $J_{model}$ on $K$ might a priori not hold in general; however, in a moment we will choose a specific $Y$ that has support away from $K$.)

    Now, it follows from \Cref{lem:linear_algebra_lemma} below that one can find a $Y(z_0)$ with $Y(z_0)\circ \d_{z_0}u = \eta(z_0)$, with moreover $D=0$ in the above block form.
    The remaining subtle point is now to extend $Y(z_0)$ to a section $Y$ satisfying \eqref{eqn:section_eta}.
    In order to do so, the classical argument, e.g.\ in \cite[Proposition 3.2.1]{McDSalBook}, requires extending $Y(z_0)$ in a small enough neighborhood of $z_0$ and then multiplying this extension by a cut-off function $\beta$ supported sufficiently near $z_0$ in such a way that \eqref{eqn:section_eta} is guaranteed.
    In our case, this cannot be done as $Y$ needs to be invariant under the holonomy of the foliation given by the second factor $Z'$ in $U$, as pointed out above.
    In other words, the cut-off function $\beta$ also needs to be invariant under such holonomy.

    However, the assumption in \Cref{item:thm_regularity__unusual_curves} that $z_0$ is the unique point mapped to $\{q_0\}\times Z'$, where $u(z_0)=(q_0,q_0')$, exactly guarantees that such a holonomy invariant cut-off function $\beta$ can be found.
    We thus have our desired $Y\in T_J\cJ$ satisfying \eqref{eqn:section_eta}.
    
    The rest of the argument follows line by line \cite[Proposition 3.2.1]{McDSalBook}, and we hence omit the details.
\end{proof}

As in the classical case treated e.g.\ in \cite{McDSalBook}, the following linear algebra lemma guarantees that the allowed movements of $J$'s are sufficient in the above proof.
\begin{lemma}
    \label{lem:linear_algebra_lemma}
    Let $\xi\in \R^{2n}$ and $\eta=\zeta+\zeta'\in \R^{2N}=\R^{2n}\oplus \R^{2n'}$ be non-zero vectors.
    Let also $J_0$ and $J_0'$ be respectively the standard complex structures on $\R^{2n}$ and $\R^{2n'}$ respectively.
    Then, there are $A\colon \R^{2n}\to \R^{2n}$ and $C\colon \R^{2n}\to\R^{2n'}$ such that 
    \[
        AJ_0+J_0A=0 \, ,
        \quad 
        C J_0+J_0' C = 0 \, ,
        \quad
        A\xi = \zeta \, ,
        \quad 
        C\xi = \zeta' \, .
    \]
\end{lemma}

\begin{proof}
    The following explicit formulas (following the ones in \cite[Lemma 3.2.2]{McDSalBook}) give the desired $A$ and $C$:
    \[
        A := \frac{1}{\vert \xi\vert ^2} \left( 
            \zeta \xi^T +J_0\zeta \xi^T J_0
        \right) \, ,
        \quad 
        C:= \frac{1}{\vert \xi\vert ^2} \left( 
            \zeta' \xi^T +J_0'\zeta' \xi^T J_0
        \right) \, .
    \]
    (Note that $J_0\xi$ is orthogonal to $\xi$ with respect to the standard scalar product on $\R^{2n}$, hence $\xi^TJ_0\xi =0$.)
\end{proof}

%%%%%%%%%%%%%%%%%%%%%%%%%%%%%%%%%%%%%%%%%%%%%%%%%%%%%%%%%%%%%%%%%%%%%%

\subsection{The moduli space of $J$-holomorphic disks}
\label{sec:mod_space}

Let $C$ be a smaller compact copy of $O\cup O'$ in $W$, that intersects $M=\partial W$ in compact neighborhoods of the binding of $N$, which we denote by $B$, and $\partial N$ in $M$, and so that it is union of two connected components, each being a compact domain of the type described in \Cref{item:normal_form_J_near_core_image} of \Cref{lem:normal_form_J_near_core} and in \Cref{item:normal_form_J_near_boundary_image} of \Cref{lem:normal_form_J_boundary}.

\medskip

Let now $\widehat\cM_*$ be the moduli space of unparametrized boundary-marked $J$-holomorphic disks, i.e.\ the set of equivalence classes of pairs $(u,z)$, where $u\colon (D^2,i)\to (W,J)$ is a $J$-holomorphic map with boundary on the core $N$ of the confoliated bLob $\tilde N\subset M=\partial W$, and $z\in \partial D^2$ is a point in the boundary of the domain of $u$, under the equivalence relation $\sim$ given by $(u,z)\sim(u',z')$ if there is a biholomorphism $\varphi$ of the closed unit disk such that  $u'=u\circ \varphi^{-1}$ and $z'=\varphi(z)$.

Let finally  $\cM_*$ be the connected component\footnote{We will not consider the other connected components of $\widehat\cM_*$, which could potentially be non-empty.} of $\widehat\cM_*$ containing classes of the form $[u,z]$ where $u$ is a model Bishop disk (as defined in \Cref{sec:bishop_family}) near the binding $B$ of the confoliated bLob.
We also point out that there is a well-defined smooth \textbf{evaluation map}
\begin{equation}
    \label{eqn:evaluation_map}
    \ev\colon \cM_*\to N\, ,
    \quad 
    [u,z]\mapsto u(z) \, .
\end{equation}
This map will be useful later in the proof of the main result, and is the reason why we work with curves with a marked point on the boundary of the domain.

\subsubsection{Regularity of curves in $\cM_*$} 

We first establish two technical lemmas about the behavior of curves we are interested in. For the following lemma, observe that in a split neighborhood $V=U\times Y$ of a confoliated bLob $\tilde N$, the confoliation $\xi$ intersects the core minus the binding of the core, i.e. $N\setminus B$, along a foliation $\xi_N=\xi_{N_c}\oplus TL$. Here $\xi_N$ denotes the intersection of $\xi$ with $TN$, and $\xi_{N_c}$ spans the pages of the contact bLob $N_c$.

\begin{lemma}
    \label{lem:boundary_disc_transverse_xiN}
    Let $u$ be a $J$-holomorphic disk in $W$ with boundary on the totally real part $N\setminus B$ of confoliated bLob's core $N\subset  M=\partial W$.
    Then, either the boundary $\gamma$ of $u$ is transverse at every point to leaves of $\xi_N$, or $\gamma$ and $u$ are tangent to $\cK_\xi|_{TN}$ and $\cK_\xi$ respectively. Moreover, in the first case, if $[u,z]\in \cM_*$ for some $z\in \partial D^2$, $\gamma$ must intersect each of the leaves exactly once; in the second case, $u$ must be constant.
\end{lemma}

\begin{proof}
    We start by proving the first claim.
    If $\gamma$ is transverse at every point to the leaves of $\xi_N$, there is nothing to prove. 
    Assume then that $\gamma$ intersects a leaf of $\xi_N$ in a non-transverse way at a certain point, i.e.\ that $\d_p \gamma$ has image contained in $(\xi_N)_{\gamma(p)}$ for some point $p\in \partial D^2$.
    By $J$-holomorphicity of $u$ and our choice of $J$ for which $\xi$ is $J$-invariant along $M$, it follows that $\d_p u$ also has image in $\xi$.
    However, according to the choice of $J$ (cf.\ \Cref{sec:arranging_boundary_data}), \Cref{lem:boundary_point_lemma_curve} applies with the weakly pluri-subharmonic function $\psi   $ in $(-\varepsilon,\delta)\times M$, and hence we deduce that $u$ is tangent to $\cK_\xi$ itself, and in particular $\gamma$ is tangent to $\cK_\xi|_{TN}$.
    This concludes the proof of the first part of the lemma.

    \medskip

    Let us now move to the last claims. 
    Assume first that $\gamma$ and $u$ are tangent to $\cK_\xi|_{TN}$ and $\cK_\xi$. Since the leaves of $\cK_\xi|_{TN}$ are weakly exact Lagrangians in the leaves of $\cK_{\xi}$ (Item \ref{item:weak_exact_Lagrangian} in \Cref{def:confoliated_bLob}),  we deduce that $u$ is constant.
    Assume now that $\gamma$ is positively transverse to $\xi_N$ at every point, and that $[u,z]\in \cM_*$ for some $z\in \partial D^2$. This means that the winding number of $\gamma$ around $B$ in $N$ must be strictly positive; now, the fact that $\cM_*$ is path-connected by definition and that it contains Bishop disks (equipped with a point in the boundary of the domain) implies directly that such winding number must be $1$, hence every leaf of the open book foliation given by $\xi_N$ can be intersected at most once, and so exactly once.
    This concludes the proof.
\end{proof}
{As an immediate consequence of \Cref{lem:boundary_disc_transverse_xiN} and \Cref{lem:boundary_point_lemma_curve} applied to the function $\psi$ on $(-\varepsilon , \delta)\times M$, we get the following:
\begin{lemma}
    \label{lem:discs_must_go_in_filling_near_boundary}   
    Let $[u,z]\in \cM_*$ be non-constant. Then, for all $p\in \partial D^2$ and $v\in T_p D^2$ pointing outwards, one has $\d t(\d_p u(v))>0$, where $t$ is the first coordinate on the collar $(-\varepsilon,\delta)_t\times M$.
\end{lemma}
}

From the previous Lemmas, we deduce the following.

\begin{lemma}
    \label{lem:injective_points_near_boundary}
    Let $[u,z]\in \cM_*$, with $u$ a \emph{non-constant} $J$-holomorphic disk. Then, there is a neighborhood of the boundary $\gamma$ where $u$ is an embedding, and such that this neighborhood intersects each $\{t\}\times \{p\}\times Y$ at most once, for all $t$ small enough and any $p\in U$.
\end{lemma}

\begin{proof}
    According to \Cref{lem:boundary_disc_transverse_xiN}, $\gamma$ is transverse to the foliation $\xi_N$ on $N$ and intersects each of its leaves exactly once.
    In particular, the map $u$ restricted to $\partial D^2$ is injective. 
    {Moreover, by the same argument as in the proof of \Cref{lem:boundary_disc_transverse_xiN} above, we also have $\d u\neq 0$ at all points of $\partial D^2$.} 
    Thus, $u$ is an embedding near any point of $\partial D^2$. 
    By compactness of $D^2$, this implies that there is a sufficiently small neighborhood $H$ of $\partial D^2$ where $u$ is an embedding. Since $H$ can be taken to be transverse to the $t$-slices by Lemma \ref{lem:discs_must_go_in_filling_near_boundary}, we can choose $H$ so that it only intersects the $t$-slices for $t\in(t_0,0]$, and does so along a circle. Finally, Lemma \ref{lem:boundary_disc_transverse_xiN} implies that the boundary of $D^2$ intersects transversely each set $\{0\}\times L_c \times Y$ (where $L_c$ is a page of the open book foliation in $N_c$) at most once. In particular, the boundary is transverse to each $\{0\}\times \{p\}\times Y$, with $p\in U$, and intersects it at most once. Since this property is $C^1$-open on circles, we deduce that the circle $\gamma_{t^*}=H\cap \{t=t^*\}$ for $t^*\in (t_0,0]$ satisfies that it is transverse to each $\{t^*\}\times \{p\}\times Y$, with $p\in U$, and it intersects it at most once.
\end{proof}

We are now ready to prove that the disks in the moduli space satisfy one of the assumptions for Theorem \ref{thm:generic_regularity_cotamed} to apply.

\begin{lemma}\label{lem:regularity_condition_satisfied}
    Let $[u,z]\in \mathcal{M}_*$ be an element with $u$ a non-constant $J$-holomorphic disk. Then $u$ satisfies \Cref{item:thm_regularity__unusual_curves} from \Cref{thm:generic_regularity_cotamed}.
\end{lemma}

\begin{proof}
    Let $u$ be a $J$-holomorphic disk in $W$ with boundary on $N\setminus (B\cup \partial N)$. 
    According to the choice of $J\in \cJ_{reg}$ as in \Cref{thm:generic_regularity_cotamed}, we know that the function $\psi$ is weakly pluri-subharmonic on an open neighborhood $(-\varepsilon,\delta)\times M$ of the boundary component $\{0\}\times M = M$ of $W$.
    Let $H$ be the neighborhood of the boundary of the disk provided by Lemma \ref{lem:injective_points_near_boundary}. There is an injective point $z_1\in H$ with image in some level set $\{t_1\}\times Op(N)\subset \{t_1\}\times M$ for a certain $t_0<t_1<0$.

    Writing $u(z_1)=(t_1,p_1,q_1)$, we need to show that $u^{-1}(\{t_1\}\times \{p_1\}\times Y)=\{z_1\}$.

    By the properties of $H$ (understood now as a neighborhood of the boundary of $D^2$ in the source space of $u$), we know that $H\cap u^{-1}(\{t_1\}\times \{p_1\}\times Y)=\{z_1\}$. Thus, if there is some other point $z_2\in u^{-1}(\{t_1\}\times \{p_1\}\times Y),$ it must belong to $D^2\setminus H$. But this implies that $\psi\circ u|_{D^2\setminus H}$ attains a maximum in the interior of $D^2$. By \Cref{lem:maximum_principle_curve}, this would imply that $u$ is tangent to $\{t=t_1\}$, a contradiction.
\end{proof}

We are now ready to prove regularity.

\begin{prop}
    \label{prop:regularity_and_index}
    Let $[u,z]\in \cM_*$.
    Then, $u$ is Fredholm regular, with Fredholm index 
    \[
    \ind(u)=  n+3\, .
    \]
    In particular, $\cM_*$ is a smooth manifold of dimension $n+1$ around $[u,z]$.
\end{prop}

\begin{proof}
    We have already proven that Bishop disks are regular and calculated their index in \Cref{prop:bishop_discs_regular}. 
    Let us then assume that $u$ is not a Bishop disk; according to \Cref{lem:uniqueness_lemma,lem:boundary_is_wall}, this means that the image of $u$ does not intersect the neighborhood $C$ of the binding of $N$ where the Bishop family is defined and where the uniqueness property holds. 
    (Recall the definition of $C$ from \Cref{sec:mod_space}. This is in fact a neighborhood of $\partial N$ too, but there is no non-constant curve intersecting the neighborhood $C\cap O'$ of $\partial N$, where $O'$ is as in \Cref{lem:boundary_is_wall}.) 
    It follows that Lemma \ref{lem:regularity_condition_satisfied} applies, and thus by Theorem \ref{thm:generic_regularity_cotamed} the disk $u$ is regular.

    \smallskip

    As far as the index computation is concerned, this simply follows from the fact that the index is locally constant on each connected component, and $n+3$ was exactly the Fredholm index of the Bishop disks according to \Cref{prop:bishop_discs_regular}.

    \smallskip
    
    At this point, the fact that $\cM_*$ is a smooth manifold of dimension $n+1$ follows from the classical arguments e.g.\ from \cite[Theorem 4.8]{WenBook}.
    This concludes the proof.
\end{proof}

\subsubsection{Compactification of the moduli space $\cM_*$} 
\label{sec:compactification_moduli_space}

Recall that the \textbf{symplectic energy} of a $J$-holomorphic $u\colon (D^2,i) \to (W,J)$ is defined to be 
\begin{equation*}
    E(u):= \frac{1}{2}\int_{D^2} |\d u|_g^2
\end{equation*}
where $|\cdot|_g$ denotes the norm induced by\footnote{Recall that in our case $J$ is only tamed by $\Omega$, and not necessarily compatible.} $g:= \frac{1}{2}(\Omega(\cdot, J\cdot) - \Omega(J\cdot,\cdot))$ and any auxiliary Riemannian metric on $D^2$ on the space of morphisms $T D^2 \to TW$.

This quantity plays a central role in compactifying the moduli spaces under consideration: Gromov's compactness theorem guarantees that moduli spaces of pseudo-holomorphic disks with compact totally real boundary conditions and subject to uniform energy bounds can be compactified by adding nodal curves (cf.\ \cite[Theorem 4.4]{Fra08} for an explicit proof in the case of a marked point at the boundary).

The first step is then to bound the energy of curves uniformly in $\cM_*$.

\begin{prop}
    \label{prop:energy_bounds}
    There is a real $K>0$ such that for every $[u,z]\in \cM_*$, one has $E(u)\leq K$.
\end{prop}

Recall that 
$E(u)=\int_{D^2}u^*\Omega$ for any $J$-holomorphic map $u$ (cf.\ for instance \cite[Lemma 2.2.1]{McDSalBook}).
In order to establish uniform energy bounds, we then have to control the integral of $u^*\Omega$ over $D^2$.

\begin{proof}
    Let $[u,z]$ be an arbitrary element in $\cM_*$.
    The fact that $\cM_*$ is path-connected (by definition) and contains the Bishop family means that one can find a homotopy $(u_s)_{s\in [0,1]}$ with $u_0$ a constant Bishop disk and $u_1$ a reparametrization of $u$.
    This can be reinterpreted as a map $U\colon [0,1]\times D^2 \to W$, and by closedness of $\Omega$ and Stokes one has 
    \[
    0=\int_{[0,1]\times D^2} \d U^*\Omega = - \int_{\{0\}\times D^2} U^*\Omega + \int_{\{1\}\times D^2} U^*\Omega - \int_{[0,1]\times \partial D^2} U^*\Omega  \, ,
    \]
    i.e.\ $E(u) = \int_{D^2} u^*\Omega =  \int_{[0,1]\times \partial D^2} U^*\Omega$ as $u_0$ is a constant map.
    
    Now, note that $[0,1]\times \partial D^2$ is mapped to $W$ via the union of $\partial u_s$'s, all of which have boundary on $N$, where $\Omega = \d\alpha\vert_N + \omegacK\vert_N = \d\alpha\vert_N$ as $\omegacK$ vanishes on $N$ according to \Cref{item:weak_exact_Lagrangian} of \Cref{def:confoliated_bLob} (and \Cref{lem:confol_bLob_can_assume_omega_only_char_fol}).
    Hence, Stokes' theorem implies that
    \[
    E(u) =  \int_{[0,1]\times \partial D^2} U^*\d\alpha =  \int_{\partial D^2} u^*\alpha \, .
    \]

    Let now $\theta \colon N\setminus B \to S^1$ be a fibration defining the open book foliation on $N$.
    Because $\xi_N$ and the foliation given by the fibers of $\theta$ coincide, there is a function $f\colon N\setminus B \to \R_{\geq 0}$ that is bounded, vanishing exactly in $\partial N$ and such that $\alpha\vert_N = f \d\theta$.
    In particular, we get 
    \[
    E(u) = \int_{\partial D^2} f\circ u \cdot  u^*\d\theta  \leq \max_{p\in N}\vert f(p)\vert \int_{\partial D^2} u^*\d\theta \leq 2\pi \max_{p\in N} \vert f(p)\vert  \, ,
    \]
    where the last inequality follows from \Cref{lem:boundary_disc_transverse_xiN}.
\end{proof}

\bigskip

Now that we know that there is a global energy bound for marked curves in $\cM_*$, we would like to apply Gromov's compactness theorem as outlined above in order to get a compactification $\cMline_*$ of $\cM_*$.
However, this cannot be done directly for the following reason: if we consider the whole core $N$ of the confoliated bLob $\tilde N$ as a boundary condition, this fails to be totally real at the binding $B$; if we remove the binding, $N\setminus B$ is indeed totally real, but is no longer compact.
This issue can be circumvented thanks to our uniqueness lemma near the binding, as in the contact case in \cite{Nie06}, as follows.

According to \Cref{lem:uniqueness_lemma}, we know that any marked curve $[u,z]\in \cM_*$ whose boundary touches the open neighborhood $O$ (defined in the lemma) of $B$ in $N$ actually is a Bishop disk. 
If we denote $\cB_*$ the subset of $\cM_*$ made of marked Bishop disks with boundary in a compactly contained smaller copy $\tilde O$ of $O$, we then have that every curve in $\cM_*\setminus \cB_*$ has boundary on $N\setminus \tilde O$, which is both totally real and compact.

Then, since no curve can touch $\partial N$ according to our choice of $J$ and \Cref{lem:boundary_is_wall}, and since \Cref{prop:energy_bounds} gives uniform energy bounds for curves in $\cM_*\setminus \cB_*$, Gromov's compactness theorem says that there is a compact topological space $\overline{\cM_*\setminus \cB_*}$ such that $(\overline{\cM_*\setminus \cB_*})\setminus (\cM_*\setminus \cB_*)$ consists of marked nodal disks made of a main disk component with sphere and/or disk bubbles attached in a bubble-tree fashion, and with the marked point lying in the boundary of one of the disk components.
(See \cite[Theorem 4.4]{Fra08} for a detailed proof.)

As \Cref{lem:uniqueness_lemma} holds for curves with boundary touching $O$ and not just $\tilde O$, we moreover know that there is no nodal curve in $(\overline{\cM_*\setminus \cB_*})\setminus (\cM_*\setminus \cB_*)$ such that its main disk component has boundary touching $O$.
This means that $\overline{\cM_*\setminus \cB_*}$ can be glued (as a topological space) to $\cM_*^O$, made of those marked Bishop disks with boundary touching $O\subset N$ (that is naturally a smooth manifold of dimension $n+1$ according to \Cref{def:bishop_family} and \Cref{prop:bishop_discs_regular}), along the subfamily of Bishop disks with boundary touching $O\setminus \tilde O$.
Moreover, according to \Cref{lem:uniqueness_lemma}, this glued moduli space can naturally be compactified by adding the constant Bishop disks with image contained in the binding of the confoliated bLob; we denote by $\cMline_*$ this compactification.

\bigskip

Our geometric setup allows, in fact, some additional control on nodal curves in $\cMline_*\setminus \cM_*$. 
More precisely, we first claim that there is no disk bubbling in our setup.
Indeed, for a nodal Gromov-limit curve $[u_\infty,z_\infty]\in \cMline_*$ of a sequence $[u_n,z_n]\in\cM_*$, we must have that the concatenation of the boundaries of all the disk components winds around the binding $B$ positively exactly once inside $N$, as this is true for the Bishop disks, $\cM_*$ is path-connected (by definition), and this winding number passes to the Gromov-limit. 
If, by contradiction, there is more than one (non-constant) disk component in $u_\infty$, this means that there is at least one of them, which we denote by $v_0$, whose winding number is non-positive.
But this contradicts the first part of \Cref{lem:boundary_disc_transverse_xiN}, which implies that the winding number of $\partial v_0$ around $B$ in $N$ is strictly positive. 
This proves that there cannot be any disk bubbling in $\cMline_*$.

Lastly, we claim that every sphere bubble has its image in $W\setminus (-\varepsilon,\delta)\times M$ and is not contained in $\{-\varepsilon\}\times M$; in particular, it intersects $W\setminus [-\varepsilon,\delta)\times M$.
Indeed, let $v$ be a non-constant sphere component. If $\psi \circ v$ attained a value in $(-\varepsilon,\delta)$ at its maximum, \Cref{lem:maximum_principle_curve} would give that $v$ lies in a slice $\{t_1\}\times M$ and is tangent to $\cK_\xi$ (recall from Section \ref{sec:arranging_boundary_data} that the characteristic foliation of the set of complex tangencies at a level set of $\psi$, whose level sets are the same as those of $t$, is $\cK_{\xi}$). If $v$ were contained in $\{-\varepsilon\}\times M$, then $t\circ v$ would be constant, and \eqref{eqn:maximum_principle_bis} and \eqref{eqn:maximum_principle} (which hold up to $t=-\varepsilon$) would give the same conclusion.
In both cases, as $\Omega\vert_{\cK_\xi}=f\mu\vert_{\cK_\xi}$ on each slice for some $f>0$, we get $0\leq v^*\Omega=(f\circ v)\, v^*\mu$, so asphericity of $(\xi,\CS_{\xi,\mu})$ gives $v^*\Omega=0$. Hence $E(v)=0$ and $v$ is constant, a contradiction.

In conclusion, marked nodal curves in $\cMline_*$ are necessarily made of one main disk component, whose boundary is positively transverse to $\xi_N$ and has winding number exactly $1$ around $B$ in $N$, and possibly several bubble components, that lie completely in $W\setminus (-\varepsilon,\delta)\times M$ and are not entirely contained in the slice $\{-\varepsilon\}\times M$.

\subsubsection{The evaluation map is a pseudocycle}

Let us now go back to the evaluation map $\ev\colon \cM_*\to N$, $[u,z]\mapsto u(z)$, considered in \eqref{eqn:evaluation_map}. 
Due to our assumption of semi-positivity for $(W,\Omega)$, we have the following key property:
\begin{prop}
    \label{prop:ev_pseudocycle}
    The evaluation map $\ev$ is a pseudocycle. 
\end{prop}
Here, \textbf{pseudocycle} is intended as in \cite[Definition 6.5.1]{McDSalBook}.
More explicitly, this means that $\ev(\cM_*)$ has compact closure and the \textbf{limit set}
\[
    \Omega_{\ev} := \bigcap_{{\footnotesize \begin{array}{c}K\subset \cM_* \\ K \text{ compact} \end{array}}} \overline{\ev(\cM_*\setminus K)}
\]
is contained in the image of a smooth map $f\colon X \to N$ defined on a smooth manifold of dimension at most $\dim \cM_* -2$.

\begin{proof}
    This can be proved exactly as in \cite[Proposition 12]{Nie06} (cf.\ also \cite[Proposition 49]{NieNotes}).
    We hence omit the details and limit ourselves to describing the general idea.

    \smallskip

    The fact that $\ev(\cM_*)\subset N$ has compact closure follows from compactness of $N$, Gromov's compactness,
    and the fact that the (non-constant) curves of interest for us have boundary that stays away from $\partial N$ according to \Cref{lem:boundary_is_wall}.

    As far as the limit set $\Omega_{\ev}$ is concerned, Gromov's compactness theorem tells that $\cM_*$ can be compactified by adding once-boundary-marked nodal curves.
    Moreover, by the discussion in \Cref{sec:compactification_moduli_space}, in our setup, there is no disk bubbling at all, and sphere bubbles intersect $W\setminus ([-\varepsilon,\delta)\times M)$.
    In other words, the only marked nodal curves we deal with are composed of a main disk component, with a marked point at its boundary, and sphere bubbles attached to its interior according to a modeling bubbling tree.
    In other words, $\cMline_*\setminus \cM_*$ is the union, over all bubbling trees $T$, of the moduli spaces $\cM_*^T$ made of those marked nodal curves that are of the following form: there is one main disk component, corresponding to a special vertex $v_0$ in $T$ and equipped with a marked point in its boundary, and attached to its interior we have a nodal curve made only of sphere components, each corresponding to a vertex of $T\setminus \{v_0\}$, and they are ``pairwise attached'' according to the edges of $T$ (where pairwise attached means that there are special additional marked points that are required to be mapped, under their associated evaluation map valued in $W$, to the same point, so that their images touch at those points).
    In particular, $\cM_*^T$ is simply a fiber product of the moduli spaces of each component, w.r.t.\ the evaluation maps associated to the additional nodal marked points (cf.\ \cite{Fra08} for more details.)
    Hence, $\cM_*^T$ is a smooth manifold if all the moduli spaces involved in the product can be arranged to be smooth manifolds and the product of the evaluation maps at the nodal points in each component is transversal to the diagonal in the self product $W^{\times 2k}$ of $W$ with itself $2k$-times, where $k$ is the number of nodes.

    Now, according to \Cref{lem:boundary_is_wall,lem:regularity_condition_satisfied}, our choice of $J$ makes all the involved (nodal or smooth) disk components Fredholm regular.
    For sphere bubbles, recall that (see e.g.\ \cite{McDSalBook}) their set of somewhere injective points is dense in the domain; hence, by the fact (argued above) that such bubbles must intersect the region $W\setminus [-\varepsilon,\delta)\times M$ where $J$ is generic, we know that these are also Fredholm regular.
    Because of this, we know that all the previously mentioned moduli spaces involved in the fiber products defining each $\cM_*^T$ are smooth manifolds, and also the involved evaluation maps at marked nodal points can be arranged to be transverse to the diagonal (up to slightly perturbing again $J$ away from the model regions, see \cite[Theorem 4.60]{WenBook}). 
    In other words, each $\cM_*^T$ is a smooth manifold.
    
    At this point, one can make (see \cite[Proposition 12]{Nie06} for details) an explicit virtual dimension count using the fiber product description and \emph{relying fundamentally on the semipositivity assumption} (that we have according to the assumptions of \Cref{thm:obstruction_fillability}) \emph{and on the fact that there is only sphere bubbling in $W\setminus ([-\varepsilon,\delta)\times M)$} (according to the previous discussion in \Cref{sec:compactification_moduli_space}).
    This computation shows moreover that the dimension of all the $\cM_*^T$'s that are non-empty is at most $\dim \cM_*-2$.

    Lastly, notice that each marked moduli space of marked nodal curves $\cM_*^T$ comes naturally with a smooth evaluation map $\ev_T$ at the marked point that is at the boundary of the domain of the unique (by the previous discussion) disk components of the nodal curve. 
    Hence, the choice of $X:=\sqcup_T \cM_*^T$ and $f:=\sqcup_T \ev_T$ shows that $\ev$ is a pseudocycle according to the above definition.
\end{proof}

%%%%%%%%%%%%%%%%%%%%%%%%%%%%%%%%%%%%%%%%%%%%%%%%%%%%%%%%%%%%%%%%%%%%%%

\subsection{Obstruction to fillability from confoliated bLobs}
\label{sec:proof_main_result}

In this section, we aim to reach a contradiction working with all the information we gathered so far, which essentially relies on the assumption (by contradiction) that a symplectic filling $W$ exists. {(We will list the differences for the case of semi-filling in \Cref{rmk:case_semifilling} below.)}
More specifically, we will use the boundary-marked moduli space $\cM_*$ together with the description of its Gromov compactification $\cMline_*$, and the evaluation map $\ev\colon \cM_*\to N$ in \eqref{eqn:evaluation_map}, which we know to be a pseudocycle.

\medskip

The first observation to make is that in our setting $\dim \cM_*=\dim N$; hence, according to \Cref{prop:ev_pseudocycle}, the limit set of the evaluation map $\Omega_{\ev}$ is covered by the image of smooth maps defined on manifolds which have dimension $\dim \cM_*-2=\dim N -2$.
(This is the meaning of the informal slogan ``under the assumption of semi-positivity, sphere bubbling is a codimension $2$ phenomenon''.)

Let then $\gamma \colon [0,1] \hookrightarrow N$ be an embedded curve, such that $\gamma^{-1}(B) = \{0\}$ and $\gamma^{-1}(\partial N) = \{1\}$.
Up to smooth perturbation of $\gamma$ relative to endpoints, we can moreover assume that $\ev$ is transverse to $\gamma$ and, by what we have just discussed in the previous paragraph, that $\ev^{-1}(\gamma)$ avoids honest-nodal curves (i.e.\ nodal curves having at least one sphere bubble) completely.

Let $\cM_*^\gamma$ be the connected component of the pullback $\ev^{-1}(\gamma)$ containing the model marked Bishop disks converging to the constant Bishop disk $[u_0,z_0]$ at $\gamma(0)$. 
Then, $\cM_*^\gamma$ is a smooth manifold by Thom's transversality theorem.
Moreover, since $\cM_*$ has dimension equal to that of $N$, the dimension of $\cM_*^\gamma$ is equal to that of $\gamma$, i.e.\ it is $1$-dimensional.
Now, it cannot be diffeomorphic to the circle: indeed, $\cM_*^\gamma$ is not compact as it contains non-constant marked Bishop disks converging to the constant one $[u_0,z_0]$ over $\gamma(0)$, but it does not contain $[u_0,z_0]$ itself (by definition of $\cM_*$).
Hence, $\cM_*^\gamma$ must be an open interval, say diffeomorphic to $(-1,1)$.

We then look at the compactification $\cMline_*^\gamma$ of $\cM_*^\gamma$ obtained by the Gromov compactness theorem and by adding the constant Bishop disk with image the point $\gamma(0)\in B\subset N$.
By construction $\cM_*^\gamma$ avoids sphere bubbling; hence $\cMline_*^\gamma\setminus \cM_*^\gamma$ can only contain the constant Bishop disk $[u_0,z_0]$. 
Now, by \Cref{lem:uniqueness_lemma}, this means that there are neighborhoods $(-1,-1+\delta_-)$ and $(1-\delta_+,1)$ of the two ends of $\cM_*^\gamma\simeq (-1,1)$ that in fact are identified.
But no matter the values of $\delta_\pm$, this contradicts the fact that $\cM_*^\gamma$ is a smooth manifold.
This concludes the proof.
\hfill\qedsymbol

\begin{remark}
    \label{rmk:case_semifilling}
    In the case where $W$ is a semi-filling, denote by $M_0$ the component of $\partial W$ containing the confoliated bLob and by $M_1$ the union of the other ones. We assume that each component of $M_1$ is equipped with a geodesible tame c-symplectic confoliation, which is aspherical unless $W$ is; this is automatic if the components of $M_1$ are contact, as then $\cK=0$.
    Near $M_1$, we then choose $J$ as in \Cref{lem:t_is_weak_plurisubharm} (without any model near a bLob), so that the collar coordinate is tame weakly pluri-subharmonic there as well.
    Then, disks in $\cM_*$ cannot touch $M_1$: otherwise, by \Cref{lem:maximum_principle_curve} they would be contained in $M_1$, while their boundary lies in $N\subset M_0$. Sphere bubbles are excluded near $M_1$ exactly as near $M_0$ in \Cref{sec:compactification_moduli_space}.
    The rest of the proof goes through unchanged.
\end{remark}

\section{Applications}
\label{sec:applications}

In this section, we prove all the consequences of \Cref{thm:obstruction_fillability} as stated in the introduction.
Namely, 
in \Cref{sec:fillings_products} we prove \Cref{cor:split_strong_sympl_conf}; in \Cref{sec:weak_fillings_bourgeois_manifolds} we describe how to deduce a non-trivial statement about contact manifolds, namely \Cref{cor:obstruction_weak_filling_bourgeois} concerning weak fillability of Bourgeois manifolds; 
lastly, in \Cref{sec:Reeb_orbits} we show how to deduce dynamical consequences as claimed in \Cref{thm:contractible_reeb_orbit}.

%%%%%%%%%%%%%%%%%%%%%%%%%%%%%%%%%%%%%%%%%%%%%%%%%%%%%%%%%%%%%%%%%%%%%%

\subsection{Fillings of products of contact and symplectic manifolds}
\label{sec:fillings_products}

We prove in this subsection \Cref{cor:split_strong_sympl_conf}. 
Recall that in this context $(M,\eta=\ker\alpha)$ is a closed and connected $(2n+1)$-dimensional overtwisted contact manifold, and $(X,\Omega_X)$ is a closed and connected symplectic $2k$-dimensional one which is symplectically aspherical and admits a weakly exact Lagrangian $L$. The c-symplectic confoliation $(\xi,\CS_{\xi,\omega})$ is tame: if $J_\eta$ is tamed by $\CS_\eta$ and $J_X$ by $\Omega_X$, then $J_\eta\oplus J_X$ is cotamed by $\CpS_\xi$ and every element of $\CS_{\xi,\omega}$. It is geodesible, with $\alpha$ and the lift of its Reeb vector field. It is also aspherical: a sphere tangent to $\cK_\xi$ lies in a leaf $\{p\}\times X$, and if $u^*\omega\geq 0$ pointwise then $u^*\omega=0$, since $\int_{S^2}u^*\Omega_X=0$ by symplectic asphericity of $X$.

\medskip

The third item in the conclusion is an immediate consequence of the main \Cref{thm:obstruction_fillability}.
Indeed, any $\tilde N$ of the form $N_c\times X$ for $N_c$ a contact bLob in $M$ is a confoliated bLob.
Such a confoliated bLob always exists, as by the $h$-principle for overtwisted contact manifolds in \cite{BEM15} there is always a ``small'' (i.e.\ contained in a smooth ball) contact bLob in $M$ (and in fact even a small Plastikstufe in the sense of \cite{Nie06}).
The fact that $\tilde N$ is transversely-exact w.r.t.\ $\Omega$ by assumption allows us to directly use \Cref{thm:obstruction_fillability}.

\medskip

As far as the first and second points in the conclusion of Corollary \ref{cor:split_strong_sympl_conf}, notice that in both cases we still have a confoliated bLob of the form $\tilde N=N_c\times X$ in the boundary of the hypothetical filling. Here $\mu=\pi_X^*\Omega_X$ is basic on any split neighborhood, the bLob is thus transversely-exact by Remark \ref{rem:strong_fillings_transversely_exact} and Lemma \ref{lem:astrfill}, respectively. 
\qedsymbol

\subsection{Weak fillings of Bourgeois contact manifolds}
\label{sec:weak_fillings_bourgeois_manifolds}

We prove in this subsection \Cref{cor:obstruction_weak_filling_bourgeois}. The proof is by contradiction.
Assume that there is a weak filling $(W^{2n+2},\Omega)$ of a Bourgeois contact $(2n+1)$-manifold $(M\times \T^2,\xi)$ with $[\Omega\vert_{\partial W} ] = [\mu]$, where $\mu$ is a volume form on $\T^2$ and $\xi$ is associated to an overtwisted contact structure $\eta$ on $M$.
Recall from \cite{Bou02} that $\xi = \ker \beta$ with $\beta = \alpha + \phi_1 \d x - \phi_2 \d y$, where $\phi = (\phi_1,\phi_2)\colon M \to \R^2$ is a map defining an open book decomposition $(B,\vartheta)$ supporting $\eta$, and $\alpha$ is a contact form for $\eta$ which is adapted to such open book.
Without loss of generality, we can assume $\mu = C \d x \wedge \d y$, for some constant $C>0$.

According to \cite[Lemma 2.10]{MNW} (or to the generalized \Cref{lem:deformation_lemma}), up to stacking on top of $W$ a smoothly trivial symplectic cobordism, we can assume that $\Omega = \d(e^t \beta) + \mu$ on a neighborhood $(T-\delta,T]\times M \times \T^2$ of the boundary of $W$.

\medskip

Let $\varepsilon>0$ be a small number and $\rho\colon [T,T+\ln(1/\varepsilon)]\to \R_{\geq0}$ a non-increasing function such that $\rho(T)=1$ and $\rho(T+\ln(1/\varepsilon))=0$, both $\varepsilon$ and $\rho$ will be determined later.
Extend $\Omega$ on the extended cobordism {$[T,T+\ln(1/\varepsilon)]\times M \times \T^2$ as: 
\begin{align*}
\Omega  =& \; \d(e^t \alpha + e^t\rho(t)\phi_1\d x - e^t\rho(t) \phi_2\d y) + \mu \\
= & \; e^t \d t \wedge\alpha +  e^t \d \alpha  + e^t(\rho+\rho') \d t \wedge(\phi_1 \d x - \phi_2 \d y)  \\
& + e^t\rho \,  (\d\phi_1\wedge \d x - \d\phi_2\wedge \d y) + C \d x \wedge \d y \, .
\end{align*}
}

\noindent
Similarly to the computations in \cite{Bou02}, we then have
\begin{align*}
    \Omega^{n+1} = & \; (n+1) e^t \d t \wedge [\alpha + (\rho'+\rho) \, (\phi_1\d x - \phi_2 \d y)] \\
    & \quad \wedge (e^t \d\alpha + e^t\rho \d \phi_1 \wedge \d x - e^t\rho  \d \phi_2 \wedge \d y + C \d x \wedge \d y)^{n} \\
     = & \; (n+1)n e^{nt} C \d t \wedge \alpha \wedge \d \alpha^{n-1}\wedge \d x \wedge \d y \\
     & \quad + (n+1) n(n-1) \rho^2(t) e^{(n+1)t}\d t \wedge \alpha \wedge \d\alpha^{n-2}\wedge \d\phi_1\wedge\d\phi_2\wedge \d x \wedge \d y\\
     & \quad + (n+1) n (\rho'(t)+\rho(t))\rho(t) e^{(n+1)t} \d t \wedge \d\alpha^{n-1} \wedge (\phi_1\d\phi_2 -\phi_2\d\phi_1) \wedge \d x \wedge \d y \, .
\end{align*}
Now call $\vol_M$ the volume form $\alpha\wedge\d\alpha^{n-1}$ on $M$, and let $f,g\colon M\to \R_{\geq0}$ such that
\[
\alpha\wedge\d\alpha^{n-2}\wedge\d\phi_1\wedge\d\phi_2 = f \vol_M \, ,
\quad 
(\phi_1\d\phi_2 - \phi_2\d\phi_1)\wedge \d\alpha^{n-1} = g \vol_M \, ;
\]
here, $f,g\geq 0$ follows from the fact that the open book $(\phi_1,\phi_2)\colon M \to \R^2$ supports the contact structure $\ker\alpha$.
Putting all terms together, we obtain
\[
\Omega^{n+1}= (n+1)ne^{nt}[C + f(n-1)e^t\rho^2 + g e^t (\rho'+\rho)\rho ] \d t\wedge \vol_M \wedge\d x \wedge \d y \, .
\]
Where $g=0$, there is nothing to prove as all terms are then $\geq 0$; we hence work on $\{g> 0\}$.
Here, to prove that $\Omega^{n+1}$ is a volume form, it is enough to prove that the {sum of the first and third terms can be arranged to be positive. 
This amounts to arranging that} $F(t):=e^t(\rho'+\rho)\rho$ is strictly bigger than $D=-\frac{C}{\operatorname{max}(g)}$.

For this, we consider a function $\rho$ of the form 
\[
\rho = \frac{e^{-t+T}-\varepsilon}{1-\varepsilon} \, ,
\]
{that indeed satisfies the desired boundary conditions $\rho(T)=1$ and $\rho(T+\ln(1/\varepsilon))=0$. Now, for this choice of $\rho$, a} direct computation shows that $F'>0$ {for $T\leq t<T+\ln(1/\varepsilon)$}. Hence, $F$ attains its maximum at $t=T+\ln(1/\varepsilon)$, where $F$ is simply zero, and its minimum at $t_0=T$, at which we have
\[
F(T)=-\varepsilon\frac{e^{T}}{1-\varepsilon} \, .
\]
For $\varepsilon$ small enough we can achieve $F(T)>D$.

Lastly, note that $\rho(T)=1$, so $\Omega$ extends (via the normal neighborhood theorem given e.g.\ in \cite[Lemma 2.6]{MNW}) the symplectic form on 
$(T-\delta,T]\times M \times \T^2$ over $[T,T+\ln(1/\varepsilon)]\times M \times \T^2$ as needed.
Moreover, $\rho(T+\ln(1/\varepsilon))=0$, so that $\Omega$ restricts at the new positive boundary $\{T+\ln(1/\varepsilon)\}\times M \times \T^2$ as 
\[
\Omega\vert_{\{T+\ln(1/\varepsilon)\}\times M\times \T^2} =  \frac{e^T}{\varepsilon} \d \alpha + C \d x \wedge \d y \, .
\]

\medskip

The result of adding this smoothly-trivial cobordism is then a symplectic manifold $(W',\Omega)$ where, on a neighborhood (given again by the standard normal neighborhood theorem for hypersurfaces in e.g.\ \cite[Lemma 2.6]{MNW}) of the form $(-\delta,0]\times M\times \T^2$ of the boundary $M\times \T^2$ one can write $\Omega = \d (e^t \gamma) + \mu$, where $\gamma=\frac{e^T}{\varepsilon}\alpha$.
This is a strong symplectic filling of the strong symplectic confoliation $(\pi_M^*\eta,\pi_{\T^2}^*\mu)$, which cannot exist by \Cref{cor:split_strong_sympl_conf}. 
\hfill
\qedsymbol

%%%%%%%%%%%%%%%%%%%%%%%%%%%%%%%%%%%%%%%%%%%%%%%%%%%%%%%%%%%%%%%%%%%%%%

\subsection{Closed Reeb orbits from confoliated bLobs}
\label{sec:Reeb_orbits}

The proof of \Cref{thm:contractible_reeb_orbit} is analogous to that in the contact case from \cite{AlbHof09}.
The only difference is that the theory of symplectic fillings of confoliations is needed to deal with the fact that the positive boundary of the cobordism is not contact.
Here are some additional details.

\medskip

Assume by contradiction that the stable Hamiltonian structure $(\lambda,\omega)$ at the bottom of the cobordism has no contractible closed Reeb orbits.
As it is customary in SFT, we complete the cobordism at the negative end with a negative symplectization of $(\lambda,\omega)$.
We then use an almost complex structure $J$ that is as in the proof of \Cref{thm:obstruction_fillability} near the positive boundary, cylindrical at the negative end, and generic elsewhere.

As done in the proof of \Cref{thm:obstruction_fillability}, one can construct a local Bishop family of $J$-holomorphic disks stemming from the binding of the core of the transversely-exact confoliated bLob.
This gives rise, as before, to a moduli space $\cM$ of unparametrized $J$-holomorphic disks, with boundary condition on the core of the transversely-exact confoliated bLob (minus its binding and boundary).

We now consider the energy of the curves in this moduli space, where we intend it here in the SFT sense at the negative end (as defined, e.g.\ in \cite[Section 9.2]{Wendl_SFTLectures}). 
For this type of energy, global energy bounds for the curves in the moduli space can be argued as in \Cref{thm:obstruction_fillability}.
Hence one can apply Gromov's and additionally SFT compactness \cite{BEHWZ} theorems. 

Now, exactly as argued in the proof of \Cref{thm:obstruction_fillability}, one can see that there cannot be any sphere- nor disk-bubbling. 
The main difference in this case, with respect to \Cref{thm:obstruction_fillability}, is that there could be, a priori, SFT-type breaking in pseudo-holomorphic buildings at the negative end. 
However, the assumption (by contradiction) of the lack of contractible closed Reeb orbits for the stable Hamiltonian structure $(\lambda,\omega)$ at the bottom of the cobordism actually prevents this from happening.
Indeed, any non-trivial SFT-type pseudo-holomorphic building arising as SFT-limit of $J$-holomorphic disks with boundary on the confoliated bLob at the positive boundary must involve at least one \emph{contractible} periodic orbit of a Reeb vector field at the negative boundary.

At this point, the same argument as in the proof of \Cref{thm:obstruction_fillability} allows us to reach a contradiction, and this originated exactly from the assumption that there are no contractible closed Reeb orbits at the negative end. This concludes the proof.
\hfill 
\qedsymbol

%%%%%%%%%%%%%%%%%%%%%%%%%%%%%%%%%%%%%%%%%%%%%%%%%%%%%%%%%%%%%%%%%%%%%%
%%%%%%%%%%%%%%%%%%%%%%%%%%%%%%%%%%%%%%%%%%%%%%%%%%%%%%%%%%%%%%%%%%%%%%

\section{Contact approximations}
\label{sec:contact_approx}

We finish this paper by introducing a new notion of convergence of contact structures to confoliations, which is natural from the perspective adopted in this paper, and we give some natural examples as evidence of the relevance of this new notion.

\subsection{Contact approximation of c-symplectic confoliations}

Before properly introducing the definition of contact approximation, let us first explain how one is led to such a definition in the context of c-symplectic confoliations. 

Given a c-symplectic confoliation $(\eta,\CS_{\eta,\mu})$, it is clear that just requiring that a sequence of contact structures $\xi_m=\ker \alpha_m$ converges in the $C^k$-topology to the hyperplane field $\eta$ completely neglects some important data.
Indeed, say for $C^k$-convergence of $\xi_m$ to $\eta$ in the space of hyperplane fields with $k\geq 1$, $[\d\alpha_m\vert_{\xi_m}]$ will $C^{k-1}$-converge to $[\d\beta\vert_{\eta}]$ wherever $\d\beta\neq 0$, with $\beta$ a one-form defining $\eta$.
In particular, both the data of $\CS_{\eta,\mu}$ and of $\cK_\eta$ are completely disregarded in the convergence as hyperplane fields.

The way that we solve this issue is, roughly, as follows.
One would like to require that $\d\alpha_m\vert_{\cK_\eta}$ gives a ``limit conformal class'' of two-forms on $\cK_\eta$ that is ``adapted'' to $\CpS_{\eta}$ and $[\mu\vert_{\cK_\eta}]$. 
The way this limit conformal class will be obtained in practice is by writing down an expression for $\d\alpha_m\vert_{\cK_\eta}$, and considering its terms which go to zero at the slowest rate as $m\to\infty$.
Up to a rescaling, these terms will then give a well-defined (non-trivial) conformal class of two-forms on $\cK_\eta$ in the limit.

Then, a natural requirement is that this limit conformal class is ``compatible'' with the c-symplectic structure on $\eta$. 
The way we formalize this ``compatibility" condition is by requiring that the limit conformal class is symplectically compatible with $\mu$ along $\cK_\eta$, in a sense that is exactly analogous to \Cref{def:symplectically_compatible}.

\begin{remark}
The reason why we do not ask directly that this limit conformal class just coincides with $[\mu\vert_{\cK_\eta}]$ is that this stronger condition is not satisfied in some very natural examples that we will describe later in \Cref{sec:examples_contact_approximations}.  
\end{remark}

\medskip

The formalization of the idea described above is slightly technical, as the notion depends on the rank of $\cK_{\xi}$ at a given point. 
The sets where this rank is constant can be very wild. 
Thus, we first need to introduce some definitions and notations that we need to make precise what $C^k$-convergence means for objects defined on possibly non-open sets.

\bigskip

Let $M$ be a smooth manifold and $X\subset M$ be a subset.
By an \textbf{$l$-form $\mu$ on $X$} we will mean a section of the bundle $\Lambda^l(T^* M\vert_X)\to X$.
We will say that such a $\mu$ is \textbf{$C^k$-differentiable} if there is an open subset $U\subset M$ with $X\subset U$ and a differential $l$-form $\widetilde{\mu}$ on $U$ of regularity $C^k$ such that $\widetilde{\mu}\vert_X = \mu$ as sections of $\Lambda^l(T^* M\vert_X)\to X$.

Given a partition $M= \sqcup_{i\in I} X_i$ of $M$ into a finite number of subsets $X_i$'s (that are not necessarily open), we will call \textbf{$C^k$-differential partitioned-form $(\mu_i)_{i\in I}$} any collection of differential forms $\mu_i$ on each $X_i$ of regularity $C^k$, of degree $\deg(\mu_i)=d_i$ possibly depending on $i$.
We will simply talk about its \textbf{degree} whenever the latter is constant in $i$, and talk about \textbf{$k$-forms} if the degree is $k$. 
If the degree is $0$, we will simply talk about \textbf{$C^k$-differential partitioned-function}.
Moreover, whenever we want to specify the partition $\cP=\{X_i\mid i\in I\}$ of $M=\sqcup_{i\in I} X_i$, we will also say \textbf{$\cP$-partitioned-form}.

\medskip

A last notion that we will need is that of convergence of the conformal class of a sequence of (non-partitioned) differential forms to the conformal class of a partitioned differential form.
This is needed for us as we will need to consider ``limit conformal classes of two-forms on $\cK_\xi$'', and the latter is not regular. 
Here are the details.

Let $(\eta_i)_{i\in I}$ be a $C^k$-differential partitioned-form on $M$ with respect to a finite partition $M=\bigsqcup_{i\in I}X_i$.
Assume that, for all $ i \in I$, $\eta_i$ is non-zero at every point of $X_i$.
Let also $(\zeta_{m,i} \mid i\in I)_{m\in \N}$ be a sequence of $C^k$-differential partitioned-forms on $M$ with respect to the same partition, and such that $\deg(\zeta_{m,i})=\deg(\eta_i)$ for all $i$ and all $m$.
We then introduce the following notion of conformal convergence.

\begin{definition}
    \label{def:convergence_conformal_class_partitioned_forms}
    We say that the sequence $(\zeta_ {m,i} \mid i\in I)_m$ 
    \textbf{conformally $C^k$-converges} to $(\eta_i \mid i \in I)$, and denote $[(\zeta_{m,i})]\xrightarrow{C^k}[(\eta_i)]$, if there are 
    \begin{itemize}
        \item[-] for all $i \in I$, an open neighborhood $O_i$ of $X_i$,
        \item[-] and a sequence of smooth partitioned-functions $(f_{m,i}\colon O_i\to \R_{>0} \mid i \in I)_m$,
    \end{itemize}
    such that
    \[
    f_{m,i} \,\wdt{\zeta}_{m,i}  \xrightarrow{C^k} \wdt{\eta}_i \quad \text{on } O_i \, ,
    \quad \text{as } m\to \infty \, ,
    \]
    where $\wdt{\zeta}_{m,i}$ and $\wdt{\eta_i}$ are some $C^k$-extensions of $\zeta_{m,i}$ and $\eta_i$ respectively over $O_i$ (that exist by definition of partitioned forms if each $O_i$ is a sufficiently small neighborhood of $X_i$).
\end{definition}
To be completely explicit, the symbol $\xrightarrow{C^k}$ above denotes the limit with respect to the $C^k$-topology on the space of $C^k$-differential forms on the open sets $O_i$'s.

{Below, we will in fact slightly abuse the above notation in the following way. 
Given a partition $\cP$ of $M$ and a finer partition $\cP'$ 
(i.e.\ one such that every subset in $\cP$ is union of subsets in $\cP'$), we will talk about sequences of $\cP'$-partitioned-forms converging to a $\cP$-partitioned-form. 
This makes sense as the latter can naturally be seen as a $\cP'$-partitioned-form as well, just by restricting each form defined on each element $X$ of $\cP$ to the subsets of $\cP'$ that give as a union $X$ itself.}

\begin{remark}
\label{rmk:conformal_convergence_forms}
Considering the limit above over the enlarged \emph{open} sets $O_i$ is necessary to make sense of $C^k$-convergence, which is only defined for differential forms defined on open subsets of manifolds. 
Moreover, allowing the choice of the extensions $\widetilde \zeta_{m,i}$ in the definition above is necessary because the limit is asked on a domain which is bigger than $X_i$ itself, and any chosen $C^k$-extension $\wdt{\zeta}_{m,i}$ of $\zeta_{m,i}$ over $O_i$ (that exists by definition of partitioned form) might have different ``dominating terms as $m\to\infty$'' 
on $O_i\setminus X_i$.
The following toy example illustrates this phenomenon.
\end{remark}

\begin{example}
    \label{exa:need_for_constant_form}
    Let $M=\R^2$ with coordinates $(x,y)\in \R^2$, 
    and $\zeta_m = x\, \d x+ \frac{1}{m}\d y$.
    Let then $X_1 = \{(x,y)\in\R^2 \mid x\neq 0 \}$, and $X_2 = M\setminus X_1 = \{(0,y)\in \R^2\}$, and define $\mu_1 = x \, \d x$ and $\mu_2 = \d y$, as $C^k$-differential forms respectively on $X_1$ and on $X_2$.
    In fact, $\mu_1,\mu_2$ are also naturally defined on the whole of $M$.
    Then, it is clear that over the open set $X_1$ the conformal class of $\zeta_m$ converges to that of $\mu_1$.
    However, as $X_2$ is not open, in order to talk about $C^k$ limits one needs to look at an open neighborhood $O_2$, and convergence of differential forms defined there.
    If we look at $\zeta_m$ and $\mu_2$, that are naturally defined over all of $M$, it is clear that $\mu_2$ is the conformal limit of $\tilde \zeta_m=\zeta_m-x \d x = \frac{1}{m} \d y$, which is indeed an extension of $\zeta_m\vert_{X_2}$ over a neighborhood $O_2$ of $X_2$, whereas $\zeta_m$ itself (that obviously also extends $\zeta_m\vert_{X_2}$) does not admit a $C^k$-differential form as a limit even up to rescaling on $O_2$.
\end{example}

\bigskip 

We are now ready to talk about contact approximations of c-symplectic confoliations.
More precisely, in what follows, we will consider differential partitioned $k$-forms with respect to partitions of our ambient confoliated manifold obtained as follows.

Let $M$ be a smooth $(2n+1)$-dimensional manifold, and $(\eta, \CS_{\eta,\mu})$ a c-symplectic confoliation on it.
For $0\leq i \leq n$, consider the subset $U_i \subset M$ given by the locus of points $p\in M$ such that $\rk(\cK_{\eta})(p)< 2(i+1)$ (or, said otherwise, $\rk(\cK_\eta)(p)\leq 2i$); 
as this rank is integer-valued and upper semi-continuous in $p$ by \Cref{lem:charact_distr_vs_kernel_distr}, such $U_i$'s are open subsets of $M$.
We also define $U_{-1}=\emptyset$.
Note that $U_{n}=M$ and $U_{i}\supset U_{i-1}$, and denote, {for $i=0,\ldots,n$, $C_i := U_{i}\setminus U_{i-1}$,} which is nothing else than the set of points where $\rk(\cK_{\eta})=2i$.
On $M$, we then consider the partition $\cP = \{C_i\mid 0\leq i \leq n\}$.

\begin{definition}
    \label{def:contact_approximation_deformation}
    Let $k\in\{0,1,\ldots,\infty\}$.
    Given a c-symplectic confoliation $(\eta=\ker\beta,\CS_{\eta,\mu})$, we call \textbf{contact $C^k$-approximation} of the latter a family of contact structures $\xi_m=\ker(\alpha_m)$ such that:
    \begin{enumerate}
        \item\label{item:contact_approximation_deformation__hyperplanes} $\xi_m$ $C^k$-converges to $\eta$ as hyperplane fields;

        \item\label{item:contact_approximation_deformation__C0_case} if $k=0$, we moreover require that $\CS_{\xi_m}$ on $\xi_m$  $C^k$-converges to $\CpS_{\eta}$ on $\eta$ at every point where the latter is non-trivial;

        \item\label{item:contact_approximation_deformation__char_distr} 
        {there exists a sub-partition $\cP'= \{C_{i,j}\mid 0\leq i \leq n \, , \; j\in J_i\}$ of $\cP=\{C_i \mid 0\leq i \leq n\}$ (i.e.\ $C_i=\cup_{j\in J_i} C_{i,j}$ for all $i=0,\ldots, n$) with $J_0=\{*\}$, and a $C^k$-differential partitioned-form $(\mu_{i,j} \mid 0\leq i \leq n, j\in J_i)$, where $\deg\mu_{i,j}$ is $2$ for $i\geq  1$ and $\mu_{0,*}$ is just the constant (on $C_{0,*}=C_0$) function equal to $1$, such that moreover:}
        
    \begin{enumerate}
        \item\label{item:contact_approximation_deformation__char_distr__non_zero_conf_class}   
        {for $i=1,\ldots, n$ and $j\in J_i$, $\beta \wedge \d \beta^{\, n-i}\wedge\mu_{i,j}$ is non-zero at every point of $C_{i,j}$;}
        \item\label{item:contact_approximation_deformation__char_distr__conformal_convergence} the $C^k$-differential $\cP'$-partitioned-form 
        {
        \[
            (\alpha_m\wedge \d\alpha_m^n \text{ on $C_{0,*}$ } \, ;\;
            \alpha_m\wedge \d\alpha_m^{n-i+1} \text{on $C_{i,j}$ , for } 1\leq i \leq n, j\in J_i)
        \]
        }
        conformally $C^k$-converges to {$(\beta\wedge\d\beta^{n-i}\wedge\mu_{i,j} \mid 0\leq i\leq n, j\in J_i)$;}

        \item\label{item:contact_approximation_deformation__char_distr__sympl_compat}  {for all $1\leq i \leq n$ and $j\in J_i$, the conformal class $[\mu_{i,j}]$} is symplectically compatible with the symplectic cone $\CS_{\eta,\mu}$ on $\cK_\eta\vert_{C_i}$, i.e.\ {(recalling $\d\beta=0$ on $\cK_\eta$) 
        \[
        \beta\wedge \d\beta^{n-i}\wedge (\mu + t \mu_{i,j})^i >0 \quad
        \text{for all } t>0 \text{, at all points of } C_{i,j}
        \, .\]}
    \end{enumerate}
    \end{enumerate}

    \noindent
    A \textbf{contact $C^k$-deformation} is defined analogously to a contact $C^k$-approximation, but with a parameter  $s\in[0,1]$ instead of $m$, with the family being $C^k$ in the parameter $s$ too, and $s\to 0$ instead of $m\to \infty$.
\end{definition}

\begin{remark} Some comments about this definition are in order.
    \label{rmk:contact_approx_remarks}
    \begin{enumerate}
        
        \item The property in \Cref{item:contact_approximation_deformation__C0_case} holds automatically for $k\geq 1$, as $\xi_m \xrightarrow{C^k} \eta$ just means that there are defining one-forms $\alpha_m$ and $\beta$ for $\xi_m$ and $\eta$ respectively such that $\alpha_m \xrightarrow{C^k} \beta$, so that $\CS_{\xi_m}=[\d\alpha_m]$ $C^{k-1}$-converges to $\CpS_\eta =[\d\beta\vert_\eta]$ at all points where $\d\beta\neq 0$.

        \item Due to the fact that, for a given real vector space of dimension $2$, there is a unique conformal class of non-degenerate two-forms, 
        the notion in \Cref{def:contact_approximation_deformation} 
        reduces to that considered in \cite{ET_confoliations} in ambient dimension $3$. 
        In particular, any contact approximation in ambient dimension $3$ as in \cite{ET_confoliations} also fits \Cref{def:contact_approximation_deformation}. 

    \item {The need to use a partitioned form with respect to a sub-partition of $\cP$ in \Cref{item:contact_approximation_deformation__char_distr} comes from the fact that, inside a same $C_i$, the differential form $\alpha_m\wedge \d\alpha_m^{n-i+1}$ might have leading terms ``of different order as $m\to\infty$'' depending on the point in $C_i$.
    For this reason, it is sensible to allow for a partition of each $C_i$ to be considered in order to single out distinct leading terms for $m\to\infty$.}

    \item {\Cref{item:contact_approximation_deformation__char_distr__non_zero_conf_class} means that each $\mu_{i,j}$ is non-vanishing on $\cK_\eta$ at every point of $C_{i,j}$ (but not necessarily non-degenerate there).}

    \item\label{item:rmk_contact_approx_remarks__qualitative_meaning} 
    In \Cref{item:contact_approximation_deformation__char_distr__conformal_convergence}, conformal convergence on (the open subset) $C_0$ simply follows from \Cref{item:contact_approximation_deformation__hyperplanes,item:contact_approximation_deformation__C0_case}, as $C_0$ is exactly the set of points where $\eta$ is contact. The qualitative meaning of convergence on the other $C_i$'s is then
    that $\CS_{\xi_m}$ induces at the limit, {over those subsets as well,} a well-defined conformal symplectic structure on $\cK_\eta$.
More precisely, over $C_i$ where $\rk(\cK_\eta)=2i$, taking the wedge product with $\beta\wedge \d\beta^{n-i}$ has the effect of isolating $\mu_{i,j}$ \emph{exactly} on the directions of $\cK_\eta$.
Indeed, $ \beta\wedge\d\beta^{n-i}$ vanishes when contracted with a vector of $\cK_\eta$, so all $(2n-2i+3)$-tuples of tangent vectors to $M$ at a given point $p\in C_{i,j}$ chosen from a basis extending one of $\cK_\eta$ on which $\beta\wedge \d\beta^{n-i}\wedge \mu_{i,j}$ does not vanish must have exactly two vectors tangent to $\cK_\eta$ which get evaluated by $\mu_{i,j}$.
The convergence condition in \Cref{item:contact_approximation_deformation__char_distr__conformal_convergence} then gives a canonical (i.e.\ independent of the defining forms $\alpha_m$ and $\beta$) way of asking that ``$\d\alpha_m$ defines a conformal symplectic structure on $\cK_\eta$ at the limit''.

Note that another ``more direct'' definition could have been to ask that 
the partitioned form {$(\beta\wedge\d\beta^n \, ; \; \beta\wedge\d\beta^{n-i}\wedge \d\alpha_m \mid 1\leq i\leq n,\, j\in J_i) $} conformally $C^k$-converges to $(\beta\wedge \d\beta^n \, ; \; \beta\wedge\d\beta^{n-i}\wedge\mu_{i,j} \mid 1\leq i\leq n,\, j\in J_i)$.
However, this is \emph{not canonical}, as it really depends (even up to rescaling) on the choice of $\alpha_m$ and not only on $\xi_m$.

        \item On $C_0$ and $C_1$, \Cref{item:contact_approximation_deformation__char_distr__conformal_convergence,item:contact_approximation_deformation__char_distr__sympl_compat} are automatic for positive contact structures $\xi_m$, as the forms involved are of top degree and $\cK_\eta$ has rank at most $2$.
        This is the reason why the notion reduces to that of \cite{ET_confoliations} in dimension $3$.

\item The partitioned-form  $(\mu_{i,j}\mid 0\leq i \leq n, j \in J_i)$ in \Cref{item:contact_approximation_deformation__char_distr} need not be unique. 
However, its restriction to $\cK_\eta$ is (up to conformal factor) uniquely determined by \Cref{item:contact_approximation_deformation__char_distr__conformal_convergence}, and this is the data that really matters for the properties required in \Cref{item:contact_approximation_deformation__char_distr__non_zero_conf_class,item:contact_approximation_deformation__char_distr__conformal_convergence,item:contact_approximation_deformation__char_distr__sympl_compat}.

    \end{enumerate}
\end{remark}

Before delving into the proof of \Cref{prop:examples_contact_deformation}, let us give a very simple example of contact approximation where the construction of suitable extensions of the restricted forms is necessary, in the sense that the naturally available extensions do not satisfy the property required to show convergence of conformal classes.

\begin{example}
Consider in $\mathbb{R}^5$ with coordinates $(z, x_1, y_1, x_2, y_2)$ the plane field $\xi=\ker \alpha_0$, with $\alpha_0= \d z + x_1^3\d y_1 + x_2 \d y_2$ equipped with the symplectically compatible two-form $\mu= \d x_1\wedge \d y_1$. This form endows the confoliation $\xi$ with the c-symplectic structure $\CS_{\xi,\mu}$. Consider the family of contact forms
$$\alpha_s= \d z + (x_1^3+s x_1)\d y_1 + x_2 \d y_2, $$
whose kernel converges to $\xi$. The partition induced by the confoliation is $C_1=\{x_1=0\}$ and $C_0=\{x_1\neq 0\}$. In $C_0$, which is open and where the rank of $\cK_{\xi}$ is zero, we simply need that $\alpha_s$ converges in the $C^\infty$-topology to $\alpha_0$, which is the case.
In $C_1$, we choose $\mu_1=\mu$, which satisfies \Cref{item:contact_approximation_deformation__char_distr__non_zero_conf_class}. 
As an extension of 
\[
[\alpha_s\wedge \d\alpha_s^2]\vert_{C_1}= 
[\left(6x_1^2+2s \right) \d z\wedge \d x_1 \wedge \d y_1 \wedge \d x_2 \wedge \d y_2]\vert_{C_1}
\]
to a neighborhood of $C_1$ we can simply choose
\[
\zeta_s= 2s \d z\wedge \d x_1 \wedge \d y_1 \wedge \d x_2 \wedge \d y_2 \, . 
\]
We choose 
$$\nu=\alpha_0\wedge \d\alpha_0 \wedge \mu= \d z\wedge \d x_1 \wedge \d y_1 \wedge \d x_2 \wedge \d y_2,$$ as an extension of $\alpha_0\wedge \d\alpha_0 \wedge \mu$ (restricted to $C_1$) to a neighborhood of $C_1$. 
Choosing as conformal factors $F_s:=\frac{1}{2s}$, we see that $F_s\zeta_s$ $C^\infty$-converges to $\nu$ in a neighborhood of $C_1$, which shows that \Cref{item:contact_approximation_deformation__char_distr__conformal_convergence} holds.
Since the c-structure is given by $\mu$, and we chose $\mu_i=\mu$, the symplectic compatibility (\Cref{item:contact_approximation_deformation__char_distr__sympl_compat}) is trivially satisfied too. 
This shows that $\ker \alpha_s$ is a contact approximation of $\xi$ endowed with $\CS_{\xi,\mu}$. 
Notice that, with the constant conformal factors $F_s$, one cannot choose $\alpha_s\wedge \d\alpha_s^2$ itself as an extension of its restriction to $C_1$ to a neighborhood, as the form $F_s \alpha_s\wedge \d\alpha_s^2$ does not converge in a neighborhood of $C_1$. 
(Since the forms involved here are of top degree, one could instead use the non-constant factors $\frac{1}{6x_1^2+2s}$. However, on strata where $\rk\cK_\xi\geq 4$ this is no longer possible, and the choice of extension becomes essential.
In the proof of \Cref{item:open_book_construction__deformation} of \Cref{prop:open_book_construction}, we will also need to construct specific extensions of the restricted forms.)
\end{example}

On the other hand, a simple example of contact plane fields that converge to a c-symplectic confoliation but do not provide a contact approximation in the sense of \Cref{def:contact_approximation_deformation} can be constructed by choosing the limit conformal form to be non-compatible with the c-symplectic structure.
For example, on $\mathbb{R}^5$ with coordinates $(z,x_1,y_1,x_2,y_2)$, the family of contact forms $\alpha_s= \d z + sx_1 \d y_1 + sx_2 \d y_2$ does not provide a contact approximation of the foliation $\xi=\ker \d z$ endowed with the c-symplectic structure induced by $\mu=-(\d x_1\wedge \d y_1 + \d x_2\wedge \d y_2)$: indeed, $\frac1s\alpha_s\wedge\d\alpha_s\to \d z\wedge\mu_2$ with $\mu_2=-\mu$, and $\d z\wedge(\mu+t\mu_2)^2=2(1-t)^2\,\d z\wedge\d x_1\wedge\d y_1\wedge\d x_2\wedge\d y_2$ vanishes at $t=1$.

%%%%%%%%%%%%%%%%%%%%%%%%%%%%%%%%%%%%%%%%%%%%%%%%%%%%%%%%%%%%%%%%%%%%%%%%%%

\subsection{Examples of contact approximations}
\label{sec:examples_contact_approximations}

We now move to describe some natural examples of contact deformations in the sense of \Cref{def:contact_approximation_deformation}, starting with \Cref{prop:examples_contact_deformation}.
For simplicity, we divide the proof into several cases, one for each example.

\begin{proof}[Proof of \Cref{item:examples_contact_deformation__fiber_sum_branched_cover} in \Cref{prop:examples_contact_deformation}]
    We treat the case of branched covers, following the ideas in \cite[Section 2]{Gir20_constructions} (see also \cite[Section 5.2]{GirThesis}).
    The case of fibered sums can be detailed in a completely analogous way, following the ideas in \cite[Section 5.3]{GirThesis}.

    \medskip

    Let $(M,\xi=\ker\alpha)$ be a contact $(2n+1)$-manifold, and $\pi\colon \hat M \to M$ a smooth branched cover over a codimension $2$ contact submanifold $B\subset M$.
    Denote also $\hat B = \pi^{-1}(B)$, and choose $\gamma$ any connection form on the sphere normal bundle $S\nu_{\hat B}^{\hat M}$ of $\hat B$ in $\hat M$.
    Let also $g\colon \hat M \to \R_{\geq 0}$ be a distance function from $\hat B$, w.r.t.\ an auxiliary Riemannian metric. 
    Up to rescaling the auxiliary metric, we assume that all values in $(0,1]$ are regular for $g$, and their preimages are hence smooth copies of $S\nu_{\hat B}^{\hat M}$ around $\hat B$.
    Let lastly $\chi\colon \R \to [0,1]$ be a function equal to $1$ near $0$ and to $0$ near $1$.

    Then $\hat\xi=\pi^*\xi$ is naturally a c-symplectic confoliation.
    Indeed, the conformal partially-symplectic structure $\CpS_{\hat \xi}=\pi^*(\CS_\xi)$ and the two-form $\mu\coloneqq\chi(g)\,g\,\d g\wedge \gamma$ (which is smooth, being near $\hat B$ an area form on the normal disks) generate a cone of non-degenerate two-forms on $\hat\xi$: as $\mu$ has rank $2$, for all $t>0$ we have $\pi^*\alpha\wedge(\pi^*\d\alpha+t\mu)^n=\pi^*(\alpha\wedge\d\alpha^n)+nt\,\pi^*(\alpha\wedge\d\alpha^{n-1})\wedge\mu$, which is positive as the first term is positive away from $\hat B$ and the second is non-negative, and positive near $\hat B$ since $B$ is a contact submanifold (up to taking the support of $\chi$ small enough). In particular, $\mu$ is non-degenerate on $\cK_{\hat\xi}$, which is the normal bundle of $\hat B$ along $\hat B$ and zero elsewhere.
    (Note that $\hat\xi$ was already proven to be a confoliation in the sense of \Cref{def:smooth_confoliation} and even of \cite{ET_confoliations} in \cite{Gir20_constructions}.)

    Consider now the family of (smooth) differential one-forms $\beta_s = \pi^*\alpha + s \varepsilon \chi(g) g^2 \gamma$.
    Denote also $\xi_s=\ker\beta_s$.
    The computations in \cite{Gir20_constructions} (analogous to the ones in \cite{Gei97}) show that, if $\varepsilon>0$ is small enough then $\xi_s$ is a contact structure for every $s\in(0,1]$.

    The only non-trivial properties to prove in order to check that this family gives a contact $C^\infty$-deformation as defined in \Cref{def:contact_approximation_deformation} are \Cref{item:contact_approximation_deformation__char_distr__conformal_convergence,item:contact_approximation_deformation__char_distr__sympl_compat}, for the partition $C_0=\hat M\setminus\hat B$ and $C_1=\hat B$.

    On $\hat M \setminus \hat B$, there is nothing to prove, as $\hat \xi$ is already contact, and $\xi_s\to\hat\xi$ for $s\to 0$ in the $C^\infty$-topology.
    We then restrict to $\hat B$, or more precisely to an open neighborhood $O$ of $\hat B$ in $\hat M$; we can in particular assume that $O$ is such that $\chi(g)=1$ over all of $O$.

    We can then check, at each point of $O$, that
    \[
    \d\beta_s = \pi^*\d\alpha + s \varepsilon \left(2 g \, \d g \wedge \gamma + g^2\d\gamma\right) \, ,
    \]
    and hence that 
    \[
    \beta_s \wedge \d\beta_s^n = \pi^*(\alpha\wedge \d \alpha^n) + s \varepsilon\, \zeta + O(s^2) \, ,
    \]
    where
    \[
    \zeta\coloneqq 2\,n \pi^*(\alpha \wedge \d\alpha^{n-1}) \wedge g \, \d g \wedge \gamma + n\, g^2\pi^*(\alpha \wedge \d\alpha^{n-1})\wedge \d\gamma + g^2\gamma \wedge \pi^*\d\alpha^n \, .
    \]
    In particular, 
    \[
    \frac{1}{s\varepsilon}[\beta_s \wedge \d\beta_s^n - \pi^*(\alpha\wedge \d \alpha^n)]  
    = \zeta  + O(s) 
    \xrightarrow{s\to 0}  \zeta \, .
    \]
    Note that $\pi^*(\alpha\wedge \d\alpha^{n})\vert_{\hat B}=0$, hence $\beta_s \wedge \d\beta_s^n - \pi^*(\alpha\wedge \d \alpha^n)$ can be seen as an extension of $(\beta_s \wedge \d\beta_s^n)\vert_{\hat B}$ over the open set $O$.
    Moreover, as $g=0$ along $\hat B$, the form $\zeta$ is an extension over $O$ of $(\pi^*\alpha\wedge\pi^*\d\alpha^{n-1}\wedge\mu_1)\vert_{\hat B}$, where $\mu_1\coloneqq 2n\mu$, which is non-zero at every point of $\hat B$.
    As $\mu_1$ is a positive multiple of $\mu$, it is trivially symplectically compatible with $\CS_{\hat\xi,\mu}$ along $\hat B$.
    This concludes the proof of \Cref{item:contact_approximation_deformation__char_distr__conformal_convergence,item:contact_approximation_deformation__char_distr__sympl_compat}, as desired.
\end{proof}

\begin{proof}[Proof of \Cref{item:examples_contact_deformation__bourgeois} in \Cref{prop:examples_contact_deformation}]
    Consider a contact form $\alpha$ for $\xi$ adapted to any open book supporting $\xi$.
    Then, Bourgeois contact structures are defined in  \cite{Bou02} via a contact form $\beta$ of the form $\alpha + \phi_1 \d x - \phi_2 \d y$, where $(x,y)$ are coordinates on $\T^2$ and $(\phi_1,\phi_2)\colon M\to \R^2$ is a map defining the given open book decomposition on $M$.
    Let then  
    \[
    \beta_t = \alpha + t\phi_1 \d x - t \phi_2 \d y \, , \text{ for } t \in [0,1] \, .
    \]
    Now, $\ker\beta_t$ clearly $C^\infty$-converges to $\xi\oplus T\T^2$ for $t\to0$.
    
    Moreover, the explicit computation as in \cite{Bou02} then shows that 
    \begin{align*}
    \beta_t \wedge \d\beta_t^n = \; & t^2 n(n-1)\alpha\wedge \d\alpha^{n-2}\wedge\d\phi_1\wedge\d\phi_2 \wedge \d x \wedge \d y
    \\
    & + t^2 n \d\alpha^{n-1}\wedge (\phi_1 \d\phi_2 - \phi_2 \d\phi_1)\wedge  \d x \wedge \d y \, .
    \end{align*}
    Now, again as explained in \cite{Bou02} (or, alternatively, as it follows from the fact that $\beta_1$ is indeed a contact form), 
    \[
    \Theta := n(n-1)\alpha\wedge \d\alpha^{n-2}\wedge\d\phi_1\wedge\d\phi_2 
    + n \d\alpha^{n-1}\wedge (\phi_1 \d\phi_2 - \phi_2 \d\phi_1)
    \]
    is a volume form on $M$ by choice of functions $(\phi_1,\phi_2)$.
    We can hence rewrite 
    \[\beta_t\wedge \d\beta_t^n = t^2 \Theta\wedge \d x \wedge \d y \, . \]

    Denote now by $g$ the positive function such that $\alpha\wedge \d\alpha^{n-1} = g\, \Theta$.
    Notice that the characteristic distribution of $\xi\oplus T\T^2$ is given by $T\T^2$, of constant rank $2$ (i.e.\ the order of every point of $M\times \T^2$ is $n-1$). 
    We choose $\mu=\d x\wedge \d y$, which satisfies \Cref{item:contact_approximation_deformation__char_distr__non_zero_conf_class} in \Cref{def:contact_approximation_deformation} trivially. The equality $f_t\beta_t\wedge \d\beta_t^n = \alpha \wedge \d\alpha^{n-1}\wedge \mu$, where $f_t=\frac{g}{t^2}$, shows that the conformal convergence in \Cref{item:contact_approximation_deformation__char_distr__conformal_convergence} is satisfied. Finally, any two positive area forms on $\T^2$ are symplectically compatible, and thus the symplectic compatibility \Cref{item:contact_approximation_deformation__char_distr__sympl_compat} is satisfied as well.
\end{proof}

\medskip

\begin{proof}[Proof of \Cref{item:examples_contact_deformation__mori} in \Cref{prop:examples_contact_deformation}]

 \Cref{item:contact_approximation_deformation__hyperplanes} in \Cref{def:contact_approximation_deformation} is immediate, so the only thing to do is to prove that \Cref{item:contact_approximation_deformation__char_distr} holds.

As most of the necessary computations are explicitly done in \cite[Section 3]{Mor19}, we use here the notations from the latter, without explicitly reintroducing everything and just recalling the necessary elements.

The aim of \cite{Mor19} was the description of a certain type of contact deformation (called $\varepsilon\tau$-deformation) of the codimension $1$ symplectic foliation $(\cF,\tau_\cF)$ on $\S^5$ found by Mitsumatsu \cite{Mit18}.

In order to describe his contact deformation, Mori works separately on two parts of the ambient $\S^5$, namely the latter is seen as equipped with an open book decomposition, where the binding $N^3$ has the special feature of admitting a ``cosymplectic structure'', namely a closed non-vanishing $1$-form $\nu$ and a closed $2$-form $\tau_N$ that is non-degenerate on $\ker\nu$; in this special case, we also know that $\tau_N$ is of the form $\eta\wedge \gamma$.
{In fact, due to the very special geometry of the manifold $N$, we also have that $\nu$ vanishes on the Reeb vector field $R$ of $\eta$, that $\d\nu = 0$, and that $\d\eta = \gamma \wedge \nu$ for some nowhere vanishing $\gamma$ vanishing on $R$ and such that $\eta\wedge\d\eta=\eta\wedge\gamma\wedge\nu>0$; cf.\ also \cite[Lemma 6.4.4]{TorThesis} and \cite[Section 3.1]{GirTou24}.}

Checking that our definition of contact $C^\infty$-deformation is satisfied over the mapping torus part of the open book is straightforward due to the easy form of the family of contact forms $\alpha_s$ and of the ambient extension $\tau$ of $\tau_\cF$ (the leafwise symplectic structure on $\cF$).
We hence restrict our attention to the case of the neighborhood $N\times D^2$ of the binding, where the computation is less trivial.

In $N\times D^2$, the family of contact forms described by Mori is of the following form\footnote{We changed the parameter $t$ from \cite{Mor19} with $s$, in order to be consistent with the notation from \Cref{def:contact_approximation_deformation}.} 
\[
    \alpha_s = (1-s)f_0\nu + s f_s \eta + g_s \d\theta + (1-s)h\d\rho \, ,
\]
where $\nu,\eta$ are as above, $(\rho,\theta)$ are polar coordinates on $D^2$, and $f_s,g_s$ are auxiliary functions as detailed in \cite[Section 3]{Mor19}.

As the ambient dimension is $5$ and $\cF$ is of corank $1$, the only wedge product to compute in order to check \Cref{item:contact_approximation_deformation__char_distr} of \Cref{def:contact_approximation_deformation} is 
$\alpha_s \wedge \d\alpha_s$.
More precisely, we would like this to converge, after rescaling, to something of the form $\alpha_0\wedge \zeta$, where $\cF=\ker\alpha_0$, and $\zeta$ is a $2$-form, in such a way that $\alpha_0\wedge(\tau+t\zeta)^2>0$ for all $t\geq0$.
In other words, we claim that the limiting c-symplectic confoliation is $(\cF,\CS_{\cF,\tau})$, where $\tau$ is the two-form on $\S^5$ defined in \cite{Mor19} and, as described before, extending $\tau_N$ on the binding and an honest leafwise symplectic structure on $\cF$ (and not just linear-symplectic).

In fact, by the fact that both $f_s$ and $g_s$ are linear interpolations from $f_0$ and $g_0$ to $f_1$ and $g_1$ respectively,
one can see that $\alpha_s$ is in fact a polynomial in $s$.
In other words, in order to find $\zeta$ we need to check the lower order (in $s$) in $\alpha_s\wedge\d\alpha_s$.

We then decompose $\alpha_s = \gamma_0 + s\gamma_1 + s^2\gamma_2$ where 
\[
    \gamma_0 = f_0 \nu + g_0 \d\theta + h \d\rho \, ,
    \quad 
    \gamma_1 = -f_0 \nu + f_0 \eta + {(g_1-g_0)}\d\theta - h \d\rho\, ,
    \quad 
    \gamma_2 = {(f_1-f_0)} \eta 
    \, . 
\] 

Taking the exterior derivatives, we then get 
\begin{align*}
    \d\gamma_0 & = f_0' \d\rho \wedge \nu + g_0' \d\rho \wedge \d\theta \, , \\
    \d\gamma_1 & = -f_0' \d\rho \wedge \nu + f_0' \d\rho \wedge \eta + f_0 \gamma \wedge \nu + {(g_1'-g_0')} \d\rho \wedge \d\theta  \, ,\\
    \d\gamma_2 & = {(f_1'-f_0')} \d\rho \wedge \eta + f_1 \gamma \wedge \nu \, .
\end{align*}
Hence, after simplification of terms that cancel outright, and after using that {$\d\nu=0$, $\d\eta=\gamma\wedge \nu$, and} $f_0 g_0'-f_0' g_0=0$ everywhere, we get that the $s$-term is given by
\[
    \gamma_0 \wedge \d\gamma_1 + \gamma_1 \wedge \d\gamma_0 = f_0 g_0 \d\theta \wedge \gamma \wedge \nu + f_0 h \d\rho \wedge \gamma \wedge \nu + (f_0 g_1' - f_0' g_1) \d\rho \wedge \d\theta \wedge \nu \, .
\]
An explicit check shows that, where $f_0\neq 0$, i.e.\ for $\rho<1/2$, this can be written of the form $\alpha_0\wedge \zeta$ where 
\[
    \zeta = g_0 \d\theta \wedge \gamma + h \d\rho \wedge \gamma + \left( g_1' - \frac{f_0'}{f_0}g_1 \right) \d\rho \wedge \d\theta \, .
\]

\noindent
Now, we can compute, always for $\rho<1/2$, that
\[
    \alpha_0\wedge \zeta^2 = (\alpha_0\wedge \zeta)\wedge \zeta = (\gamma_0 \wedge \d\gamma_1 + \gamma_1 \wedge \d\gamma_0)\wedge \zeta = 0
\]
where the last equality comes from the fact that each term of the wedge vanishes as it contains twice $\d\rho$, $\d\theta$ or $\gamma$.
Moreover, again for $\rho<1/2$, we have that
\begin{align*}
    \alpha_0\wedge \zeta \wedge \tau  &\, = 
    (\gamma_0 \wedge \d\gamma_1 + \gamma_1 \wedge \d\gamma_0)\wedge \tau \\
    & \, =  \left[ -f_0 g_0 f' \eta \wedge \gamma \wedge \nu + (f_0 g_1' - f_0' g_1) \delta \nu \wedge \omega \right] \wedge \d\rho \wedge \d\theta  \\
    & \, = (f_0 g_1' - f_0' g_1) \delta \nu \wedge \omega \wedge \d\rho \wedge \d\theta \, ,
\end{align*}
where the last equality follows from the fact that $g_0=0$ for $\rho<1/2$. 
Using the signs listed in \cite[Page 9]{Mor19}, one can see that each of these is indeed strictly positive on $\{\rho<1/2\}$.
Lastly, we know that $\alpha_0\wedge \tau^2>0$ as $\tau$ is leafwise symplectic on $\cF=\ker\alpha_0$ by \cite{Mit18}.
Hence,
\[
    \alpha_0\wedge (\tau + t \zeta)^2>0 \quad
    \text{for all $t>0$}
\]
as desired. 

\medskip

In order to finish checking \Cref{item:contact_approximation_deformation__char_distr} of \Cref{def:contact_approximation_deformation}, all is left to do is hence to understand what is the dominating term of $\alpha_s\wedge \d\alpha_s$ on $\{\rho\geq 1/2\}$, and prove that it can also be written as $\alpha_0\wedge \zeta'$ for some $\zeta'$ satisfying that $\alpha_0\wedge (\tau+t\zeta')^2>0$ for all $t\geq0$ on $\{\rho\geq 1/2\}$.

The $s^2$-term is given by
\[
    \gamma_0 \wedge \d\gamma_2 + \gamma_1 \wedge \d\gamma_1 + \gamma_2 \wedge \d\gamma_0 \, ,
\]
where, in the region $\rho\geq 1/2$ where $f_0=0$, we simply have
\begin{align*}
    \gamma_0 \wedge \d\gamma_2  & \, = -g_0 f_1' \d\rho \wedge \d\theta \wedge \eta + f_1 g_0 \d\theta \wedge \gamma \wedge \nu + f_1 h d\rho \wedge \gamma \wedge \nu \, ,
    \\
    \gamma_1 \wedge \d\gamma_1  &  \, = 0 \, ,
    \\
    \gamma_2 \wedge \d\gamma_0 & \, = f_1 g_0' \d\rho \wedge \d\theta \wedge \eta
    \, .
\end{align*}
Putting all together, we obtain
\[
    (g_0'f_1-g_0 f_1') \d\rho \wedge \d\theta \wedge \eta + f_1 g_0 \d\theta \wedge \gamma \wedge \nu + f_1 h \d\rho \wedge \gamma \wedge \nu \, .
\]
An explicit computation shows that this can be rewritten as $\alpha_0\wedge \zeta'$ where 
\[
\zeta' = f_1 \gamma \wedge \nu + \frac{f_1 g_0' - f_1' g_0}{g_0^2 + h^2} \Big( h \d\theta \wedge \eta - g_0 \d\rho \wedge \eta \Big) \, .
\]
(Note that the $g_0^2+h^2$ at the denominator is strictly positive on $\{\rho\geq 1/2\}$.)

We now need to show that $\alpha_0\wedge (\tau+t\zeta')^2>0$ for all $t\geq 0$; we claim that each term in $t$ is in fact non-negative, with $\alpha_0\wedge \tau^2$ being strictly positive.

First, the fact that $\alpha_0\wedge \tau^2$ is strictly positive is again consequence of $\tau$ being leafwise symplectic on $\cF=\ker\alpha_0$ by \cite{Mit18}.

Second, we can compute
\begin{align*}
    \alpha_0 \wedge (\zeta')^2  = & \, \alpha_0 \wedge \zeta' \wedge \zeta' = (\gamma_0 \wedge \d\gamma_2 + \gamma_1 \wedge \d\gamma_1 + \gamma_2 \wedge \d\gamma_0)\wedge \zeta' \\
    = & \, (g_0'f_1-g_0f_1') \left[
        f_1 + f_1\frac{g_0^2}{g_0^2+h^2}+f_1\frac{h^2}{g_0^2+h^2}
    \right]\d\rho \wedge \d\theta \wedge \eta\wedge \gamma\wedge \nu \\
    = & \, 2 f_1 (g_0'f_1-g_0f_1') \d\rho \wedge \d\theta \wedge \eta\wedge \gamma\wedge \nu \, .
\end{align*}

\noindent
Now, recalling the special properties of the forms $\eta$, $\nu$ and $\gamma$ on $N$, we know that $\eta\wedge\gamma\wedge\nu>0$ on $N$.
In other words, it is enough to check that the coefficient $f_1(g_0'f_1-g_0f_1')$ is non-negative in order to conclude that the above $\alpha_0\wedge (\zeta')^2$ is a non-negative volume form.
Recalling that on $\{\rho\geq 1/2\}$ we have $f_1=f$ and $g_0=\overline b g$, where $\overline b$, $f$ and $g$ are as defined in \cite{Mor19}, hence in particular $fg'-gf'=1$, we get
\[
    f_1(g_0'f_1-g_0f_1') = f\overline{b}(fg'-gf')+f^2g\overline{b}' \, ,
\]
where each term is $\geq 0$.
Hence, $\alpha_0\wedge(\zeta')^2\geq 0$.

For the last term, wedging $\gamma_0 \wedge \d\gamma_2 + \gamma_1 \wedge \d\gamma_1 + \gamma_2 \wedge \d\gamma_0=\alpha_0\wedge \zeta'$ with $\tau$, we get 
\[
    \alpha_0\wedge\zeta' \wedge \tau = (f^2_1 g_0' - 2f_1 f'_1 g_0) \d\rho \wedge \d\theta \wedge \eta \wedge \gamma \wedge \nu + \delta (f_1 g_0' - f'_1 g_0) \d\rho \wedge \d\theta \wedge \eta \wedge \omega \, ,
\]
which can be seen to be non-negative on $\{\rho\geq 1/2\}$ via the study of signs of each function from \cite[Page 9]{Mor19}.

All these together then yield that, on $\{\rho\geq 1/2\}$, $\alpha_0\wedge (\tau + t \zeta')^2>0$ for all $t\geq 0$, as desired.
This concludes.
\end{proof}

\medskip

We are now ready to deal with \Cref{item:examples_contact_deformation__bertelson_meigniez}, which concerns linear deformations of locally conformal symplectic foliations to contact structures.
In fact, this naturally holds for higher codimensional characteristic distributions as well, so we first generalize the discussion in \cite[Section 1.3]{BerMei23}.

Namely, let $\xi$ be a hyperplane field on a closed manifold $M^{2n+1}$ such that $\cK_\xi$ has constant rank $2k\leq 2n$, with $1\leq k \leq n$, and consider also a defining one-form $\alpha$ for $\xi$.
Let moreover $\nu$ be the \textbf{holonomy form} associated to $\alpha$, i.e.\ a one-form such that $\d\alpha^{n-k+1}=(n-k+1)\nu\wedge\alpha\wedge\d\alpha^{n-k}$;
such a form exists by the assumption that $\cK_\xi$ is of constant rank $2k$, as it is always involutive.
(Indeed, one can locally write $\alpha = e^u \alpha_0$ with $\alpha_0$ pulled back from a contact form as in \Cref{prop:normal_form_xi}, and take $\nu = \d u$; these local choices glue together to a well-defined one-form. Note that the required equality uniquely determines only $\nu\vert_{\cK_\xi}$, and there is freedom to add any one-form vanishing on $\cK_\xi$.)
Now assume that there is another one-form $\lambda$ on $M$, such that $\d_\nu\lambda\coloneqq \d\lambda - \nu \wedge\lambda$ is (positively) non-degenerate on $\cK_\xi$.
Then, a direct computation analogous to the one in the corank $1$ case from \cite{BerMei23} shows that the linear deformation
$\alpha_t := \alpha + t \lambda$
satisfies 
\[
    \alpha_t\wedge (\d\alpha_t)^n = \binom{n}{k} t^k \alpha\wedge \d\alpha^{n-k}\wedge (\d_\nu \lambda)^k + O(t^{k+1})\, ,
\]
so that $\alpha_t$ defines a contact structure for all $t>0$ small enough.

We point out that $\nu$ really depends on $\alpha$, and if another one-form $\alpha'=e^f\alpha$ with $f$ a function is considered, then one can take as associated holonomy form $\nu'= \nu + \d f $.
Then, one can consider the two-form $\d_{\nu'}(\lambda')$ with $\lambda'=e^f \lambda$, which just coincides with $e^f\d_\nu\lambda$ hence will still be leafwise non-degenerate on $\cK_\xi$.

Going back to \Cref{item:examples_contact_deformation__bertelson_meigniez} of \Cref{prop:examples_contact_deformation}, we hence want to prove that this family of contact structures $(\ker\alpha_t)_{t\in(0,\varepsilon]}$ is, for $t\to0$, a $C^\infty$-contact deformation of the c-symplectic confoliation $(\xi,\CS_{\xi,\mu})$ on $M$, where $\mu$ is any two-form on $\xi$ that is symplectically compatible with $\CpS_\xi$ and so that $\mu\vert_{\cK_\xi}=\d_\nu\lambda\vert_{\cK_\xi}$.

\begin{proof}[Proof of \Cref{item:examples_contact_deformation__bertelson_meigniez} in \Cref{prop:examples_contact_deformation}]
We can compute 
\begin{align*}
\alpha_t\wedge \d\alpha_t^{n-k+1} = &\, (\alpha + t \lambda)\wedge (\d\alpha + t\d\lambda)^{n-k+1} \\
= & \, t (n-k+1) \alpha \wedge \d\alpha^{n-k} \wedge \d\lambda + t \lambda \wedge \underbrace{\d\alpha^{n-k+1}}_{(n-k+1)\nu\wedge\alpha\wedge\d\alpha^{n-k}} + O(t^2) \\
= & \, t (n-k+1) \alpha \wedge \d\alpha^{n-k} \wedge \d_\nu\lambda \, +\,  O(t^2) \, .
\end{align*}
As $\cK_\xi$ has constant rank $2k$, we have $M=C_k$ and we choose $\mu_k = \d_\nu\lambda$, which satisfies \Cref{item:contact_approximation_deformation__char_distr__non_zero_conf_class} in \Cref{def:contact_approximation_deformation}.
The convergence $\frac{1}{(n-k+1)t} \alpha_t\wedge \d\alpha_t^{n-k+1} \rightarrow \alpha\wedge \d\alpha^{n-k}\wedge \d_\nu\lambda$ shows that \Cref{item:contact_approximation_deformation__char_distr__conformal_convergence} is satisfied. The symplectic compatibility (\Cref{item:contact_approximation_deformation__char_distr__sympl_compat}) holds trivially in this case as $\mu_k\vert_{\cK_\xi}=\mu\vert_{\cK_\xi}$ by previous choice of $\mu$, hence
$\alpha\wedge \d\alpha^{n-k}\wedge (\mu+s\mu_k)^k=(1+s)^k\alpha\wedge \d\alpha^{n-k}\wedge (\d_\nu\lambda)^k>0$.
Hence $\ker\alpha_t$ is a contact deformation of the c-symplectic confoliation $(\xi,\CS_{\xi,\mu})$, as desired.
\end{proof}

\medskip

We proceed to analyze the examples from \cite[Theorem A]{MNW}, generalizing the discussion in \cite[Section 2.3]{ET_confoliations} for the case of the $3$-torus:
\begin{proof}[Proof of \Cref{item:examples_contact_deformation__giroux_torsion} in \Cref{prop:examples_contact_deformation}]
    First, recall that $\xi_k$ on $M\times \T^2$ from \cite[Theorem A]{MNW} is defined by the one-form
    \[
    \alpha_k = \frac{1+\cos (ks)}{2} \alpha_+ + \frac{1-\cos (ks)}{2}\alpha_- + \sin(ks) \d t \, ,
    \]
    where $(s,t)$ are coordinates on $\T^2$ and $(\alpha_+,\alpha_-)$ is a Liouville pair on $M$.
    In fact, $M$ is the quotient of $\R^n\times \R^{n+1}$ under a certain co-compact lattice for the action
    \[
    (\tau_1,\ldots,\tau_n)\cdot(t_1,\ldots,t_n,\theta_0,\ldots,\theta_n) 
    := (t_1+\tau_1 ,\ldots, t_n + \tau_n , 
    e^{-(\tau_1+\ldots + \tau_n)}\theta_0, e^{\tau_1}\theta_1,\ldots, e^{\tau_n}\theta_n)\, ,
    \]
    and on this cover, one simply has
    \[
    \alpha_\pm = \pm e^{t_1+\ldots + t_n}\d\theta_0 +{e^{-t_1}\d\theta_1 + \ldots + e^{-t_n} \d\theta_n} \, .
    \]

    The first step is to change the contact structures $\xi_k$ by a contact isotopy to contact structures $\eta_k$ defined by 
    \[
    \beta_k = \d s + \alpha_k = \d s +  \frac{1+\cos (ks)}{2} \alpha_+ + \frac{1-\cos (ks)}{2}\alpha_- + \sin(ks) \d t \, .
    \]
    That $\eta_k$ is contact isotopic to $\xi_k$ simply follows from Gray stability, via the linear interpolation $\rho \d s +\alpha_k$ for $\rho\in[0,1]$, which is contact at all times $\rho$: indeed, using the expressions for $\alpha_\pm$ on $\R^n\times \R^{n+1}$ one sees that $\d\alpha_k$ has $\d t$ contribution only paired with $\d s$, so that $\d s \wedge \d\alpha_k^{n+1}=0$ and hence $(\rho\d s+\alpha_k)\wedge \d \alpha_k^{n+1}=\alpha_k\wedge \d\alpha_k^{n+1}>0$ for all $\rho\in[0,1]$.

    Now, for $r\in[0,1]$, consider the family 
    \[
    \beta_{k,r} = \d s + r\alpha_k = \d s +  r\left[\frac{1+\cos (ks)}{2} \alpha_+ + \frac{1-\cos (ks)}{2}\alpha_- + \sin(ks) \d t\right] \, ,
    \]
    which is contact for all $r>0$ as, by the same computation, $\beta_{k,r}\wedge \d \beta_{k,r}^{n+1}=r^{n+2}\alpha_k\wedge \d\alpha_k^{n+1}$.
    Note that at $r=0$ we simply have $\beta_{k,0} = \d s$ defining the foliation $\xi=\ker(\d s)$; on this, we consider the c-symplectic structure given by $\CS_{\xi,\omega}$ where $\omega$ is defined by the following expression on $\R^n\times \R^{n+1}$ (that is invariant under the action hence passes to the quotient $M$):
    \[
    \omega = e^{t_1+\ldots + t_n}\d t \wedge \d \theta_0 - \sum_{i=1}^n e^{-t_i} \d t_i \wedge \d\theta_i \, .
    \]
    We claim that $\xi_{k,r}=\ker(\beta_{k,r})$ gives, for all $k\geq 1$ and as $r\to 0$, the desired contact $C^\infty$-deformations of $(\xi,\CS_{\xi,\omega})$.

    For this, first note that on the cover $\R^n\times \R^{n+1}$ we have
    \[
        \frac{1+\cos (ks)}{2} \alpha_+ + \frac{1-\cos (ks)}{2}\alpha_- = 
        \cos(ks)e^{t_1+\ldots + t_n}\d\theta_0 + {e^{-t_1}\d\theta_1 + \ldots + e^{-t_n} \d\theta_n} \, ,
    \]
    \[
        \frac{1+\cos (ks)}{2} \d\alpha_+ + \frac{1-\cos (ks)}{2}\d\alpha_- =
        -\sum_{i=1}^n e^{-t_i}\d t_i \wedge\d\theta_i + \cos(ks) e^{\sum_i t_i} \sum_{i=1}^n \d t_i \wedge \d\theta_0 \, ,
    \]
    \[
    - \alpha_+ + \alpha_- = - 2 e^{t_1+\ldots + t_n }\d\theta_0 \, .
    \]
    Then, we can compute 
    \begin{align*}
    \d\beta_{k,r} = & \,  r \d\alpha_k \\
    = & \,  r\frac{k}{2}\sin(ks) \d s \wedge (-\alpha_++\alpha_-) + r\frac{1+\cos (ks)}{2} \d\alpha_+ \\
    & + r\frac{1-\cos (ks)}{2}\d\alpha_- {+ rk\cos(ks)\d s \wedge \d t} \\ 
    = & \,  -  rk\sin(ks)e^{t_1+\ldots + t_n }\d s \wedge\d\theta_0 -r\sum_{i=1}^n e^{-t_i}\d t_i \wedge\d\theta_i  \\
    & + r\cos(ks) e^{\sum_i t_i} \sum_{i=1}^n \d t_i \wedge \d\theta_0 {+ r k\cos(ks)\d s \wedge \d t} \, .
    \end{align*}

    Moreover, $\xi$ is a foliation so $\cK_\xi$ has rank $2n+2$ (and coincides with $\xi$), so according to \Cref{def:contact_approximation_deformation} we need to compute $\beta_{k,r}\wedge \d\beta_{k,r}$ and write it, up to multiplication by a positive function $f_r$, in the form $\d s \wedge \mu$ for some $\mu$.
    We can compute
    \begin{align*}
        \beta_{k,r}\wedge \d\beta_{k,r} = \; & 
        r \d s \wedge \d \alpha_k + O(r^2) \\
        = \; & r \d s \wedge \left[ 
             -  k\sin(ks)e^{t_1+\ldots + t_n }\d s \wedge\d\theta_0 -\sum_{i=1}^n e^{-t_i}\d t_i \wedge\d\theta_i \right. \\
    & \left. + \cos(ks) e^{\sum_i t_i} \sum_{i=1}^n \d t_i \wedge \d\theta_0  +  k\cos(ks)\d s \wedge \d t \, 
        \right] + O(r^2) \\
        = \; & r \d s \wedge \left[ 
             -\sum_{i=1}^n e^{-t_i}\d t_i \wedge\d\theta_i 
    + \cos(ks) e^{\sum_i t_i} \sum_{i=1}^n \d t_i \wedge \d\theta_0 \, 
        \right] + O(r^2) \, ,
    \end{align*}
    where the first equality comes from substituting $r\d\alpha_k$ from the previous computation of $\d\beta_{k,r}$, and the second equality comes from the fact that the terms containing $\d s$ in parenthesis contribute zero when wedged with the first $\d s$.
    Then, we can choose 
    \[
        \mu =  -\sum_{i=1}^n e^{-t_i}\d t_i \wedge\d\theta_i 
        + \cos(ks) e^{\sum_i t_i} \sum_{i=1}^n \d t_i \wedge \d\theta_0 \, ,
    \]
    and the conformal convergence is satisfied by choosing $f_r=1/r$.
    This $\mu$ satisfies Items \ref{item:contact_approximation_deformation__char_distr__non_zero_conf_class} and \ref{item:contact_approximation_deformation__char_distr__conformal_convergence} in \Cref{def:contact_approximation_deformation}. The remaining property (\Cref{item:contact_approximation_deformation__char_distr__sympl_compat}) that we need to check is that $[\omega]$ and $[\mu]$ span a cone of symplectic structures on $\ker(\d s)$.

    This follows from the fact that, for $a,b\in\R_{>0}$, we have 
    \[
    a \omega + b \mu = -(a+b) \sum_{i=1}^n e^{-t_i}\d t_i \wedge\d\theta_i + a e^{t_1+\ldots + t_n}\d t\wedge \d\theta_0  + b \cos(ks) e^{\sum_i t_i} \sum_{i=1}^n \d t_i \wedge \d\theta_0 \, .
    \]
    As only terms containing a $\d t$ contribute non-trivially (and the last sum cannot contribute as it contains $\d\theta_0$), we moreover get that
    \begin{align*}
    \d s \wedge (a \omega + b\mu)^{n+1} = & \,  (n+1)a(a+b)^{n} e^{t_1+\ldots + t_n} \d s \wedge\d t \wedge  \d\theta_0 \wedge \left[
    - \sum_{i=1}^n e^{-t_i}\d t_i \wedge\d\theta_i
        \right]^{n}\\
    = & \,  (-1)^n(n+1)!\,a(a+b)^{n} \d s \wedge \d t\wedge \d\theta_0\wedge \d t_1\wedge \d\theta_1 \wedge \ldots \wedge \d t_n \wedge\d\theta_n \, .
    \end{align*}
    Now, recall from \cite[Equation 8.1]{MNW} that $\R^n\times\R^{n+1}$ is intended with the orientation making $\alpha_+$ a positive contact form on it, and hence on $M$ after taking the quotient; an explicit computation shows that 
    \[
    \alpha_+\wedge\d\alpha_+^n =  (-1)^n (n+1)! \d\theta_0\wedge \d t_1\wedge \d\theta_1 \wedge \ldots \wedge \d t_n \wedge\d\theta_n \, ,
    \]
    and $M\times \T^2$ is oriented by $\alpha_+\wedge \d\alpha_+^n\wedge \d s \wedge \d t$, which is the orientation making $\xi_k$ a positive contact structure (as $\alpha_k\wedge \d\alpha_k^{n+1}=(n+1)k\alpha_+\wedge \d\alpha_+^n\wedge \d s \wedge \d t$ along $\{s=0\}$ for instance).
    As $\d s \wedge \d t \wedge \d \theta_0 \wedge \d t_1 \wedge \ldots \wedge \d\theta_n=\d\theta_0\wedge \d t_1\wedge \ldots \wedge \d \theta_n\wedge \d s \wedge \d t$, we do indeed have that, for all $a,b>0$, $\d s \wedge (a \omega + b\mu)^{n+1} >0$ on  $M\times \T^2$.
    But this means exactly that $\omega$ and $\mu$ span a cone of non-degenerate two-forms on $\xi=\ker(\d s)$, as desired.
\end{proof}

\medskip

Finally, we show that the confoliations constructed in \Cref{prop:open_book_construction} can be deformed to a contact structure adapted to the open book when the pages satisfy the natural condition for this to happen.

\begin{proof}[Proof of \Cref{item:open_book_construction__deformation} in \Cref{prop:open_book_construction}]
\Cref{item:contact_approximation_deformation__hyperplanes} from \Cref{def:contact_approximation_deformation} will be immediate, so we focus on \Cref{item:contact_approximation_deformation__char_distr}.

Recall that the monodromy $\varphi$ of the open book is a symplectomorphism of the Liouville page $(W,\d\lambda)$ that is exact, i.e.\ so that $\varphi^*\lambda = \lambda + \d V$ for some function $V$ that is constant near $\partial W$. 
We can moreover assume $V$ to be positive everywhere, up to translation by a constant.
This yields a contact form $\alpha=\lambda+ \d\theta$, where we abuse notation and denote by $\lambda$ the one-form induced in the mapping torus. 
Near the boundary $(-\frac{\delta}{2},0] \times \partial W \times S^1$ it is of the form $\alpha=e^r\lambda_{\partial W} +\d\theta$, where $\lambda_{\partial W}=i^*\lambda$. 
Here $i$ is the natural inclusion of the boundary $\partial W$ inside $W$. 
We glue $U=\partial W \times D^2_\delta$ to the boundary of the mapping torus by a map 
\begin{align*}
    F: \partial W\times (\delta/2,\delta)_r\times S^1_\theta &\longrightarrow (-\frac{\delta}{2},0] \times \partial W \times S^1 \\
    (x, r, \theta) &\longmapsto (\delta/2-r, x , \theta).
\end{align*}
Then the contact form looks like $e^{\delta/2-r}\lambda_{\partial W} + \d\theta$ near $\partial U$.

Finally, we extend the contact form over $r\in[0,\delta)$ as
$$\alpha= f_1(r) \lambda_{\partial W} + g_1(r)\d\theta,$$
with $f_1$ equal to $c=e^{\delta/2}$ near $0$, $f_1'\leq 0$, and $g_1$ non-decreasing, quadratic near $0$ (with $g_1''(0)>0$), and such that $f_1(r)=e^{\delta/2-r}$ and $g_1(r)=1$ for all $r\geq \delta/2-\varepsilon_1$, for some small $\varepsilon_1>0$. In particular, $\alpha$ coincides on the whole gluing region $\{\delta/2<r<\delta\}$ with the contact form $e^{\delta/2-r}\lambda_{\partial W}+\d\theta$ of the mapping torus.
Let $f_0$ be a non-increasing function such that $f_0(r)=1$ near $r=0$, $f_0(r)=0$ for all $r\geq b$ (in particular on the whole gluing region), and $f_0'<0$ on $\{0<f_0<1\}=(a,b)$, with $b<\delta/2$, and let $g_0=1-f_0(r)$, so that $f_0^2+g_0^2>0$ everywhere; we also assume that $g_1'>0$ on $\{f_0=1\}\setminus \{r=0\}$. 
We endow the confoliation $\xi=\ker \alpha_0$, where $\alpha_0=f_0\lambda_{\partial W} + g_0 \d\theta$ extended as $\d\theta$ away from $U$, with the c-symplectic structure $\CS_{\xi,\Omega}$ where
  $$\Omega=c\,\d\lambda_{\partial W} + h(r)\d r\wedge \d\theta - \d(l(r)\lambda_{\partial W}),$$
as in \Cref{sec:construction}, but $l(r)$ will have a different behavior near $r=\delta$ since we are considering a neighborhood of the boundary of the pages where the symplectic form looks like $\d(e^{\delta/2-r}\lambda_{\partial W})$ rather than {$\d((\delta-r)\lambda_{\partial W})$}.
We choose $h$ and $l$ in the lines of \Cref{ex:functions}, so that in particular $h(r)=\frac{1}{2}$ and $l(r)=\frac{1}{2}r$ whenever $f'_0(r)\neq 0$,
i.e.\ in some interval $[a,b]\subset (0,\delta/2)$.
On the gluing region $r\in(\delta/2,\delta)$, $\Omega$ has to coincide with the symplectic form $\d(e^{\delta/2-r}\lambda_{\partial W})$ of the pages; as $\Omega=(c-l)\,\d\lambda_{\partial W}-l'\,\d r\wedge\lambda_{\partial W}+h\,\d r\wedge\d\theta$, this amounts to $h=0$ and $l(r)=c-e^{\delta/2-r}$ there.
We hence choose $l$ to be strictly increasing away from a neighborhood of $r=0$, with $l(r)=c-e^{\delta/2-r}$ for $r\geq\delta/2$ (note that $l(\delta/2)=e^{\delta/2}-1>b/2$, so this is compatible with $l=\frac12 r$ on $[a,b]$); in particular $l<c$ everywhere, so that $J_B$ tames $(c-l)\,\d\lambda_{\partial W}$ and the arguments of \Cref{sec:construction} apply verbatim.
Hence, $\Omega$ extends as $\d\lambda$ away from $U$, i.e. as a maximal rank two-form that induces a symplectic form on the pages of the open book.
 
\medskip

Define the family of one-forms
$$\alpha_s= f_s(r) \lambda_{\partial W} + g_s(r) \d\theta,$$ 
where $f_s(r)=(1-s) f_0(r) + s f_1(r)$ and likewise $g_s(r)=(1-s) g_0(r) + s g_1(r)$. 
Choosing conveniently $f_0,f_1,g_1$ among functions verifying the properties asked previously, we can moreover arrange that $f_s g_s' - g_s f_s'>0$ for all $s>0$ and $r>0$, and also $(f_sg_s'-g_sf_s')/r>0$ near $r=0$.

Each such $\alpha_s$ extends as $s\lambda + \d\theta$ away from $U$, where $\lambda$ is here to be intended as the fiberwise Liouville form on the mapping torus as constructed by Giroux \cite{Gir02,Gir20}. 

Notice then that $\alpha_s$ is a contact form for small $s>0$.
For this, we compute 
$$\d\alpha_s= f_s \d\lambda_{\partial W} + f_s' \d r \wedge \lambda_{\partial W} + g_s' \d r \wedge \d\theta.$$
Then, the contact condition is obvious away from $U$, and near the binding one has 
\begin{equation}\label{eq:obcontact1}
\alpha_s\wedge (\d\alpha_s)^{n}=nf_s^{n-1} \left( f_s g_s' - g_s f_s' \right) \lambda_{\partial W}\wedge (\d\lambda_{\partial W})^{n-1} \wedge \d r\wedge \d\theta,
\end{equation}
which is positive by the choice of the functions.
However, it vanishes for $s=0$ at all points where $f_0=0$ or $g_0=0$, which are two disjoint sets (by choice of $f_0,g_0$) over which $\alpha_0$ is not a contact structure, i.e.\ has a non-trivial characteristic distribution. 
In $\{g_0=0\}$ the characteristic distribution of $\xi$ has rank two, while in $\{f_0=0\}$ we have $\cK_\xi=\xi$. 

In $\{g_0=0\}$ we also have $f_0=1$, so that $\alpha_0=\lambda_{\partial W}$ and $\alpha_0\wedge \d\alpha_0^{n-1}=\lambda_{\partial W}\wedge (\d\lambda_{\partial W})^{n-1}$ by \eqref{eq:obcontact1}. 
We hence choose $\mu = n g_1' \d r \wedge \d\theta$, then 
\begin{equation}\label{eq:obcontact2}
\alpha_0\wedge \d\alpha_0^{n-1}\wedge \mu = n g_1' \lambda_{\partial W}\wedge \d \lambda_{\partial W}^{n-1} \wedge \d r\wedge \d\theta.
\end{equation}
The form $\mu$ clearly satisfies \Cref{item:contact_approximation_deformation__char_distr__non_zero_conf_class} of \Cref{def:contact_approximation_deformation}, as $\alpha_0\wedge (\d\alpha_0)^{n-1}\wedge \mu \neq 0$ (recall that $g_1'>0$ for $r>0$ on this set, and that $g_1'\d r \wedge \d\theta$ is a non-zero multiple of the area form at $r=0$, as $g_1$ is quadratic with $g_1''(0)>0$).
We take as an extension of \eqref{eq:obcontact2} to a neighborhood of $\{g_0=0\}$ the form
$$\eta=n g_1'\,\lambda_{\partial W}\wedge \d \lambda_{\partial W}^{n-1} \wedge \d r\wedge \d\theta.$$
Consider the extensions $\hat{f}_s= (1-s)+sf_1$ and $\hat g_s=sg_1$ of $f_s$ and $g_s$ to a neighborhood of $\{g_0=0\}$. Then we take an extension of \eqref{eq:obcontact1} of the form 
$$ \zeta_s= n \hat f_s^{n-1} \left( \hat f_s \hat g_s' - \hat g_s \hat f_s' \right) \lambda_{\partial W}\wedge (\d\lambda_{\partial W})^{n-1} \wedge \d r\wedge \d\theta.$$
Notice that $ \hat f_s \hat g_s' - \hat g_s \hat f_s'= sg_1' + O(s^2)$ and $\hat f_s\to 1$.
Choosing as conformal factor $F_s\coloneqq \frac{1}{s}$, we see that $F_s \zeta_s$ $C^\infty$-converges to $\eta$, and thus that \Cref{item:contact_approximation_deformation__char_distr__conformal_convergence} is satisfied. 
Finally, notice that $\mu$ and $\Omega$ are symplectically compatible along the characteristic distribution of $\ker \alpha_0$, since $\cK_\xi|_{\{g_0=0\}}=\langle \partial_r, \partial_\theta \rangle$ and $\Omega|_{\cK_\xi}= h(r) \d r\wedge \d \theta$ are both positive area forms (as $h=r$ near $0$ and $h>0$ on $(0,b]$), and thus \Cref{item:contact_approximation_deformation__char_distr__sympl_compat} is satisfied too.
\\

Over the set $\{f_0=0\}$, we have $\alpha_0 = \d\theta$, so that $\cK_\xi = \xi= \langle \partial_r \rangle \oplus T (\partial W)$.
Moreover, 
\[
\alpha_s \wedge \d\alpha_s = f_s^2\lambda_{\partial W} \wedge \d \lambda_{\partial W}
+f_s g_s' \lambda_{\partial W}\wedge \d r \wedge \d\theta 
+ g_s f_s \d\theta \wedge \d \lambda_{\partial W} 
+ g_sf_s'\d\theta \wedge \d r \wedge \lambda_{\partial W} \, .
\]
{Now, where $f_0=0$ one has $g_0=1$ and hence $f_s=sf_1$ and $g_s=1-s+sg_1$. 
Thus, $f_s^2=s^2f_1^2+O(s^2)$, $f_sg_s'=s^2f_1g_1'+O(s^2)$, $g_sf_s=sf_1+O(s^2)$ and $g_sf_s'=sf_1'+O(s^2)$, so that
\begin{align*}
    \frac{1}{s}\alpha_s \wedge \d\alpha_s = & \, f_1 \d\theta \wedge \d \lambda_{\partial W} 
+ f_1'\d\theta \wedge \d r \wedge \lambda_{\partial W} + O(s) \\
= & \, \alpha_0\wedge (-f_1'\lambda_{\partial W}\wedge \d r + f_1\d \lambda_{\partial W}) + O(s)\, .
\end{align*}
Hence, we can use} 
the two-form 
$$\mu:= -f_1'\lambda_{\partial W}\wedge \d r + f_1 \d \lambda_{\partial W},$$
defined in a neighborhood of $\{f_0=0\}$. As $f_1>0$ and $-f_1'\geq 0$, the form $\mu$ satisfies that $\alpha_0\wedge \mu \neq 0$, and so \Cref{item:contact_approximation_deformation__char_distr__non_zero_conf_class} holds.
We choose $\d\theta\wedge \mu$ as an extension of $\alpha_0\wedge \mu$ to a neighborhood of $\{f_0=0\}$. 
We cannot choose $\alpha_s\wedge \d\alpha_s$ itself as an extension of this form to a neighborhood of $\{f_0=0\}$:
indeed, doing so $\frac{1}{s} \alpha_s\wedge \d\alpha_s$ would not converge away from $\{f_0=0\}$ (as it has a non-zero constant term on the complement of $\{f_0=0\}$ itself). 
Instead, we choose $\tilde f_s= sf_1$ as an extension of $f_s|_{\{f_0=0\}}$, and $\tilde g_s= (1-s)+sg_1$ as an extension of $g_s$ to a neighborhood of $\{f_0=0\}$.
With these functions, we consider the form
$$ \zeta_s= \tilde f_s^2\lambda_{\partial W} \wedge \d \lambda_{\partial W}
+ \tilde f_s \tilde g_s' \lambda_{\partial W}\wedge \d r \wedge \d\theta 
+ \tilde g_s \tilde f_s \d\theta \wedge \d \lambda_{\partial W} 
+ \tilde g_s \tilde f_s'\d\theta \wedge \d r \wedge \lambda_{\partial W} \, $$
as an extension of $\alpha_s\wedge \d\alpha_s$ to a neighborhood of $\{f_0=0\}$. 
Choosing as conformal factors $F_s:= \frac{1}{s}$, we have that $F_s\zeta_s$ converges to $\d\theta \wedge \mu$ in a neighborhood of $\{f_0=0\}$, i.e. \Cref{item:contact_approximation_deformation__char_distr__conformal_convergence} is satisfied. 
In addition, $\mu$ is symplectically compatible with $\Omega$ on the characteristic distribution (that is just $\ker\alpha_0$) and thus \Cref{item:contact_approximation_deformation__char_distr__sympl_compat} is satisfied too. Indeed, we had $l'>0$ whenever $f_0=0$ and $l(r)$ is increasing and smaller than $c$ everywhere, so the relevant terms of $\Omega$ are $(c-l)\d\lambda_{\partial W} + l'\lambda_{\partial W}\wedge \d r$,
and for all $t>0$ we have 
\begin{multline*}
\d\theta\wedge \left(
    (c-l+tf_1)\d\lambda_{\partial W}+(l'-tf_1')\lambda_{\partial W}\wedge \d r 
\right)^n = \\
n(c-l+tf_1)^{n-1}(l'-tf_1')\lambda_{\partial W}\wedge \d\lambda_{\partial W}^{n-1}\wedge \d r \wedge \d\theta >0 \, .
\end{multline*}
Away from $U$, one similarly has $\frac{1}{s}\alpha_s\wedge \d\alpha_s \to \d\theta\wedge \d\lambda$, which is compatible with $\Omega=\d\lambda$.
\end{proof}

%%%%%%%%%%%%%%%%%%%%%%%%%%%%%%%%%%%%%%%%%%%%%%%%%%%%%%%%%%%%%%%%%
%%%%%%%%%%%%%%%%%%%%%%%%%%%%%%%%%%%%%%%%%%%%%%%%%%%%%%%%%%%%%%%%%
%%%%%%%%%%%%%%%%%%%%%%%%%%%%%%%%%%%%%%%%%%%%%%%%%%%%%%%%%%%%%%%%%

\appendix

\section{The space of cotaming complex structures}
In this appendix, we establish the following proposition:
\begin{prop}
    \label{prop:space_cotaming_Js_contractible}
    Let $E\to M$ be a real vector bundle over a (possibly open) smooth manifold $M$.
    Consider also $\omega$ a linear alternating two-form on $E\to M$, and $\cK:=\ker\omega=\{v\in E \mid \iota_v \omega =0\}$.
    Assume that $\cK\subset E$ is a \emph{honest} vector sub-bundle of $E$ (i.e.\ that \emph{$\ker\omega$ has constant rank}).
    Let finally $\mu$ be a fiberwise non-degenerate linear alternating two-form on $\cK$.
    Then, the space $\cJ_\tau(E,\omega,\mu)$ of linear complex structures $J$ on $E$ that are tamed by $\omega$ (in particular, $J\cK\subset \cK$) and so that $J\vert_{\cK}$ is tamed by $\mu$ is either empty or contractible.
\end{prop}

For its proof, we will need the following technical statement, whose proof is easy and hence omitted.

\begin{lemma}
    \label{lem:tamed_iff_tamed_in_quotient}
    Let $E,M,\omega,\cK,\mu$ be as in \Cref{prop:space_cotaming_Js_contractible}.
    Consider also a linear complex structure $J$ on $E\to M$ such that $J\cK=\cK$. 
    Lastly, denote by $\baromega$ and $\barJ$ respectively the non-degenerate alternating two-form and complex structure on $E/\cK$ naturally induced by $\omega$ and $J$.
    Then, the following are equivalent:
    \begin{itemize}
        \item[-] $J$ is tamed by $\omega$;
        \item[-] $\barJ$ is tamed by $\baromega$. 
    \end{itemize}
\end{lemma}

\begin{proof}[Proof of \Cref{prop:space_cotaming_Js_contractible}]
    Let $E,M,\omega,\cK,\mu$ be as in \Cref{prop:space_cotaming_Js_contractible}.
    Fix also an auxiliary vector sub-bundle $F\subset E$ which is complementary to $\cK$ in $E$.

    Let $\pi\colon E \to F$ be the fiberwise projection onto $F$ parallel to $\cK$.
    In other words, $\pi(v)=v_F$ for $v=v_\cK\oplus v_F \in E = \cK\oplus F$.
    Consider then the smooth map $r\colon[0,1]\times E\to E$ given by $r(t,v)=(1-t)v+t\pi(v)=v_F+(1-t)v_\cK$ for $v=v_\cK\oplus v_F\in E=\cK\oplus F$.
    This is a fiberwise retraction of $E$ onto $F$.

    For any linear complex structure $J$ on $E\to M$, one can consider the family $J_t$ of endomorphisms of $E$ given by, for $v=v_\cK\oplus v_F\in E=\cK\oplus F$,
    \[
    J_t (v) := J v_\cK + r_t J v_F = [J v_\cK + (1-t)(J v_F-\pi Jv_F)]\oplus\pi Jv_F \in E = \cK\oplus F .
    \]
    Now, a direct computation gives that 
    \[
    J_t^2(v_\cK \oplus v_F) = [-v_\cK - (1-t)(v_F+\pi J \pi Jv_F)]\oplus \pi J \pi J v_F \, ,
    \]
    and the fact that $J_0=J$ is known to be a linear complex structure on $E$ directly implies that
    \[
    v_F+\pi J \pi Jv_F = 0 \, , 
    \quad
    \pi J \pi J v_F = - v_F \, .
    \]
    In other words, $J_t$ is in fact a linear complex structure on $E\to M$ for all $t\in [0,1]$.

    It is immediate to check that $J_t \cK= \cK$ for all $t\in [0,1]$ as well.
    Moreover, as the composition of the projection onto the second factor $E=\cK\oplus F\to F$ with $F\to E/\cK$ is an isomorphism and $\overline{J_t}=\barJ$ on $E/\cK$ for all $t\in[0,1]$, according to \Cref{lem:tamed_iff_tamed_in_quotient} all of the $J_t$'s are tamed by $\omega$.

    In other words, this gives that $\cJ_\tau(E,\omega,\mu)$ retracts onto 
    \begin{equation}
    \label{eqn:space_split_Js}
    \{ J_\cK \oplus J_F \mid 
    J_\cK \text{ tamed by } \mu \, , \;
    J_F \text{ tamed by } \omega\vert_F \, \} 
    \subset 
    \cJ_\tau(E,\omega,\mu) \, .
    \end{equation}

    Now, the space in \eqref{eqn:space_split_Js} is just the direct sum of the space of linear complex structures on $\cK$ tamed by $\mu$ and of the one of linear complex structures on $F$ tamed by $\omega_F$.
    As $\mu$ and $\omega_F$ are non-degenerate on $\cK$ and $F$ respectively, each of the spaces of which we take the product is hence contractible by an argument completely analogous to the classical one e.g.\ in \cite[Section 2.2]{WenBook}.
    This concludes the proof.
\end{proof}

%%%%%%%%%%%%%%%%%%%%%%%%%%%%%%%%%%%%%%%%%%%%%%%%%%%%%%%%%%%%%%%%%%%%%%

\bibliographystyle{alpha}
\bibliography{biblio}

%%%%%%%%%%%%%%%%%%%%%%%%%%%%%%%%%%%%%%%%%%%%%%%%%%%%%%%%%%%%%%%%%%%%%%

\Addresses

\end{document}